\documentclass[a4paper, notitlepage, 12pt]{article}
\usepackage{geometry}
\fontfamily{times}
\usepackage{setspace,relsize}               
\usepackage{moreverb}                        
\usepackage{xurl}
\usepackage{hyperref}
\hypersetup{colorlinks=true,citecolor=blue}
\usepackage{amsmath}
\usepackage{mathtools} 
\usepackage{amsthm}
\usepackage{amssymb}
\usepackage{indentfirst}
\usepackage{todonotes}
\usepackage[authoryear,round]{natbib}
\usepackage[pdftex]{lscape}
\usepackage[utf8]{inputenc}

\usepackage{libertine}
\usepackage[libertine]{newtxmath}
\usepackage{ulem}
\usepackage{dsfont}

\usepackage{physics}
\usepackage{tikz}
\usepackage{mathdots}
\usepackage{cancel}
\usepackage{color}
\usepackage{siunitx}
\usepackage{array}
\usepackage{multirow}
\usepackage{gensymb}
\usepackage{tabularx}
\usepackage{extarrows}
\usepackage{booktabs}
\usetikzlibrary{patterns}
\usetikzlibrary{shapes}

\usepackage{svg}

\title{\vspace{-9ex}\centering \bf On the lumpability of tree-valued Markov chains}
\author{
Rodrigo B. Alves, Yuri F. Saporito \& Luiz M. Carvalho \\
School of Applied Mathematics, Getulio Vargas Foundation.
}
\newtheorem{theorem}{Theorem}[]
\newtheorem{lemma}{Lemma}[]
\newtheorem{proposition}{Proposition}[]
\newtheorem{definition}{Definition}[]
\newtheorem{corollary}{Corollary}[]
\newtheorem{remark}{Remark}[]
\DeclareMathOperator*{\argmin}{arg\,min}

\DeclareMathOperator*{\pr}{\operatorname{Pr}}

\newcommand{\TT}{\boldsymbol{T}}

\begin{document}
\maketitle

\begin{abstract}
Phylogenetic trees constitute an interesting class of objects for stochastic processes due to the non-standard nature of the space they inhabit.
In particular, many statistical applications require the construction of Markov processes on the space of trees, whose cardinality grows superexponentially with the number of leaves considered.
We investigate whether certain lower-dimensional projections of tree space preserve the Markov property in tree-valued Markov processes.
We study exact lumpability of tree shapes and $\varepsilon$-lumpability of clades, exploiting the combinatorial structure of the SPR graph to obtain bounds on the lumping error under the random walk and Metropolis-Hastings processes. 
Finally, we show how to use these results in empirical investigation, leveraging exact and $\varepsilon$-lumpability to improve Monte Carlo estimation of tree-related quantities.

\end{abstract}

\tableofcontents
\newpage 
\section{Introduction}
\label{sec:intro}

Phylogenetic trees (phylogenies) are planar graphs that describe evolutionary relationships between biological entities, usually inferred on data from DNA or RNA molecular sequences.
External nodes represent sampled species, whilst internal nodes encode unobserved ancestors~\citep{Steel2014}.
Moreover, branch lengths encode information about the amount of divergence between nodes, sometimes measured in calendar units of time.
Phylogenies find significant applications in Ecology and Medicine and especially in Molecular Epidemiology~\citep{Dudas2017,Candido2020}.
Phylogenetic trees also have an inherent probabilistic structure, which makes them a compelling subject of study in their own right -- see for example~\cite{Aldous1996, Mossel2003, Blum2006, Dinh2017, diSanto2022}.

Bayesian estimation of phylogenetic trees from genomic data has predominantly relied on Markov chain Monte Carlo (MCMC) techniques to sample from the posterior distribution, which is intractable even in the simplest toy examples~\citep{Drummond2012,Ronquist2012}.
The vast, high-dimensional nature of tree space, combined with the lack of a canonical representation ~\citep{Billera2001, Gavryushkin2016, Whidden2017, Lueg2021} make this a particularly challenging problem.
Moreover, a crucial step to employing MCMC in practice is to be able to diagnose convergence and mixing problems~\citep{Cowles1996,Warren2017,Vehtari2021, Brusselmans2024}.
A potential strategy for implementing and examining MCMC in phylogenetic space is to identify lower-dimensional projections that preserve the Markov property -- either exactly or approximately.
Specific structures within phylogenetic trees, such as their shapes and clades (which can be thought of as subtrees encompassing particular taxa), are of particular interest in this context.
Interestingly, these projections are also directly interpretable: tree shapes give insight into population processes affecting a population's evolution, whilst clades can be used directly in testing evolutionary hypotheses by checking which clades are more likely conditional on the available data.

In practice, researchers commonly track both convergence and mixing of tree-valued MCMC by analysing clade (subtree) indicator variables.
Loosely speaking -- a formal definition is given later in the paper --, one constructs a binary\textit{clade indicator} $Y(c)$, which is $1$ when clade $c$ is in a tree $X$ and zero otherwise.
For assessing convergence, for instance, it is common to look at the frequencies of a given clade in multiple chains and assess whether they are sufficiently close (e.g. \textit{via} the so-called average standard deviation of split frequencies, ASDSF, \cite{Ronquist2012, Warren2017}).
Mixing can be estimated by looking at the estimated effective sample size (ESS, \cite{Vehtari2021}) for each clade indicator marginally. 
The interested reader is referred to  \cite{Magee2023} and Section 2.3 of \cite{Kelly2022} for more details on how lower-dimension projections may be used for diagnosing phylogenetic MCMC.

An important question is thus: if one has a Markov process $(X_k)_{k \geq 0}$ on the space of trees, is the induced process  $(Y_{0}(c))_{0 \geq 0} \in \{0, 1\}$ for each clade $c$ also Markov?
In other words, one might ask whether the tree-valued Markov process is \textit{lumpable} with respect to the clade partition.
If lumpability is not exact, it is then of interest to understand the \textit{lumping error},  which in a sense quantifies that departure from a Markov process.
Figure~\ref{fig:cladeautocorr} shows the projection onto two clades (\{t1, t2\} and \{t1, t2, t3\}) for a lazy Metropolis-Hastings (lMH) on the space of rooted trees with $n = 4$, $5$, $6$ and $7$ leaves -- see below for rigorous definitions.
We show the autocorrelation spectra of the clade indicators along with the curve for a fitted two-state Markov chain (red). 
The figure suggests that while the process in clade space looks mostly Markov, there can be significant departures -- see e.g. the bottom right panel.

\begin{figure}[!h]
    \centering
    \includegraphics[scale=0.5]{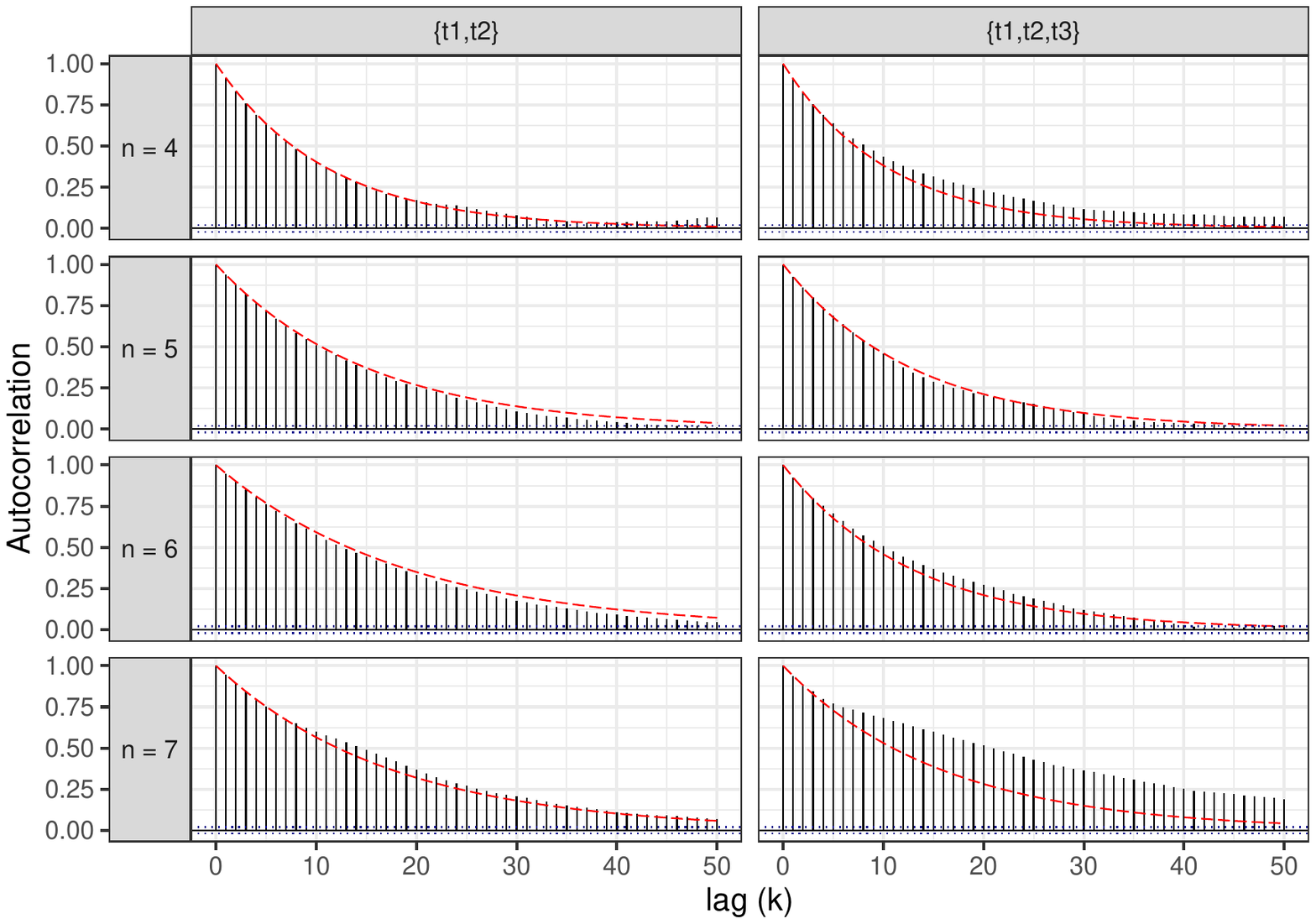}
    \caption{\textbf{Autocorrelation spectra of clade indicators for the lazy Metropolis-Hastings}.
    We show the empirical autocorrelation spectra up to lag $k=50$ (black bars) for indicators of clades \{t1, t2\} and \{t1, t2, t3\} when sampling from a lazy Metropolis-Hastings with $\rho = 0.9$ on a single realisation.
    The autocorrelation function of the best-fitting two-state Markov chain is also shown (red line).
    }
    \label{fig:cladeautocorr}
\end{figure}

These observations motivate a detailed study of the conditions for exact and approximate lumpability of tree-valued Markov processes.
In this paper, we study lumpability of two Markov processes, the lazy random walk and a Metropolis-Hastings process, with respect to tree shapes and clades.
We show that both processes are exactly lumpable with respect to tree shapes, but not for clades, in which case we give bounds on the lumping error.
With the first set of results in hand, we recover the stationary distribution of the original process by running a Markov chain on a smaller state space than the original.
With the second result, we use the lumping error to estimate the clade probability.
We finish with an empirical investigation using our theoretical results to better understand convergence and mixing in tree space.
The paper is organised as follows: in Section~\ref{sec:prelim} we give a detailed definition of the tree space studied in this paper along with a description of the Markov chains we consider.
A brief review of lumpability of Markov chains is also given.
We then move on study lumpability (Section~\ref{sec:tree_lumpability}) and $\varepsilon$-lumpability (Section~\ref{sec:eps_tree_lumpability}) of the two processes considered.
The empirical investigation drawing on previous results is detailed in Section~\ref{sec: Emp_Inv}.



\subsection{Preliminaries}
\label{sec:prelim}

This section gives the necessary mathematical background on the space of rooted trees, tree shapes and clades.
The interested reader is referred to~\cite{Semple2003},~\cite{Steel2014} and~\cite{Whidden2017} for more in-depth information.
We also review exact and approximate lumpability for Markov chains.

First, we introduce some notation to be used throughout this work.
Let $(X_k)_{k \ge 0}$ denote an irreducible discrete time Markov Chain on a countable state space $\mathcal{S}$, characterised by transition probability $P$ and assuming that the stationary distribution $\pi$ exists. For any $x, y \in \mathcal{S}$, we write
\[
p(x,y) : = \operatorname{Pr}(X_{k+1} = y \mid X_{k} = x)\,.
\] 
For a set $A \in \mathcal{S}$ we write
\[
p(x, A) := \sum_{z \in A} p(x, z) \,.
\]

Letting $\mu$ and $\nu$ be two probability measures, we denote by $||\mu -\nu||$  the total variation distance between $\mu$ and $\nu$.

\paragraph{Rooted phylogenetic trees:}
A labelled, rooted binary phylogenetic tree is a planar graph such that each vertex (or node) either has two children or none.
The terminal vertices, furthest from the root, are termed 'leaves' and are 1-degree vertices. A rooted binary phylogenetic tree with $n$ leaves will have $n$ labels. Each leaf uniquely corresponds to a label from the set $L:= \{1, 2, \dots, n\}$,  meaning the leaves of an $n$-labelled tree bijectively map to the set $L$.
All other vertices have degree 3, with a distinguished 2-degree vertex identified as the \textit{root}.
Any vertex that is not a leaf is termed an internal vertex (node) and we say that an internal vertex has a bigger depth if it is closer to the root.
Also, we refer to these vertices as parents and those directly below as descendants. 

Throughout this text, the term `tree', unless otherwise stated, refers to a labelled, rooted binary phylogenetic tree.
We denote by $\TT_n$ as the set of all trees with $n$ leaves.
This space is quite large and its cardinality~\citep{Schroder1870} 
\begin{equation*}
 |\TT_n| = (2n -3)!! = \frac{(2n-2)!}{2^{n-1}(n-1)!} \,,   
\end{equation*}
grows superexponentially on $n$ -- see Table~\ref{tab:cardinalities}. 
In a tree, we will assume time progresses from the root toward the leaves.
A \textit{path} between two vertices in a tree consists of an alternating sequence of vertices and edges.
Two vertices are \textit{adjacent} if they are connected by an edge.
For an internal vertex $w$ in a tree, we define $\gamma(w)$ as the number of degree-3 vertices along its path to the root.
We say a vertex $w$ in a tree $x \in \TT_n$ is a \textit{descendant} of a vertex $u \in x$ if there exists a path from $u$ to $w$ that moves strictly forward in time; in this case, $u$ is termed an \textit{ancestor} of $w$.   

A \textit{subtree} $z$ of a tree $x \in \TT_n$ is defined with the property that for every vertex $w \in x$ contained in $z$, all its descendants are also included in $z$. Then $z$ belongs to $\TT_m$, where $m \le n$. 
A clade, denoted as $c$, represents a specific type of subtree within a tree and can be represented as a subset of $L$.
This subset indicates the leaves of the subtree (or clade) $c$.
As an illustration, consider a tree $x \in \TT_6$ containing the clade 
$c:= \{1,2,5\}$.
This implies the existence of a subtree in $x$ with these specific leaves -- see Figure~\ref{fig:x_6_clade}.

\begin{figure}[h]
    \centering

\tikzset{every picture/.style={line width=0.75pt}} 

\begin{tikzpicture}[x=0.75pt,y=0.75pt,yscale=-1,xscale=1]

\draw    (150,51.67) -- (241.67,139) ;
\draw    (150,51.67) -- (58.33,139.67) ;
\draw   (55.83,143.54) .. controls (55.83,141.4) and (56.95,139.67) .. (58.33,139.67) .. controls (59.71,139.67) and (60.83,141.4) .. (60.83,143.54) .. controls (60.83,145.68) and (59.71,147.42) .. (58.33,147.42) .. controls (56.95,147.42) and (55.83,145.68) .. (55.83,143.54) -- cycle ;
\draw    (90,110.67) -- (109,139.67) ;
\draw    (210.67,110.67) -- (194.33,141) ;
\draw    (110,91.67) -- (140.33,139.67) ;
\draw    (190.67,90.67) -- (163.67,139.67) ;
\draw   (106.5,143.54) .. controls (106.5,141.4) and (107.62,139.67) .. (109,139.67) .. controls (110.38,139.67) and (111.5,141.4) .. (111.5,143.54) .. controls (111.5,145.68) and (110.38,147.42) .. (109,147.42) .. controls (107.62,147.42) and (106.5,145.68) .. (106.5,143.54) -- cycle ;
\draw   (137.83,143.54) .. controls (137.83,141.4) and (138.95,139.67) .. (140.33,139.67) .. controls (141.71,139.67) and (142.83,141.4) .. (142.83,143.54) .. controls (142.83,145.68) and (141.71,147.42) .. (140.33,147.42) .. controls (138.95,147.42) and (137.83,145.68) .. (137.83,143.54) -- cycle ;
\draw   (161.17,143.54) .. controls (161.17,141.4) and (162.29,139.67) .. (163.67,139.67) .. controls (165.05,139.67) and (166.17,141.4) .. (166.17,143.54) .. controls (166.17,145.68) and (165.05,147.42) .. (163.67,147.42) .. controls (162.29,147.42) and (161.17,145.68) .. (161.17,143.54) -- cycle ;
\draw   (191.83,144.88) .. controls (191.83,142.73) and (192.95,141) .. (194.33,141) .. controls (195.71,141) and (196.83,142.73) .. (196.83,144.88) .. controls (196.83,147.02) and (195.71,148.75) .. (194.33,148.75) .. controls (192.95,148.75) and (191.83,147.02) .. (191.83,144.88) -- cycle ;
\draw   (239.17,142.88) .. controls (239.17,140.73) and (240.29,139) .. (241.67,139) .. controls (243.05,139) and (244.17,140.73) .. (244.17,142.88) .. controls (244.17,145.02) and (243.05,146.75) .. (241.67,146.75) .. controls (240.29,146.75) and (239.17,145.02) .. (239.17,142.88) -- cycle ;
\draw  [dash pattern={on 0.84pt off 2.51pt}] (51,87.67) -- (148.33,87.67) -- (148.33,161) -- (51,161) -- cycle ;

\draw (53.07,152.6) node [anchor=north west][inner sep=0.75pt]  [font=\tiny]  {$1$};
\draw (106,151.53) node [anchor=north west][inner sep=0.75pt]  [font=\tiny]  {$2$};
\draw (238.33,150.15) node [anchor=north west][inner sep=0.75pt]  [font=\tiny]  {$3$};
\draw (136.93,150.07) node [anchor=north west][inner sep=0.75pt]  [font=\tiny]  {$5$};
\draw (159.83,150.94) node [anchor=north west][inner sep=0.75pt]  [font=\tiny]  {$4$};
\draw (135.2,28.8) node [anchor=north west][inner sep=0.75pt]  [font=\scriptsize]  {$x$};
\draw (189.67,152.15) node [anchor=north west][inner sep=0.75pt]  [font=\tiny]  {$6$};
\draw (99.87,71.47) node [anchor=north west][inner sep=0.75pt]  [font=\scriptsize]  {$c$};

\end{tikzpicture}    
\caption{\textbf{A tree $x \in \TT_6$ and one of its clades.}
The clade $c$ as a subtree with leaves $\{1,2,5\}$ is shown in the dashed rectangle.}
    \label{fig:x_6_clade}
\end{figure}
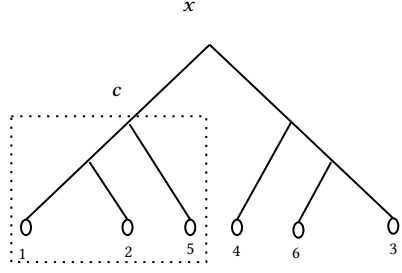


We define $\boldsymbol{C}_n$ as the space comprising all possible clades with the set of labels $L$.
Let $\mathcal{C}: \TT_n \to \boldsymbol{C}_n$ denote a mapping such that, for any tree in $\TT_n$, it yields all the clades contained within that tree.
For example, let $x \in \TT_6$ be the tree define in the Figure~\ref{fig:x_6_clade}, then we obtain $\mathcal{C}(x) = \{\{1\}, \{2\}, \{3\}, \{4\}, \{5\}, \{6\}, \{1,2\}, \{3,6\}, \{1,2,5\}, \{3,4,6\}, \{1,\dots, 6\}\} $.

For $x \in \boldsymbol{T}_n$, we fix a clade $c \in \boldsymbol{C}_n$ and consider the map $s_c : \boldsymbol{T}_n \to \{0, 1\}$. Defined as the indicator function that determines if a tree contains the specific clade $c$. 
For all $c \in \boldsymbol{C}_n$, these two sets form a partition of $\TT_n$ divided into trees that do not contain the clade $c$ and those that do (Definition~\ref{def:clade_partition}).

\begin{definition}[\textbf{Clade partition of tree space}]
\label{def:clade_partition}
Let $\bar{S}_n(c) = \{S_0(c), S_1(c) \} $ be the partition of $\boldsymbol{T}_n$ induced by clade $c \in \boldsymbol{C}_n$, for which we will write $S_0(c) := \{y \in \boldsymbol{T}_n : s_c(y) = 0\}$ and $S_1(c)  := \{y \in \boldsymbol{T}
_n : s_c(y) = 1\} = \boldsymbol{T}_n \setminus S_0(c)$.
\end{definition}

Let us define $\boldsymbol{S}_n$ as the set of unlabelled rooted binary trees, representing trees from $\TT_n$ without labels on the leaves.
The cardinality of $\boldsymbol{S}_n$ is less than that of $\TT_n$ since a single unlabelled tree in $\boldsymbol{S}_n$ can correspond to multiple labelled configurations in $\TT_n$. The differences in cardinalities can be observed in the provided Table~\ref{tab:cardinalities} and for more detailed visualisation of the elements in $\boldsymbol{S}_n$ and $\TT_n$, we refer to Figure~\ref{fig:TT_n and SS_n}. 

\begin{table}[!ht]
\label{tab:}
\caption{Counting rooted trees and shapes.}
\centering
\begin{tabular}{@{}ccc@{}}
\toprule
$n$ & $|\TT_n|$    &   $|\boldsymbol{S}_n|$ \\ \midrule
4 & 15    & 2  \\
5 & 105   & 3  \\
6 & 945   & 6  \\
7 & 10395 & 11 \\ \bottomrule
\end{tabular}
\label{tab:cardinalities}
\end{table}

\begin{figure}[h]
    \centering

\tikzset{every picture/.style={line width=0.75pt}} 

\begin{tikzpicture}[x=0.75pt,y=0.75pt,yscale=-1,xscale=1]

\draw    (141,39.5) -- (189.5,130.75) ;
\draw    (141,39.5) -- (92.5,130.75) ;
\draw    (108.5,102) -- (120,129.75) ;
\draw    (121,78) -- (143.5,129.75) ;
\draw    (174.5,102.5) -- (163,130.25) ;
\draw   (90,134.63) .. controls (90,132.48) and (91.12,130.75) .. (92.5,130.75) .. controls (93.88,130.75) and (95,132.48) .. (95,134.63) .. controls (95,136.77) and (93.88,138.5) .. (92.5,138.5) .. controls (91.12,138.5) and (90,136.77) .. (90,134.63) -- cycle ;
\draw   (117.5,133.63) .. controls (117.5,131.48) and (118.62,129.75) .. (120,129.75) .. controls (121.38,129.75) and (122.5,131.48) .. (122.5,133.63) .. controls (122.5,135.77) and (121.38,137.5) .. (120,137.5) .. controls (118.62,137.5) and (117.5,135.77) .. (117.5,133.63) -- cycle ;
\draw   (141,133.63) .. controls (141,131.48) and (142.12,129.75) .. (143.5,129.75) .. controls (144.88,129.75) and (146,131.48) .. (146,133.63) .. controls (146,135.77) and (144.88,137.5) .. (143.5,137.5) .. controls (142.12,137.5) and (141,135.77) .. (141,133.63) -- cycle ;
\draw   (160.5,134.13) .. controls (160.5,131.98) and (161.62,130.25) .. (163,130.25) .. controls (164.38,130.25) and (165.5,131.98) .. (165.5,134.13) .. controls (165.5,136.27) and (164.38,138) .. (163,138) .. controls (161.62,138) and (160.5,136.27) .. (160.5,134.13) -- cycle ;
\draw   (187,134.63) .. controls (187,132.48) and (188.12,130.75) .. (189.5,130.75) .. controls (190.88,130.75) and (192,132.48) .. (192,134.63) .. controls (192,136.77) and (190.88,138.5) .. (189.5,138.5) .. controls (188.12,138.5) and (187,136.77) .. (187,134.63) -- cycle ;
\draw    (272.33,39.5) -- (320.83,130.75) ;
\draw    (272.33,39.5) -- (223.83,130.75) ;
\draw    (239.83,102) -- (251.33,129.75) ;
\draw    (252.33,78) -- (274.83,129.75) ;
\draw    (305.83,102.5) -- (294.33,130.25) ;
\draw   (221.33,134.63) .. controls (221.33,132.48) and (222.45,130.75) .. (223.83,130.75) .. controls (225.21,130.75) and (226.33,132.48) .. (226.33,134.63) .. controls (226.33,136.77) and (225.21,138.5) .. (223.83,138.5) .. controls (222.45,138.5) and (221.33,136.77) .. (221.33,134.63) -- cycle ;
\draw   (248.83,133.63) .. controls (248.83,131.48) and (249.95,129.75) .. (251.33,129.75) .. controls (252.71,129.75) and (253.83,131.48) .. (253.83,133.63) .. controls (253.83,135.77) and (252.71,137.5) .. (251.33,137.5) .. controls (249.95,137.5) and (248.83,135.77) .. (248.83,133.63) -- cycle ;
\draw   (272.33,133.63) .. controls (272.33,131.48) and (273.45,129.75) .. (274.83,129.75) .. controls (276.21,129.75) and (277.33,131.48) .. (277.33,133.63) .. controls (277.33,135.77) and (276.21,137.5) .. (274.83,137.5) .. controls (273.45,137.5) and (272.33,135.77) .. (272.33,133.63) -- cycle ;
\draw   (291.83,134.13) .. controls (291.83,131.98) and (292.95,130.25) .. (294.33,130.25) .. controls (295.71,130.25) and (296.83,131.98) .. (296.83,134.13) .. controls (296.83,136.27) and (295.71,138) .. (294.33,138) .. controls (292.95,138) and (291.83,136.27) .. (291.83,134.13) -- cycle ;
\draw   (318.33,134.63) .. controls (318.33,132.48) and (319.45,130.75) .. (320.83,130.75) .. controls (322.21,130.75) and (323.33,132.48) .. (323.33,134.63) .. controls (323.33,136.77) and (322.21,138.5) .. (320.83,138.5) .. controls (319.45,138.5) and (318.33,136.77) .. (318.33,134.63) -- cycle ;

\draw (88.4,140.6) node [anchor=north west][inner sep=0.75pt]  [font=\tiny]  {$1$};
\draw (140,140.2) node [anchor=north west][inner sep=0.75pt]  [font=\tiny]  {$2$};
\draw (160,140.2) node [anchor=north west][inner sep=0.75pt]  [font=\tiny]  {$3$};
\draw (117.6,139.4) node [anchor=north west][inner sep=0.75pt]  [font=\tiny]  {$5$};
\draw (187,140.03) node [anchor=north west][inner sep=0.75pt]  [font=\tiny]  {$4$};

\end{tikzpicture}
    
    \caption{\textbf{Labelled and unlabelled rooted trees.}
    On the right, we depict a rooted, labelled binary phylogenetic tree with 6 leaves.
    On the left, we present the same tree, but without labels. From the unlabelled tree on the left, we can generate 30 distinct trees in $\TT_5$.}
    \label{fig:TT_n and SS_n}
\end{figure}
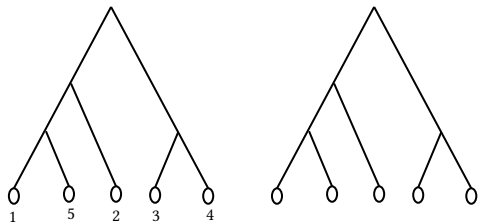

The set $\boldsymbol{S}_n$ can also be interpreted as a partition of $\boldsymbol{T}_n$.
For instance, consider $F \in \boldsymbol{S}_n$.
As observed, $F$ can generate a specific number of trees in $\boldsymbol{T}_n$, which depends on the shape of $F$.
This is achieved by assigning the labels $\{1, 2, \dots, n\}$ to the leaves.
Consequently, we obtain a partition $\bar{F} := \{F_1, F_2, \dots, F_v\}$ of $\boldsymbol{T}_n$, where each $F_i$ represents a distinct element of $\boldsymbol{S}_n$.
Here, $v$ denotes the number of unique tree shapes, which is given by $\frac{1}{n} \binom{2(n-1)}{n-1}$, a Catalan number (see \cite{Billera2001}).
We refer to $\bar{F}$ as the \textit{tree shape partition}.

We can also define operations on  $\TT_n$ to endow the space with a metric.
A rooted \textit{subtree-prune-regraft} (rSPR) operation on a tree $t \in \TT_n$, removes a subtree from 
$t$, suppresses the parent of this subtree, and then reattaches it to another edge or directly to the root of tree $t$, subsequently forming a new parent. A graphical description is given in Figure~\ref{fig:spr_des}.
Thus essentially it is possible to have three different rSPR operations (using the notation of Figure~\ref{fig:spr_des}):
\begin{itemize}
    \item [i)] Cut edge $e_1$, suppress the parent (a degree-3 vertex) of the subtree below $e_1$, and then regraft onto edge $e_2$ creating a new degree-3 vertex.  

    \item [ii)] Cut edge $e_3$, suppress the root, and regraft onto edge $e_1$ generating a new degree-3 vertex and a new root that was a degree-3 vertice before.

    \item[iii)] Cut edge $e_1$, suppress the parent, and regraft into the root, creating a new root.  
\end{itemize}

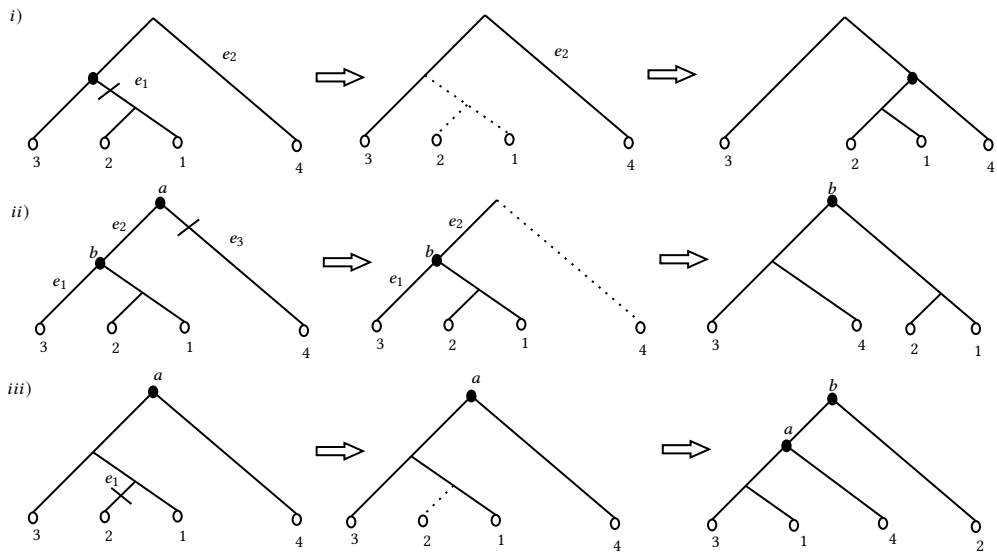
\begin{figure}[h]
    \centering

\tikzset{every picture/.style={line width=0.75pt}} 

\begin{tikzpicture}[x=0.75pt,y=0.75pt,yscale=-1,xscale=1]

\draw    (80.13,31.83) -- (151.81,93.3) ;
\draw    (80.13,31.83) -- (20.15,92.46) ;
\draw    (50.14,62.15) -- (92.32,91.63) ;
\draw    (71.23,76.89) -- (55.99,91.94) ;
\draw   (18.08,95.19) .. controls (18.08,93.69) and (19,92.46) .. (20.15,92.46) .. controls (21.29,92.46) and (22.22,93.69) .. (22.22,95.19) .. controls (22.22,96.7) and (21.29,97.92) .. (20.15,97.92) .. controls (19,97.92) and (18.08,96.7) .. (18.08,95.19) -- cycle ;
\draw  [fill={rgb, 255:red, 0; green, 0; blue, 0 }  ,fill opacity=1 ] (48.07,62.15) .. controls (48.07,60.64) and (49,59.42) .. (50.14,59.42) .. controls (51.28,59.42) and (52.21,60.64) .. (52.21,62.15) .. controls (52.21,63.66) and (51.28,64.88) .. (50.14,64.88) .. controls (49,64.88) and (48.07,63.66) .. (48.07,62.15) -- cycle ;
\draw   (53.92,94.67) .. controls (53.92,93.16) and (54.85,91.94) .. (55.99,91.94) .. controls (57.14,91.94) and (58.06,93.16) .. (58.06,94.67) .. controls (58.06,96.18) and (57.14,97.4) .. (55.99,97.4) .. controls (54.85,97.4) and (53.92,96.18) .. (53.92,94.67) -- cycle ;
\draw   (90.25,94.36) .. controls (90.25,92.85) and (91.18,91.63) .. (92.32,91.63) .. controls (93.47,91.63) and (94.39,92.85) .. (94.39,94.36) .. controls (94.39,95.86) and (93.47,97.09) .. (92.32,97.09) .. controls (91.18,97.09) and (90.25,95.86) .. (90.25,94.36) -- cycle ;
\draw   (149.74,96.03) .. controls (149.74,94.52) and (150.67,93.3) .. (151.81,93.3) .. controls (152.96,93.3) and (153.89,94.52) .. (153.89,96.03) .. controls (153.89,97.54) and (152.96,98.76) .. (151.81,98.76) .. controls (150.67,98.76) and (149.74,97.54) .. (149.74,96.03) -- cycle ;
\draw   (161.74,59.56) -- (175.61,59.56) -- (175.61,57.5) -- (184.86,61.63) -- (175.61,65.75) -- (175.61,63.69) -- (161.74,63.69) -- cycle ;
\draw    (63.36,64.6) -- (52.69,72.75) ;
\draw    (245.78,30.33) -- (317.47,91.8) ;
\draw    (245.78,30.33) -- (185.8,90.96) ;
\draw  [dash pattern={on 0.84pt off 2.51pt}]  (215.79,60.65) -- (257.97,90.13) ;
\draw  [dash pattern={on 0.84pt off 2.51pt}]  (236.88,75.39) -- (221.64,90.44) ;
\draw   (219.57,93.17) .. controls (219.57,91.66) and (220.5,90.44) .. (221.64,90.44) .. controls (222.79,90.44) and (223.72,91.66) .. (223.72,93.17) .. controls (223.72,94.68) and (222.79,95.9) .. (221.64,95.9) .. controls (220.5,95.9) and (219.57,94.68) .. (219.57,93.17) -- cycle ;
\draw   (255.9,92.86) .. controls (255.9,91.35) and (256.83,90.13) .. (257.97,90.13) .. controls (259.12,90.13) and (260.05,91.35) .. (260.05,92.86) .. controls (260.05,94.36) and (259.12,95.59) .. (257.97,95.59) .. controls (256.83,95.59) and (255.9,94.36) .. (255.9,92.86) -- cycle ;
\draw   (315.39,94.53) .. controls (315.39,93.02) and (316.32,91.8) .. (317.47,91.8) .. controls (318.61,91.8) and (319.54,93.02) .. (319.54,94.53) .. controls (319.54,96.04) and (318.61,97.26) .. (317.47,97.26) .. controls (316.32,97.26) and (315.39,96.04) .. (315.39,94.53) -- cycle ;
\draw   (327.39,58.06) -- (341.26,58.06) -- (341.26,56) -- (350.51,60.13) -- (341.26,64.25) -- (341.26,62.19) -- (327.39,62.19) -- cycle ;
\draw   (183.73,93.69) .. controls (183.73,92.19) and (184.66,90.96) .. (185.8,90.96) .. controls (186.95,90.96) and (187.87,92.19) .. (187.87,93.69) .. controls (187.87,95.2) and (186.95,96.42) .. (185.8,96.42) .. controls (184.66,96.42) and (183.73,95.2) .. (183.73,93.69) -- cycle ;
\draw    (425.36,31.33) -- (497.05,92.8) ;
\draw    (425.36,31.33) -- (365.38,91.96) ;
\draw   (363.31,94.69) .. controls (363.31,93.19) and (364.24,91.96) .. (365.38,91.96) .. controls (366.52,91.96) and (367.45,93.19) .. (367.45,94.69) .. controls (367.45,96.2) and (366.52,97.42) .. (365.38,97.42) .. controls (364.24,97.42) and (363.31,96.2) .. (363.31,94.69) -- cycle ;
\draw  [fill={rgb, 255:red, 0; green, 0; blue, 0 }  ,fill opacity=1 ] (457.06,62.07) .. controls (457.06,60.56) and (457.99,59.34) .. (459.13,59.34) .. controls (460.28,59.34) and (461.2,60.56) .. (461.2,62.07) .. controls (461.2,63.57) and (460.28,64.8) .. (459.13,64.8) .. controls (457.99,64.8) and (457.06,63.57) .. (457.06,62.07) -- cycle ;
\draw   (494.97,95.53) .. controls (494.97,94.02) and (495.9,92.8) .. (497.05,92.8) .. controls (498.19,92.8) and (499.12,94.02) .. (499.12,95.53) .. controls (499.12,97.04) and (498.19,98.26) .. (497.05,98.26) .. controls (495.9,98.26) and (494.97,97.04) .. (494.97,95.53) -- cycle ;
\draw    (459.13,62.07) -- (428.44,92.75) ;
\draw    (443.79,77.41) -- (462.82,91.25) ;
\draw   (426.41,95.17) .. controls (426.41,93.66) and (427.34,92.44) .. (428.49,92.44) .. controls (429.63,92.44) and (430.56,93.66) .. (430.56,95.17) .. controls (430.56,96.68) and (429.63,97.9) .. (428.49,97.9) .. controls (427.34,97.9) and (426.41,96.68) .. (426.41,95.17) -- cycle ;
\draw   (461.56,93.86) .. controls (461.56,92.35) and (462.48,91.13) .. (463.63,91.13) .. controls (464.77,91.13) and (465.7,92.35) .. (465.7,93.86) .. controls (465.7,95.36) and (464.77,96.59) .. (463.63,96.59) .. controls (462.48,96.59) and (461.56,95.36) .. (461.56,93.86) -- cycle ;
\draw    (83.69,124.83) -- (155.37,186.3) ;
\draw    (83.69,124.83) -- (23.71,185.46) ;
\draw    (53.7,155.15) -- (95.88,184.63) ;
\draw    (74.79,169.89) -- (59.55,184.94) ;
\draw   (21.63,188.19) .. controls (21.63,186.69) and (22.56,185.46) .. (23.71,185.46) .. controls (24.85,185.46) and (25.78,186.69) .. (25.78,188.19) .. controls (25.78,189.7) and (24.85,190.92) .. (23.71,190.92) .. controls (22.56,190.92) and (21.63,189.7) .. (21.63,188.19) -- cycle ;
\draw  [fill={rgb, 255:red, 0; green, 0; blue, 0 }  ,fill opacity=1 ] (81.61,124.83) .. controls (81.61,123.33) and (82.54,122.1) .. (83.69,122.1) .. controls (84.83,122.1) and (85.76,123.33) .. (85.76,124.83) .. controls (85.76,126.34) and (84.83,127.56) .. (83.69,127.56) .. controls (82.54,127.56) and (81.61,126.34) .. (81.61,124.83) -- cycle ;
\draw   (57.47,187.67) .. controls (57.47,186.16) and (58.4,184.94) .. (59.55,184.94) .. controls (60.69,184.94) and (61.62,186.16) .. (61.62,187.67) .. controls (61.62,189.18) and (60.69,190.4) .. (59.55,190.4) .. controls (58.4,190.4) and (57.47,189.18) .. (57.47,187.67) -- cycle ;
\draw   (93.8,187.36) .. controls (93.8,185.85) and (94.73,184.63) .. (95.88,184.63) .. controls (97.02,184.63) and (97.95,185.85) .. (97.95,187.36) .. controls (97.95,188.86) and (97.02,190.09) .. (95.88,190.09) .. controls (94.73,190.09) and (93.8,188.86) .. (93.8,187.36) -- cycle ;
\draw   (153.3,189.03) .. controls (153.3,187.52) and (154.23,186.3) .. (155.37,186.3) .. controls (156.52,186.3) and (157.44,187.52) .. (157.44,189.03) .. controls (157.44,190.54) and (156.52,191.76) .. (155.37,191.76) .. controls (154.23,191.76) and (153.3,190.54) .. (153.3,189.03) -- cycle ;
\draw   (165.3,152.56) -- (179.17,152.56) -- (179.17,150.5) -- (188.41,154.63) -- (179.17,158.75) -- (179.17,156.69) -- (165.3,156.69) -- cycle ;
\draw    (103.07,133.1) -- (92.4,141.25) ;
\draw  [dash pattern={on 0.84pt off 2.51pt}]  (251.71,123.33) -- (323.39,184.8) ;
\draw    (251.71,123.33) -- (191.73,183.96) ;
\draw    (221.72,153.65) -- (263.9,183.13) ;
\draw    (242.81,168.39) -- (227.57,183.44) ;
\draw   (189.66,186.69) .. controls (189.66,185.19) and (190.58,183.96) .. (191.73,183.96) .. controls (192.87,183.96) and (193.8,185.19) .. (193.8,186.69) .. controls (193.8,188.2) and (192.87,189.42) .. (191.73,189.42) .. controls (190.58,189.42) and (189.66,188.2) .. (189.66,186.69) -- cycle ;
\draw   (225.5,186.17) .. controls (225.5,184.66) and (226.43,183.44) .. (227.57,183.44) .. controls (228.72,183.44) and (229.64,184.66) .. (229.64,186.17) .. controls (229.64,187.68) and (228.72,188.9) .. (227.57,188.9) .. controls (226.43,188.9) and (225.5,187.68) .. (225.5,186.17) -- cycle ;
\draw   (261.83,185.86) .. controls (261.83,184.35) and (262.76,183.13) .. (263.9,183.13) .. controls (265.04,183.13) and (265.97,184.35) .. (265.97,185.86) .. controls (265.97,187.36) and (265.04,188.59) .. (263.9,188.59) .. controls (262.76,188.59) and (261.83,187.36) .. (261.83,185.86) -- cycle ;
\draw   (321.32,187.53) .. controls (321.32,186.02) and (322.25,184.8) .. (323.39,184.8) .. controls (324.54,184.8) and (325.47,186.02) .. (325.47,187.53) .. controls (325.47,189.04) and (324.54,190.26) .. (323.39,190.26) .. controls (322.25,190.26) and (321.32,189.04) .. (321.32,187.53) -- cycle ;
\draw   (333.32,151.06) -- (347.19,151.06) -- (347.19,149) -- (356.44,153.13) -- (347.19,157.25) -- (347.19,155.19) -- (333.32,155.19) -- cycle ;
\draw    (419.14,123.83) -- (490.82,185.3) ;
\draw    (419.14,123.83) -- (359.16,184.46) ;
\draw    (389.15,154.15) -- (431.33,183.63) ;
\draw   (357.08,187.19) .. controls (357.08,185.69) and (358.01,184.46) .. (359.16,184.46) .. controls (360.3,184.46) and (361.23,185.69) .. (361.23,187.19) .. controls (361.23,188.7) and (360.3,189.92) .. (359.16,189.92) .. controls (358.01,189.92) and (357.08,188.7) .. (357.08,187.19) -- cycle ;
\draw  [fill={rgb, 255:red, 0; green, 0; blue, 0 }  ,fill opacity=1 ] (417.06,123.83) .. controls (417.06,122.33) and (417.99,121.1) .. (419.14,121.1) .. controls (420.28,121.1) and (421.21,122.33) .. (421.21,123.83) .. controls (421.21,125.34) and (420.28,126.56) .. (419.14,126.56) .. controls (417.99,126.56) and (417.06,125.34) .. (417.06,123.83) -- cycle ;
\draw   (429.25,186.36) .. controls (429.25,184.85) and (430.18,183.63) .. (431.33,183.63) .. controls (432.47,183.63) and (433.4,184.85) .. (433.4,186.36) .. controls (433.4,187.86) and (432.47,189.09) .. (431.33,189.09) .. controls (430.18,189.09) and (429.25,187.86) .. (429.25,186.36) -- cycle ;
\draw   (488.75,188.03) .. controls (488.75,186.52) and (489.68,185.3) .. (490.82,185.3) .. controls (491.97,185.3) and (492.89,186.52) .. (492.89,188.03) .. controls (492.89,189.54) and (491.97,190.76) .. (490.82,190.76) .. controls (489.68,190.76) and (488.75,189.54) .. (488.75,188.03) -- cycle ;
\draw    (473.36,170.39) -- (458.12,185.44) ;
\draw   (456.05,188.17) .. controls (456.05,186.66) and (456.97,185.44) .. (458.12,185.44) .. controls (459.26,185.44) and (460.19,186.66) .. (460.19,188.17) .. controls (460.19,189.68) and (459.26,190.9) .. (458.12,190.9) .. controls (456.97,190.9) and (456.05,189.68) .. (456.05,188.17) -- cycle ;
\draw    (80.13,219.83) -- (151.81,281.3) ;
\draw    (80.13,219.83) -- (20.15,280.46) ;
\draw    (50.14,250.15) -- (92.32,279.63) ;
\draw    (71.23,264.89) -- (55.99,279.94) ;
\draw   (18.08,283.19) .. controls (18.08,281.69) and (19,280.46) .. (20.15,280.46) .. controls (21.29,280.46) and (22.22,281.69) .. (22.22,283.19) .. controls (22.22,284.7) and (21.29,285.92) .. (20.15,285.92) .. controls (19,285.92) and (18.08,284.7) .. (18.08,283.19) -- cycle ;
\draw  [fill={rgb, 255:red, 0; green, 0; blue, 0 }  ,fill opacity=1 ] (78.06,219.83) .. controls (78.06,218.33) and (78.99,217.1) .. (80.13,217.1) .. controls (81.27,217.1) and (82.2,218.33) .. (82.2,219.83) .. controls (82.2,221.34) and (81.27,222.56) .. (80.13,222.56) .. controls (78.99,222.56) and (78.06,221.34) .. (78.06,219.83) -- cycle ;
\draw   (53.92,282.67) .. controls (53.92,281.16) and (54.85,279.94) .. (55.99,279.94) .. controls (57.14,279.94) and (58.06,281.16) .. (58.06,282.67) .. controls (58.06,284.18) and (57.14,285.4) .. (55.99,285.4) .. controls (54.85,285.4) and (53.92,284.18) .. (53.92,282.67) -- cycle ;
\draw   (90.25,282.36) .. controls (90.25,280.85) and (91.18,279.63) .. (92.32,279.63) .. controls (93.47,279.63) and (94.39,280.85) .. (94.39,282.36) .. controls (94.39,283.86) and (93.47,285.09) .. (92.32,285.09) .. controls (91.18,285.09) and (90.25,283.86) .. (90.25,282.36) -- cycle ;
\draw   (149.74,284.03) .. controls (149.74,282.52) and (150.67,281.3) .. (151.81,281.3) .. controls (152.96,281.3) and (153.89,282.52) .. (153.89,284.03) .. controls (153.89,285.54) and (152.96,286.76) .. (151.81,286.76) .. controls (150.67,286.76) and (149.74,285.54) .. (149.74,284.03) -- cycle ;
\draw   (161.74,247.56) -- (175.61,247.56) -- (175.61,245.5) -- (184.86,249.63) -- (175.61,253.75) -- (175.61,251.69) -- (161.74,251.69) -- cycle ;
\draw    (59.21,268.1) -- (69.29,276.25) ;
\draw   (334.51,246.64) -- (348.38,246.64) -- (348.38,244.57) -- (357.62,248.7) -- (348.38,252.82) -- (348.38,250.76) -- (334.51,250.76) -- cycle ;
\draw  [fill={rgb, 255:red, 0; green, 0; blue, 0 }  ,fill opacity=1 ] (51.62,155.15) .. controls (51.62,153.64) and (52.55,152.42) .. (53.7,152.42) .. controls (54.84,152.42) and (55.77,153.64) .. (55.77,155.15) .. controls (55.77,156.66) and (54.84,157.88) .. (53.7,157.88) .. controls (52.55,157.88) and (51.62,156.66) .. (51.62,155.15) -- cycle ;
\draw  [fill={rgb, 255:red, 0; green, 0; blue, 0 }  ,fill opacity=1 ] (219.65,153.65) .. controls (219.65,152.14) and (220.57,150.92) .. (221.72,150.92) .. controls (222.86,150.92) and (223.79,152.14) .. (223.79,153.65) .. controls (223.79,155.16) and (222.86,156.38) .. (221.72,156.38) .. controls (220.57,156.38) and (219.65,155.16) .. (219.65,153.65) -- cycle ;
\draw    (238.97,221.91) -- (310.65,283.37) ;
\draw    (238.97,221.91) -- (178.98,282.54) ;
\draw    (208.97,252.22) -- (251.16,281.7) ;
\draw  [dash pattern={on 0.84pt off 2.51pt}]  (230.07,266.96) -- (214.83,282.01) ;
\draw   (176.91,285.27) .. controls (176.91,283.76) and (177.84,282.54) .. (178.98,282.54) .. controls (180.13,282.54) and (181.06,283.76) .. (181.06,285.27) .. controls (181.06,286.77) and (180.13,288) .. (178.98,288) .. controls (177.84,288) and (176.91,286.77) .. (176.91,285.27) -- cycle ;
\draw  [fill={rgb, 255:red, 0; green, 0; blue, 0 }  ,fill opacity=1 ] (236.89,221.91) .. controls (236.89,220.4) and (237.82,219.18) .. (238.97,219.18) .. controls (240.11,219.18) and (241.04,220.4) .. (241.04,221.91) .. controls (241.04,223.41) and (240.11,224.64) .. (238.97,224.64) .. controls (237.82,224.64) and (236.89,223.41) .. (236.89,221.91) -- cycle ;
\draw   (212.75,284.74) .. controls (212.75,283.24) and (213.68,282.01) .. (214.83,282.01) .. controls (215.97,282.01) and (216.9,283.24) .. (216.9,284.74) .. controls (216.9,286.25) and (215.97,287.47) .. (214.83,287.47) .. controls (213.68,287.47) and (212.75,286.25) .. (212.75,284.74) -- cycle ;
\draw   (249.08,284.43) .. controls (249.08,282.92) and (250.01,281.7) .. (251.16,281.7) .. controls (252.3,281.7) and (253.23,282.92) .. (253.23,284.43) .. controls (253.23,285.94) and (252.3,287.16) .. (251.16,287.16) .. controls (250.01,287.16) and (249.08,285.94) .. (249.08,284.43) -- cycle ;
\draw   (308.58,286.1) .. controls (308.58,284.6) and (309.5,283.37) .. (310.65,283.37) .. controls (311.79,283.37) and (312.72,284.6) .. (312.72,286.1) .. controls (312.72,287.61) and (311.79,288.83) .. (310.65,288.83) .. controls (309.5,288.83) and (308.58,287.61) .. (308.58,286.1) -- cycle ;
\draw    (419.14,223.41) -- (490.82,284.87) ;
\draw    (419.14,223.41) -- (359.16,284.04) ;
\draw    (376.11,267.22) -- (400,284.25) ;
\draw   (357.08,286.77) .. controls (357.08,285.26) and (358.01,284.04) .. (359.16,284.04) .. controls (360.3,284.04) and (361.23,285.26) .. (361.23,286.77) .. controls (361.23,288.27) and (360.3,289.5) .. (359.16,289.5) .. controls (358.01,289.5) and (357.08,288.27) .. (357.08,286.77) -- cycle ;
\draw  [fill={rgb, 255:red, 0; green, 0; blue, 0 }  ,fill opacity=1 ] (417.06,223.41) .. controls (417.06,221.9) and (417.99,220.68) .. (419.14,220.68) .. controls (420.28,220.68) and (421.21,221.9) .. (421.21,223.41) .. controls (421.21,224.91) and (420.28,226.14) .. (419.14,226.14) .. controls (417.99,226.14) and (417.06,224.91) .. (417.06,223.41) -- cycle ;
\draw   (397.92,286.98) .. controls (397.92,285.47) and (398.85,284.25) .. (400,284.25) .. controls (401.14,284.25) and (402.07,285.47) .. (402.07,286.98) .. controls (402.07,288.49) and (401.14,289.71) .. (400,289.71) .. controls (398.85,289.71) and (397.92,288.49) .. (397.92,286.98) -- cycle ;
\draw   (488.75,287.6) .. controls (488.75,286.1) and (489.68,284.87) .. (490.82,284.87) .. controls (491.97,284.87) and (492.89,286.1) .. (492.89,287.6) .. controls (492.89,289.11) and (491.97,290.33) .. (490.82,290.33) .. controls (489.68,290.33) and (488.75,289.11) .. (488.75,287.6) -- cycle ;
\draw    (396.26,246.97) -- (444.45,283.25) ;
\draw   (442.37,285.98) .. controls (442.37,284.47) and (443.3,283.25) .. (444.45,283.25) .. controls (445.59,283.25) and (446.52,284.47) .. (446.52,285.98) .. controls (446.52,287.49) and (445.59,288.71) .. (444.45,288.71) .. controls (443.3,288.71) and (442.37,287.49) .. (442.37,285.98) -- cycle ;
\draw  [fill={rgb, 255:red, 0; green, 0; blue, 0 }  ,fill opacity=1 ] (394.19,246.97) .. controls (394.19,245.47) and (395.11,244.24) .. (396.26,244.24) .. controls (397.4,244.24) and (398.33,245.47) .. (398.33,246.97) .. controls (398.33,248.48) and (397.4,249.7) .. (396.26,249.7) .. controls (395.11,249.7) and (394.19,248.48) .. (394.19,246.97) -- cycle ;

\draw (91.44,99.91) node [anchor=north west][inner sep=0.75pt]  [font=\tiny]  {$1$};
\draw (54.5,100.22) node [anchor=north west][inner sep=0.75pt]  [font=\tiny]  {$2$};
\draw (18.54,99.69) node [anchor=north west][inner sep=0.75pt]  [font=\tiny]  {$3$};
\draw (149.47,102.83) node [anchor=north west][inner sep=0.75pt]  [font=\tiny]  {$4$};
\draw (69.21,60.4) node [anchor=north west][inner sep=0.75pt]  [font=\tiny]  {$e_{1}$};
\draw (113.07,46.9) node [anchor=north west][inner sep=0.75pt]  [font=\tiny]  {$e_{2}$};
\draw (6.98,25.4) node [anchor=north west][inner sep=0.75pt]  [font=\tiny]  {$i)$};
\draw (257.09,98.41) node [anchor=north west][inner sep=0.75pt]  [font=\tiny]  {$1$};
\draw (220.15,98.72) node [anchor=north west][inner sep=0.75pt]  [font=\tiny]  {$2$};
\draw (184.19,98.19) node [anchor=north west][inner sep=0.75pt]  [font=\tiny]  {$3$};
\draw (315.12,101.33) node [anchor=north west][inner sep=0.75pt]  [font=\tiny]  {$4$};
\draw (278.72,45.4) node [anchor=north west][inner sep=0.75pt]  [font=\tiny]  {$e_{2}$};
\draw (363.77,99.19) node [anchor=north west][inner sep=0.75pt]  [font=\tiny]  {$3$};
\draw (494.7,102.33) node [anchor=north west][inner sep=0.75pt]  [font=\tiny]  {$4$};
\draw (427,100.72) node [anchor=north west][inner sep=0.75pt]  [font=\tiny]  {$2$};
\draw (462.75,99.41) node [anchor=north west][inner sep=0.75pt]  [font=\tiny]  {$1$};
\draw (95,192.91) node [anchor=north west][inner sep=0.75pt]  [font=\tiny]  {$1$};
\draw (58.06,193.22) node [anchor=north west][inner sep=0.75pt]  [font=\tiny]  {$2$};
\draw (22.09,192.69) node [anchor=north west][inner sep=0.75pt]  [font=\tiny]  {$3$};
\draw (153.03,195.83) node [anchor=north west][inner sep=0.75pt]  [font=\tiny]  {$4$};
\draw (28.32,161.4) node [anchor=north west][inner sep=0.75pt]  [font=\tiny]  {$e_{1}$};
\draw (116.63,139.9) node [anchor=north west][inner sep=0.75pt]  [font=\tiny]  {$e_{3}$};
\draw (7.58,124.4) node [anchor=north west][inner sep=0.75pt]  [font=\tiny]  {$ii)$};
\draw (57.95,132.4) node [anchor=north west][inner sep=0.75pt]  [font=\tiny]  {$e_{2}$};
\draw (262.35,190.76) node [anchor=north west][inner sep=0.75pt]  [font=\tiny]  {$1$};
\draw (226.08,191.72) node [anchor=north west][inner sep=0.75pt]  [font=\tiny]  {$2$};
\draw (190.12,191.19) node [anchor=north west][inner sep=0.75pt]  [font=\tiny]  {$3$};
\draw (321.05,194.33) node [anchor=north west][inner sep=0.75pt]  [font=\tiny]  {$4$};
\draw (196.34,159.9) node [anchor=north west][inner sep=0.75pt]  [font=\tiny]  {$e_{1}$};
\draw (225.98,130.9) node [anchor=north west][inner sep=0.75pt]  [font=\tiny]  {$e_{2}$};
\draw (430.45,191.91) node [anchor=north west][inner sep=0.75pt]  [font=\tiny]  {$4$};
\draw (357.54,191.69) node [anchor=north west][inner sep=0.75pt]  [font=\tiny]  {$3$};
\draw (488.48,194.83) node [anchor=north west][inner sep=0.75pt]  [font=\tiny]  {$1$};
\draw (456.63,193.72) node [anchor=north west][inner sep=0.75pt]  [font=\tiny]  {$2$};
\draw (91.44,287.91) node [anchor=north west][inner sep=0.75pt]  [font=\tiny]  {$1$};
\draw (54.5,288.22) node [anchor=north west][inner sep=0.75pt]  [font=\tiny]  {$2$};
\draw (18.54,287.69) node [anchor=north west][inner sep=0.75pt]  [font=\tiny]  {$3$};
\draw (149.47,290.83) node [anchor=north west][inner sep=0.75pt]  [font=\tiny]  {$4$};
\draw (54.4,260.4) node [anchor=north west][inner sep=0.75pt]  [font=\tiny]  {$e_{1}$};
\draw (5.8,213.4) node [anchor=north west][inner sep=0.75pt]  [font=\tiny]  {$iii)$};
\draw (81.07,114.4) node [anchor=north west][inner sep=0.75pt]  [font=\tiny]  {$a$};
\draw (46.69,144.9) node [anchor=north west][inner sep=0.75pt]  [font=\tiny]  {$b$};
\draw (214.42,144.4) node [anchor=north west][inner sep=0.75pt]  [font=\tiny]  {$b$};
\draw (415.33,112.9) node [anchor=north west][inner sep=0.75pt]  [font=\tiny]  {$b$};
\draw (250.28,289.98) node [anchor=north west][inner sep=0.75pt]  [font=\tiny]  {$1$};
\draw (213.34,290.3) node [anchor=north west][inner sep=0.75pt]  [font=\tiny]  {$2$};
\draw (177.37,289.76) node [anchor=north west][inner sep=0.75pt]  [font=\tiny]  {$3$};
\draw (308.31,292.9) node [anchor=north west][inner sep=0.75pt]  [font=\tiny]  {$4$};
\draw (78.7,208.4) node [anchor=north west][inner sep=0.75pt]  [font=\tiny]  {$a$};
\draw (237.53,209.9) node [anchor=north west][inner sep=0.75pt]  [font=\tiny]  {$a$};
\draw (401.41,290.38) node [anchor=north west][inner sep=0.75pt]  [font=\tiny]  {$1$};
\draw (489.52,292.8) node [anchor=north west][inner sep=0.75pt]  [font=\tiny]  {$2$};
\draw (357.54,291.26) node [anchor=north west][inner sep=0.75pt]  [font=\tiny]  {$3$};
\draw (445.86,289.38) node [anchor=north west][inner sep=0.75pt]  [font=\tiny]  {$4$};
\draw (393.4,235.9) node [anchor=north west][inner sep=0.75pt]  [font=\tiny]  {$a$};
\draw (415.92,211.9) node [anchor=north west][inner sep=0.75pt]  [font=\tiny]  {$b$};

\end{tikzpicture} 
    
    \caption{\textbf{The three possible rooted subtree prune-and-regraft (rSPR) operations.}
    As described in the text, there are three main ways a rSPR operation can be performed, depending on how it interacts with the root.
    Please notice that open circles mean leaves and closed circles internal vertices.}
    \label{fig:spr_des}
\end{figure}

Building on this premise, we can define the rSPR-graph as $G_n:= G(\TT_n, E_n)$, wherein the vertices represent trees, and the set of edges $E_n$ indicates that an rSPR operation can transform one tree into another.
Furthermore, we introduce a measure of distance, the SPR-metric $d_{rSPR}:\TT_n \times \TT_n \to \mathbb{N}$, such that for any pair of trees $t, t^{\prime} \in \TT_n$, $d_{rSPR}(t, t^{\prime}) \geq 0$ represents the minimum number of rSPR operations required to transform $t$ into $t^{\prime}$.
We also define, for a tree $t \in \TT_n$, the set $N(t):=\{ z \in \TT_n : d_{rSPR}(t,z) = 1 \}$, i.e. the set of neighbours of $t$ in $G_n$.


The diameter of the rSPR-graph is then defined as the maximum distance in $d_{rSPR}$ between two trees.
This diameter is $\mathcal{O}(n)$, as proved in Proposition 5.1 of~\cite{Song2003}.
From this proposition, it can be inferred that the rSPR-graph is connected.
Furthermore, the neighbourhood size of the rSPR-graph is $\mathcal{O}(n^2)$, as demonstrated in Corollary 4.2 of~\cite{Song2003}.
This result also yields the exact formula for the maximum and minimum number of neighbours of a tree in $\TT_n$.
The trees that have the maximum and minimum number of neighbours are, respectively, called the balanced tree and the ladder tree, as shown in Figure~\ref{fig:bal_lad}.
It is also noteworthy that computing the distance between two trees in this context is NP-hard, a complexity issue addressed in~\cite{bordewich2005}.

\begin{figure}[h]
    \centering

\tikzset{every picture/.style={line width=0.75pt}} 

\begin{tikzpicture}[x=0.75pt,y=0.75pt,yscale=-1,xscale=1]

\draw    (130.5,21) -- (247.5,118.25) ;
\draw    (130.5,21) -- (40,119.75) ;
\draw    (98,55) -- (158,120.25) ;
\draw    (70,88.5) -- (99.5,119.75) ;
\draw    (211,87.5) -- (187,119.25) ;
\draw   (37.66,122.14) .. controls (37.63,120.82) and (38.68,119.75) .. (40,119.75) .. controls (41.32,119.75) and (42.41,120.82) .. (42.44,122.14) .. controls (42.47,123.46) and (41.42,124.54) .. (40.1,124.54) .. controls (38.77,124.54) and (37.68,123.46) .. (37.66,122.14) -- cycle ;
\draw   (97.16,122.14) .. controls (97.13,120.82) and (98.18,119.75) .. (99.5,119.75) .. controls (100.82,119.75) and (101.91,120.82) .. (101.94,122.14) .. controls (101.97,123.46) and (100.92,124.54) .. (99.6,124.54) .. controls (98.27,124.54) and (97.18,123.46) .. (97.16,122.14) -- cycle ;
\draw   (155.66,122.64) .. controls (155.63,121.32) and (156.68,120.25) .. (158,120.25) .. controls (159.32,120.25) and (160.41,121.32) .. (160.44,122.64) .. controls (160.47,123.96) and (159.42,125.04) .. (158.1,125.04) .. controls (156.77,125.04) and (155.68,123.96) .. (155.66,122.64) -- cycle ;
\draw   (184.66,121.64) .. controls (184.63,120.32) and (185.68,119.25) .. (187,119.25) .. controls (188.32,119.25) and (189.41,120.32) .. (189.44,121.64) .. controls (189.47,122.96) and (188.42,124.04) .. (187.1,124.04) .. controls (185.77,124.04) and (184.68,122.96) .. (184.66,121.64) -- cycle ;
\draw   (245.16,120.64) .. controls (245.13,119.32) and (246.18,118.25) .. (247.5,118.25) .. controls (248.82,118.25) and (249.91,119.32) .. (249.94,120.64) .. controls (249.97,121.96) and (248.92,123.04) .. (247.6,123.04) .. controls (246.27,123.04) and (245.18,121.96) .. (245.16,120.64) -- cycle ;
\draw    (372,19.5) -- (489,116.75) ;
\draw    (372,19.5) -- (281.5,118.25) ;
\draw    (356,38) -- (448,116.75) ;
\draw    (311.5,87) -- (341,118.25) ;
\draw    (334.5,60) -- (398,117.75) ;
\draw   (279.16,120.64) .. controls (279.13,119.32) and (280.18,118.25) .. (281.5,118.25) .. controls (282.82,118.25) and (283.91,119.32) .. (283.94,120.64) .. controls (283.97,121.96) and (282.92,123.04) .. (281.6,123.04) .. controls (280.27,123.04) and (279.18,121.96) .. (279.16,120.64) -- cycle ;
\draw   (338.66,120.64) .. controls (338.63,119.32) and (339.68,118.25) .. (341,118.25) .. controls (342.32,118.25) and (343.41,119.32) .. (343.44,120.64) .. controls (343.47,121.96) and (342.42,123.04) .. (341.1,123.04) .. controls (339.77,123.04) and (338.68,121.96) .. (338.66,120.64) -- cycle ;
\draw   (396.16,119.64) .. controls (396.13,118.32) and (397.18,117.25) .. (398.5,117.25) .. controls (399.82,117.25) and (400.91,118.32) .. (400.94,119.64) .. controls (400.97,120.96) and (399.92,122.04) .. (398.6,122.04) .. controls (397.27,122.04) and (396.18,120.96) .. (396.16,119.64) -- cycle ;
\draw   (486.66,119.14) .. controls (486.63,117.82) and (487.68,116.75) .. (489,116.75) .. controls (490.32,116.75) and (491.41,117.82) .. (491.44,119.14) .. controls (491.47,120.46) and (490.42,121.54) .. (489.1,121.54) .. controls (487.77,121.54) and (486.68,120.46) .. (486.66,119.14) -- cycle ;
\draw   (446.16,119.64) .. controls (446.13,118.32) and (447.18,117.25) .. (448.5,117.25) .. controls (449.82,117.25) and (450.91,118.32) .. (450.94,119.64) .. controls (450.97,120.96) and (449.92,122.04) .. (448.6,122.04) .. controls (447.27,122.04) and (446.18,120.96) .. (446.16,119.64) -- cycle ;

\draw (36.5,127.9) node [anchor=north west][inner sep=0.75pt]  [font=\tiny]  {$1$};
\draw (96,128.4) node [anchor=north west][inner sep=0.75pt]  [font=\tiny]  {$2$};
\draw (154,128.4) node [anchor=north west][inner sep=0.75pt]  [font=\tiny]  {$3$};
\draw (183.5,126.9) node [anchor=north west][inner sep=0.75pt]  [font=\tiny]  {$5$};
\draw (244,124.9) node [anchor=north west][inner sep=0.75pt]  [font=\tiny]  {$4$};
\draw (278,126.4) node [anchor=north west][inner sep=0.75pt]  [font=\tiny]  {$1$};
\draw (337.5,126.9) node [anchor=north west][inner sep=0.75pt]  [font=\tiny]  {$2$};
\draw (445.66,125.04) node [anchor=north west][inner sep=0.75pt]  [font=\tiny]  {$3$};
\draw (395,124.9) node [anchor=north west][inner sep=0.75pt]  [font=\tiny]  {$5$};
\draw (485.5,123.4) node [anchor=north west][inner sep=0.75pt]  [font=\tiny]  {$4$};

\end{tikzpicture}
    
    \caption{\textbf{A balanced (left) and a ladder (right) trees.}
    These represent the trees with the largest and smallest neighbourhoods in the rSPR graph, respectively.}
    \label{fig:bal_lad}
\end{figure}
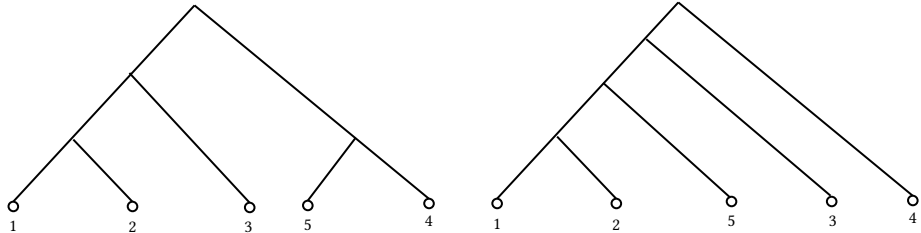

\medskip

\paragraph{Lumpability: }
The preceding discussion makes clear that the space of trees is very high dimensional as $n$ increases, presenting considerable challenges when employing Markov Chain Monte Carlo (MCMC) methods to sample from probability distributions with support on  $\TT_n$.
One potential approach to address these challenges involves identifying lower-dimensional projections that either precisely or approximately retain the Markov property.
To achieve this, we will leverage the concept of \textit{lumpability} as outlined in~\cite{buchholz1994}.
One important application of lumpability was demonstrated in~\cite{lee2003}, where the authors employed this technique to develop a faster two-stage algorithm for computing a PageRank vector.

In a different context, as presented in~\cite{Simper2022}, the application of lumpability enabled the authors to demonstrate that an adjacent swap chain on labelled and unlabelled ranked tree shapes has a mixing time of at least $\mathcal{O}(n^3)$ and at most $\mathcal{O}(n^4)$, where $n$ represents the number of labels.
\cite{Simper2022} -- Lemma 4 therein-- prove that the spectral gap of the lumped chain is less than or equal to that of the original chain.
This finding is directly connected to comparing the relaxation times of the original and the lumped chains.

Before introducing the concept of lumpability, we will need the following definition.
Consider a Markov chain $(X_k)_{k \ge 0}$ on a finite state-space $\mathcal{S}$ and a partition of this space $\bar{S}$. Now, instead of looking at each iteration the process $(X_k)_{k \ge 0}$ is on the state-space $\mathcal{S}$, we observe in which partition the process is, thus having a \textit{projected process} $(Y_k)_{k \ge 0}$ on the state space $\bar{S}$.

\begin{definition}[\textbf{Lumpability}]
\label{def:lump}
Let $(X_k)_{k\geq 0}$  be Markov chain on a finite state-space $\mathcal{S} = \{f_1, f_2, \ldots, f_r\}$ and transition probabilities matrix $P$.
We say $(X_k)_{k\geq 0}$ is~\textbf{lumpable} with respect to a partition of $\bar{S} = \{E_1, E_2, \ldots, E_v\}$ of the state space if the projected chain $(Y_k)_{k \ge 0}$ on $\bar{S}$ is also a Markov chain. 
\end{definition}

Let us assume the same notation from Definition~\ref{def:lump}.
If $(X_k)_{k \ge 0}$ is lumpable with respect to the partition $\bar{S}$, by Definition~\ref{def:lump}, for any $x,y \in E_i$ we have 
\begin{equation*}
 \sum_{z \in E_j} p(x, z) = \sum_{z \in E_j} p(y, z)\,.
\end{equation*}
This is the so-called row sum criterion.


Now it is possible to derive an interesting result that relates the stationary distribution of the original process to the projected one (see Algorithm 3.1 in~\cite{Sumita1989}).

\begin{proposition}\label{prop: sta_lump}
Let $(X_k)_{k \ge 0}$ be an aperiodic and irreducible Markov Chain on finite state space $\mathcal{S}$ with stationary distribution $\pi_X$ and transition matrix $P$. Suppose that $(X_k)_{k \ge 0}$ is lumpable with respect to a partition $\bar{S}:=\{E_1, E_2, \dots, E_v\}$. The projected process $(Y_k)_{k \ge 0}$ on $\bar{S}$ has stationary distribution $\pi_Y$ and transition matrix $\Tilde{P}$ where $\pi_Y (E_i) = \sum_{x \in E_i} \pi_X (x)$ for all $i \in \{1,2, \dots, v\}$.   
\end{proposition}
\begin{proof}
First one can note that $\sum_{i = 1}^v \pi_Y(E_i) = \sum_{x \in S} \pi_X (x)  = 1$, since $\bar{S}$ is a partition of $\mathcal{S}$.
 

Let us denote $|E_l| = m_l$ for all $l \in \{1,2, \dots, v \}$ and for any $E_i \in \bar{S}$, $E_i = \{z_1^i, z_2^i, \dots, z_{m_i}^i \}$ to conclude
\begin{equation}\label{eq: sta_lump_1}
\begin{split}
\pi_Y (E_i) \Tilde{p}(E_i, E_j) & = \sum_{l = 1}^{m_i} \pi_X (z_l^i) \Tilde{p}(E_i, E_j) = \sum_{l = 1}^{m_i} \sum_{h = 1}^{m_j} \pi_X (z_l^i) p(z_l^i, z_h^j) 
\\
& = \sum_{h = 1}^{m_j} \sum_{l = 1}^{m_i} \pi_X (z_h^j) p(z_h^j, z_l^i) = \sum_{h = 1}^{m_j} \pi_X (z_h^j) \Tilde{p}(E_j, E_i) = \pi_Y(E_j)\Tilde{p}(E_j, E_i) \,. 
\end{split}
\end{equation}
In the second equality in~\eqref{eq: sta_lump_1} we used lumpability property and Lemma 2.5 in~\cite{Levin2017}. We used the detailed balanced equation for the original process in the third equality in~\eqref{eq: sta_lump_1}.
Since $\sum_{i = 1}^K \pi_Y(Q_i) = 1$ and by~\eqref{eq: sta_lump_1} we see that the detailed balanced equation holds, hence we finish the proof. 
\end{proof}



\paragraph{$\varepsilon$-Lumpability for Markov chains:} Finding a suitable partition of the state space where the lumpability condition holds exactly is not always possible.
In such cases, a partition can be found where a certain \textit{lumping error} emerges, resulting in an "almost" Markovian process.
We define this concept as $\varepsilon$-\textit{lumpability}, with $\varepsilon$ being a positive constant that quantifies this error~\citep{buchholz1994}.
The idea of $\varepsilon$-lumpability is based on the principle that a Markov chain can be altered by relatively small perturbations of the transition probabilities, leading to a new resulting Markov chain that is lumpable. This approach allows for a controlled approximation of the original chain, facilitating analysis and computation.
This technique serves as a method to reduce the dimensionality of the problem and sometimes to obtain bounds for the stationary distribution of a Markov Chain, as discussed in~\cite{franceschinis1994} and~\cite{dayar1997}. 

\begin{definition}[\textbf{Lumping error and $\varepsilon$-lumpability} ]
Consider again a partition $ \bar{S} = \{E_1, \ldots,  E_v\}$ of $\mathcal{S}$.
For $x, y \in E_i$, define the \textbf{lumping error} as
\begin{equation}
\label{eq:lumping_error}
    R_{i,j}(x, y) = \sum_{z \in E_j} p(x, z) - \sum_{z \in E_j} p(y, z), \text{ for } i,j \in \{1,2, \dots, v\}\,.
\end{equation}
When $|R_{i,j}(x, y)| \leq \varepsilon$ for every pair $x, y$ and every $i, j \in \{1,2, \dots, v\}$, we say the Markov chain is $\varepsilon$-lumpable with respect to $\bar{S}$~\citep{Bittracher2021}.
\end{definition}

\subsubsection{Tree-valued Markov chains}
\label{sec:treeMC}

We now move on to describe two tree-valued Markov processes which we shall study in detail in this paper.
We start with a process whose stationary distribution is uniform on $\TT_n$, which is a common choice of prior distribution in Bayesian analysis~\citep{Alfaro2006}.
We then study a random walk with a non-uniform stationary distribution. 
Throughout we will denote a tree-valued process on $G_n$ by $(X_k)_{k \geq 0}$, meaning $X_k \in \TT_n$ for all $k \geq 0$.
Moreover, whenever there is no confusion, other operations such as conditional probabilities $\pr(X_{k + 1} = x \mid X_{k} = y)$ will also be implicitly defined on $\TT_n$ without explicit indexing on the number of leaves $n$.

\paragraph{SPR Metropolis-Hastings random walk: } We start with the so-called Metropolis-Hastings random walk, studied by, among others, \cite{Whidden2017}.
In this Markov chain, transitions between states (trees) occur proportional to the ratio of their SPR neighbourhoods.
Concretely, consider a process with the following transition probability for $x, y \in \TT_n$ 
\begin{equation}
\label{eq:MH_SPR_RW}
    p_{\text{MH}}(x, y) := \pr(X_{k+1} = y \mid X_{k} = x) = \begin{cases}
    \frac{1}{|N(x)|}\min\left\{1, \frac{|N(x)|}{|N(y)|} \right\}, y \in N(x),
    \\
    1 - \sum_{z \in N(x)} \frac{1}{|N(x)|}\min\left\{1, \frac{|N(x)|}{|N(z)|} \right\}, y = x
    \\
    0, y \notin N(x).
    \end{cases}
\end{equation}
This leads to a uniform stationary distribution, i.e., $\pi_{\text{MH}}(x) = 1/|\TT_n|$ for all $x \in \TT_n$ (see Example 3.3 in~\cite{Levin2017}).


\paragraph{SPR lazy random walk: } Now we explore a process whose invariant distribution is \textbf{not} the uniform on $\boldsymbol{T}_n$ the \textit{lazy} SPR-random walk for $\rho \in (0, 1)$, its transition probability is define as
\begin{equation}
\label{eq:simple_SPR_RW}
    p_{\text{RW}}(x, y) := \pr(X_{k+1} = y \mid X_{k} = x) = \begin{cases}
    (1-\rho)\frac{1}{|N(x)|}, \text{ if }y \in N(x),
    \\
    \rho, \text{ if } x=y,
    \\
    0, \text{otherwise}.
    \end{cases}
\end{equation}
It is straightforward to see that in this case the stationary distribution is given by
\begin{equation*}
    \pi_{\text{RW}}(x) = \frac{|N(x)|}{2|E_n|},
\end{equation*}
for all $x \in \TT_n$ (see Example 1.12 in~\cite{Levin2017}).

In this paper, we will refer to the two processes of interest simply as the Metropolis-Hastings and the lazy random walks, omitting the prefix 'SPR' for simplicity. These are the only processes considered throughout our analysis; additional specifications will be provided whenever necessary.

\subsubsection{The neighbourhood of a tree: a few combinatorial results}\label{sec:prelim2}

In this paper, we are concerned with the question of whether a Markov process $(X_k)_{k\geq 0}$ taking values in $\boldsymbol{T}_n$  retains its Markov property when projected onto the tree shape space or clade space, i.e., we would like to investigate whether the projected process $(Y^c_k)_{k\geq 0}$ on the partition $\bar{S}:=\{S_0(c), S_1(c)\}$ is Markovian.
Let $x$ be a tree in $\TT_n$, and let $w$ be a degree-3 vertex in $x$. We recall that $\gamma_x (w)$ is the number of degree-3 vertices on the path from $w$ to the root of $x$. According to Proposition 3.2 in~\cite{Song2003}, for any tree $x \in \TT_n$ with $n \ge 3$, it is possible to determine the number of neighbours of $x$ in the SPR-graph, denoted by $|N(x)|$.
\begin{proposition}[\textbf{SPR neighbourhood~\citep{Song2003}}]\label{prop_song_nei}
Let $n \ge 3$ and $x \in \TT_n$. Let $\{w_1, w_2, \dots, w_{n-2}\}$ be the set of degree-3 vertices in $x$. Then $|N(x)|$ is given by 
\begin{equation}\label{eq:song_nei}
    |N(x)| = 2(n-2)(2n-5) - 2\sum_{i=1}^{n-2} \gamma_x (w_i) \,.
\end{equation}
\end{proposition}

From Proposition~\ref{prop_song_nei}, we can derive the exact formula to determine the tree configurations with the maximum and minimum numbers of neighbours, which are respectively termed the balanced tree and the ladder tree.
These findings are detailed in Corollary 4.2 of \cite{Song2003}.

\begin{corollary}[\textbf{SPR neighbourhood maximum and minimum~\citep{Song2003}}]\label{max_min_nei}
Let $n \ge 4$, then we have the following
\begin{equation*}
\begin{split}
 & \max_{x \in \TT_n} |N(x)| = 4(n - 2)^2 - 2 \sum_{j = 1}^{n-2} \lfloor \log_2 (j+1) \rfloor \quad \text{and}
 \\
& \min_{x \in \TT_n} |N(x)| = 3n^2 - 13n + 14 \,.
\end{split}    
\end{equation*}
\end{corollary}

The results presented in Proposition~\ref{prop_song_nei} and Corollary~\ref{max_min_nei} are crucial for our subsequent analysis, where we aim to calculate, within the neighbourhood of a given tree, the number of trees that possess a specific clade.


In this section, we introduce a collection of auxiliary results that provide insights into the neighbourhood of a tree $x \in S_1(c)$.
With that aim, we now introduce several definitions.
For any tree $x \in \TT_n$, let $\phi_c^x$ denote the subtree generated by the degree-3 vertices in $x$, which serves as the least common ancestor only of the elements of $c$.
We define a map $f_c: S_1(c) \to \boldsymbol{T}_k$, where $k := n - |c| + 1$, that transforms the subtree corresponding to the clade $c$ into a leaf $l$, resulting in a label set in $\TT_k$ of $\{1,2, \dots, n\}/c \cup \{l\}$. Refer to Figure~\ref{fig:A_1^x,c} for an illustrative depiction. Furthermore, we denote by $g_{c}^x: \boldsymbol{T}_k \to S_1(c)$ a map, that replaces the leaf $l$ with $\phi_c^x$, effectively reversing the operation performed by $f_c$.

\begin{figure}[h]
    \centering
    \tikzset{every picture/.style={line width=0.68pt}} 

\begin{tikzpicture}[x=0.68pt,y=0.68pt,yscale=-1,xscale=1]

\draw    (268.84,28.75) -- (338.47,99.45) ;
\draw    (268.84,28.75) -- (205.16,97.74) ;
\draw    (242.52,57.71) -- (263.75,74.75) ;
\draw    (221.29,79.86) -- (240.82,96.89) ;
\draw   (205.16,97.74) -- (218.32,117.33) -- (192,117.33) -- cycle ;
\draw   (240.82,96.89) -- (253.98,116.48) -- (227.66,116.48) -- cycle ;
\draw   (263.75,74.75) -- (276.91,94.34) -- (250.59,94.34) -- cycle ;
\draw   (338.47,99.45) -- (351.63,119.04) -- (325.31,119.04) -- cycle ;
\draw [color={black}  ,draw opacity=1 ] [dash pattern={on 4.5pt off 4.5pt}]  (292.19,119.25) .. controls (301.96,87.31) and (323.19,57.07) .. (358,100.51) ;
\draw    (153.84,140.75) -- (223.47,211.45) ;
\draw    (153.84,140.75) -- (90.16,209.74) ;
\draw    (127.52,169.71) -- (148.75,186.75) ;
\draw    (106.29,191.86) -- (125.82,208.89) ;
\draw   (90.16,209.74) -- (103.32,229.33) -- (77,229.33) -- cycle ;
\draw   (125.82,208.89) -- (138.98,228.48) -- (112.66,228.48) -- cycle ;
\draw   (148.75,186.75) -- (161.91,206.34) -- (135.59,206.34) -- cycle ;
\draw   (218.85,216.07) .. controls (218.85,213.52) and (220.92,211.45) .. (223.47,211.45) .. controls (226.02,211.45) and (228.1,213.52) .. (228.1,216.07) .. controls (228.1,218.63) and (226.02,220.7) .. (223.47,220.7) .. controls (220.92,220.7) and (218.85,218.63) .. (218.85,216.07) -- cycle ;
\draw   (389.47,175.95) -- (402.63,195.54) -- (376.31,195.54) -- cycle ;
\draw    (389.47,165.2) -- (389.47,175.95) ;
\draw   (194.15,155.14) -- (198.45,141.69) -- (200.87,144.52) -- (214.57,132.82) -- (219.4,138.47) -- (205.7,150.17) -- (208.11,153) -- cycle ;
\draw   (364.29,159.1) -- (350.2,158.19) -- (352.35,155.16) -- (337.68,144.71) -- (342,138.65) -- (356.67,149.11) -- (358.83,146.08) -- cycle ;
\draw   (392.47,162.2) .. controls (392.47,160.54) and (391.13,159.2) .. (389.47,159.2) .. controls (387.81,159.2) and (386.47,160.54) .. (386.47,162.2) .. controls (386.47,163.85) and (387.81,165.2) .. (389.47,165.2) .. controls (391.13,165.2) and (392.47,163.85) .. (392.47,162.2) -- cycle ;

\draw (333.74,105.07) node [anchor=north west][inner sep=0.75pt]  [font=\scriptsize]  {$c$};
\draw (218.74,224.57) node [anchor=north west][inner sep=0.75pt]  [font=\scriptsize]  {$l$};
\draw (385.74,181.57) node [anchor=north west][inner sep=0.75pt]  [font=\scriptsize]  {$c$};
\draw (253,13.9) node [anchor=north west][inner sep=0.75pt]  [font=\footnotesize]  {$x\ \in \ \boldsymbol{T}_{n}$};
\draw (115,123.4) node [anchor=north west][inner sep=0.75pt]  [font=\footnotesize]  {$x^{\prime } \ \in \ \boldsymbol{T}_{n-|c|+1}$};
\draw (371.5,136.9) node [anchor=north west][inner sep=0.75pt]  [font=\footnotesize]  {$\phi _{c}^{x} \ \in \ \boldsymbol{T}_{|c|}$};

\end{tikzpicture}

\bigskip
    \caption{\textbf{A description of the map $f_c$ acting on a tree $x \in \TT_n$.}
    It is possible to see in left the tree $x' \in \TT_{n -|c| +1}$ and the tree with the clade $c$, $\phi_c^x \in \TT_{|c|}$.}
    \label{fig:A_1^x,c}

\end{figure}
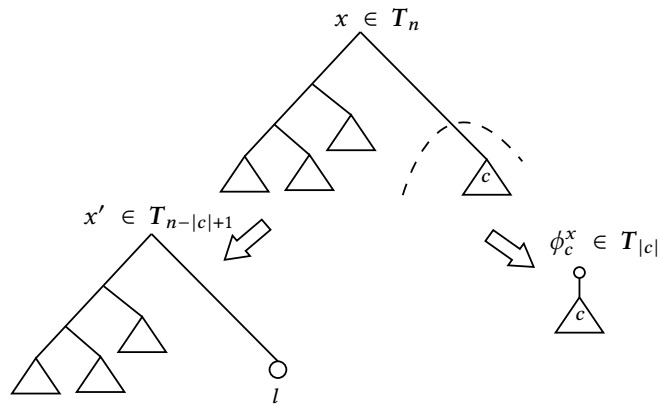

Let $x$ be a tree in $\TT_n$ such that $x \in S_1(c)$, and define $x' = f_c(x) \in \TT_k$. We then introduce $h_{x'}: \TT_{|c|} \to \TT_n$ as a transformation that eliminates the virtual root from the tree in $\TT_n$ and incorporates this tree into $x'$ by substituting the leaf $l$ (see Figure~\ref{fig:A_1^x,c}).







For a tree $x \in S_1(c)$, we define $A_1^{x,c}$ as the set of trees that contain the clade and are neighbours of $x$, that is, $A_1^{x,c} := S_1(c) \cap N(x)$. Conversely, we define $A_0^{x,c} := N(x) \setminus A_1^{x,c}$ as its complement. 
The next result elucidates the cardinality of the set $A_1^{x,c}$.

\begin{lemma}[\textbf{The neighbourhood of tree containing a clade.}]
\label{lem:card_A_1^xc}
If $x \in S_1 (c)$ then $|A_1^{x,c}| = |N(f_c(x´))| +|N(\phi_c^x)|$.
\end{lemma}

The proof of Lemma~\ref{lem:card_A_1^xc} will be postponed to  Appendix~\ref{sec:app_proofs}.

From Lemma~\ref{lem:card_A_1^xc}, we can determine the number of neighbours of $x \in S_1(c)$ that lack the clade $c$. This calculation requires only the number of degree-3 vertices along the path from the least common ancestor of the clade to the root.
We denote this least common ancestor by $w_c$.
We can now obtain Lemma~\ref{lem:A_0}, which counts the number of neighbours of a tree $x \in S_1(c)$ which do not have $c$ as a constituent clade.
\begin{lemma}[\textbf{Counting the neighbours that do not share a clade $c$ with $x \in S_1(c)$.}]
\label{lem:A_0}
If $n \ge 4$ and $x \in S_1(c)$ then  
\begin{equation*}
  |A_0^{x,c}| = \begin{cases}
   -8|c|^2 + 8|c|n - 2\gamma_{x}(w_c)(|c|-1) + 6|c| -8n-2\,, \text{ for } |c| \ge 3
   \\
    8n-22 - 2\gamma_x(w_c)\,, \text{ for } |c| = 2\,.
  \end{cases}  
\end{equation*}

\end{lemma}

\begin{proof}
We start with $|c| \ge 3$. Let $x^\prime = f_c(x)$. Then by Proposition~\ref{prop_song_nei}, since $x^\prime \in \TT_k$, with $k = n- |c| +1$, we obtain 
\begin{equation*}
|N(x^\prime)| = 2(n-|c|-1)(2(n-|c|+1)-5) - 2\sum_{i=1}^{k-2} \gamma_{x^\prime}(w_i)   
\end{equation*}

Moreover,
\begin{equation}\label{eq:|Nx'|}
\begin{split}
\sum_{i=1}^{k-2} & \gamma_{x^\prime}(w_i) =  \left( \sum_{i=1}^{n-2} \gamma_x (w_i) - \left( \gamma_x (w_c) + \sum_{j=1}^{|c|-2} (\gamma_x (w_c) +1 + \gamma_{\phi_c^x}(w_j)) \right)\right)
\\
& = \left( \sum_{i=1}^{n-2} \gamma_x (w_i) - (|c|-1)\gamma_x (w_c) - |c|+2 - \sum_{j=1}^{|c|-2} \gamma_{\phi_c^x}(w_j)  \right)\,.
\end{split}
\end{equation}
The second equality in~\eqref{eq:|Nx'|} we obtain by the fact that $x^\prime = f_c(x)$, hence the structure of the tree is maintained, except, of course, the clade $c$, which is replaced by the leaf $l$.

Thus, by Proposition~\ref{prop_song_nei}, we obtain
\begin{equation}\label{eq:A_0^x_geral}
\begin{split}
& |N(x)| - |N(x^\prime)| - |N(\phi_c^x)| = 
\\
& = 2(n - 2)(2n -5) -2\sum_{i=1}^{n-2} \gamma_x (w_i) - |N(x^\prime)| - 2(c-2)(2c-5) + 2\sum_{j=1}^{|c|-2} \gamma_{\phi_c^x} (w_j) 
\\
& = -8|c|^2 + 8|c|n - 2\gamma_{x}(w_c)(|c|-1) + 6|c| -8n-2\,.
\end{split}
\end{equation}
The last equality in~\eqref{eq:A_0^x_geral} is true by the fact that $|N(x)| = |A_0^{x,c}| + |A_1^{x,c}|$, and by Lemma~\ref{lem:card_A_1^xc}. Hence we finish the proof for $|c| \ge 3$.

Now we will prove the case $|c| = 2$. First observe that with $|c| = 2$ we have $|N(\phi_c^x)| = 0$. By Proposition~\ref{prop_song_nei}, since $x^\prime \in \TT_k$, with $k = n-1$, and using the same techniques employed in equation~\eqref{eq:|Nx'|}, we obtain
\begin{equation}\label{eq:Nx'_2}
\begin{split}
|N(x^\prime)| & = 2(n-3)(2(n-1)-5) - 2\sum_{i=1}^{n-3} \gamma_{x^\prime}(w_i)
\\
& = 2(n-3)(2(n-1)-5) - 2 \left( \sum_{i=1}^{n-2} \gamma_x (w_i) - \gamma_x (w_c) \right)
\end{split} 
\end{equation}

Hence, by Proposition~\ref{prop_song_nei}  and equation~\eqref{eq:Nx'_2} we find
\begin{equation*}
\begin{split}
|A_0^{x,c}| & = |N(x)|-|N(x^\prime)|
\\
& = 2(n - 2)(2n -5) -2\sum_{i=1}^{n-2} \gamma_x (w_i) - 2(n-3)(2(n-1)-5) + 2 \left( \sum_{i=1}^{n-2} \gamma_x (w_i) - \gamma_x (w_c) \right)
\\
& = 8n-22 - 2\gamma_x(w_c) \,,
\end{split}
\end{equation*}
concluding the proof.
\end{proof}

In our subsequent result, we identify the tree $x$ in $\TT_n$ that exhibits the maximum ratio of the number of neighbours that have the clade $c$ to the total number of neighbours of $x$. Furthermore, we compute this ratio (refer to Figure~\ref{fig:max_A_0}).

\begin{figure}[h]
    \centering
    \tikzset{every picture/.style={line width=0.75pt}} 

\begin{tikzpicture}[x=0.75pt,y=0.75pt,yscale=-1.5,xscale=1.5]

\draw    (241,41) -- (289.5,90.25) ;
\draw   (289.5,90.25) -- (302.63,111) -- (276.38,111) -- cycle ;
\draw    (241,41) -- (152.5,109.75) ;
\draw    (209.5,65.5) -- (252,110.25) ;
\draw    (196.75,75.38) -- (230.5,111.75) ;
\draw    (168.75,98.88) -- (180.5,111.25) ;
\draw  [dash pattern={on 0.84pt off 2.51pt}]  (153.5,120.25) -- (254,120.25) ;
\draw [shift={(256,120.25)}, rotate = 180] [color={rgb, 255:red, 0; green, 0; blue, 0 }  ][line width=0.75]    (10.93,-3.29) .. controls (6.95,-1.4) and (3.31,-0.3) .. (0,0) .. controls (3.31,0.3) and (6.95,1.4) .. (10.93,3.29)   ;
\draw [shift={(151.5,120.25)}, rotate = 0] [color={rgb, 255:red, 0; green, 0; blue, 0 }  ][line width=0.75]    (10.93,-3.29) .. controls (6.95,-1.4) and (3.31,-0.3) .. (0,0) .. controls (3.31,0.3) and (6.95,1.4) .. (10.93,3.29)   ;
\draw [line width=0.75]  [dash pattern={on 0.84pt off 2.51pt}]  (199,94) -- (189,99.25) ;

\draw (284.88,98.65) node [anchor=north west][inner sep=0.75pt]  [font=\small]  {$\phi _{c}^{x}$};
\draw (199,124.9) node [anchor=north west][inner sep=0.75pt]  [font=\scriptsize]  {$n-|c|\ \text{leaves}$};

\end{tikzpicture}

\caption{\textbf{A tree for which the maximum number of neighbours share a clade $c$}.
Observe that if we denote this tree by $x \in \TT_n$, then the $f_c(x)$ is a ladder tree in $\TT_{n-|c|+1}$.}
    \label{fig:max_A_0}
\end{figure}
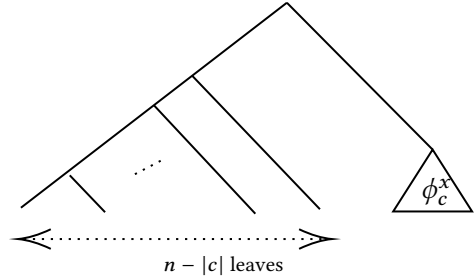

\begin{lemma}[\textbf{The tree which maximises the proportion of neighbours with a clade.}]
\label{lem:max_A_N}
Let $x \in S_1(c)$ be such that the $\phi_c^x$ is a ladder tree and a direct descendent of the root and the subtree generated by the other direct descendant from the root is also a ladder tree. Then for all $z \in S_1(c)$ we have
\begin{equation*}
\frac{|A_0^{x,c}|}{|N(x)|}  \ge \frac{|A_0^{z,c}|}{|N(z)|} \,, \quad \text{and} 
\end{equation*}
\begin{equation*}
\frac{|A_0^{x,c}|}{|N(x)|}  =\begin{cases}
  \dfrac{-8|c|^2 + 8|c|n + 6|c| -8n-2}{3n^2 -2|c|^2 +2|c|n-15n+16} \quad \text{for } |c| \ge 3 
  \vspace{10pt}\\
  \dfrac{8n-22}{3n^2 -11n +8} \quad \text{for } |c|=2\,.
  \end{cases}
\end{equation*}
\end{lemma}

\begin{proof}
We begin our proof for $|c| \ge 3$. By Lemma~\ref{lem:A_0} we have
\begin{equation}\label{eq:A_0^x_max}
|A_0^{x,c}| =  -8|c|^2 + 8|c|n + 6|c| -8n-2 \,,  
\end{equation}
since $\gamma_x (w_c) = 0$. Again, by Lemma~\ref{lem:A_0}, we obtain $|A_0^{z,c}| = -8|c|^2 + 8|c|n + 6|c| -8n-2 - 2\gamma_z(w_c)(|c|-1)$. Hence, $|A_0^{z,c}|$ can be rewritten as $|A_0^{z,c}| = |A_0^{x,c}| - 2\gamma_z(w_c)(|c|-1)$ for all $z \in S_1(c)$.

Now for all $z \in S_1(c)$, we set $z^\prime=f_c(z)$, and by Lemma~\ref{lem:card_A_1^xc}, we find
\begin{equation}\label{eq:A_0^w}
\begin{split}
\frac{|A_0^{z,c}|}{|N(z)|} & = \frac{|A_0^{z,c}|}{|N(z^\prime)|+|N(\phi_c^{z})|+ |A_0^{z,c}|} 
\\
& = \frac{|A_0^{x,c}| - 2\gamma_z(w_c)(|c|-1)}{|N(z^\prime)|+|N(\phi_c^{z})|+ |A_0^{x,c}| - 2\gamma_z(w_c)(|c|-1)} \,.
\end{split}
\end{equation}

It is important to notice that, for all $z \in S_1(c)$, we have $|N(x^\prime)| + |N(\phi_c^{x})| \leq |N(z^\prime)|+|N(\phi_c^{z})|$, by Corollary~\ref{max_min_nei}, since $x^\prime$ and $\phi_c^x$ are ladder trees. Then, we obtain
\begin{equation}\label{eq:A_0/Nx}
\begin{split}
\frac{|A_0^{x,c}|}{|N(x)|} & = \frac{|A_0^{x,c}|}{|N(x^\prime)|+|N(\phi_c^{x})|+ |A_0^{x,c}|} 
\\
& \ge  \frac{|A_0^{x,c}|}{|N(z^\prime)|+|N(\phi_c^{z})|+ |A_0^{x,c}|}
\\
& \geq \frac{|A_0^{x,c}| - 2\gamma_z(w_c)(|c|-1)}{|N(z^\prime)|+|N(\phi_c^{z})|+ |A_0^{x,c}| - 2\gamma_z(w_c)(|c|-1)} = \frac{|A_0^{z,c}|}{|N(z)|}\,.  
\end{split}    
\end{equation}
The second inequality in~\eqref{eq:A_0/Nx} holds because $\gamma_z(w) \in \{0, 1, \dots, n-3 \}$, for all $z \in \TT_n$, and by the fact that $a/b \ge (a-\delta)/(b-\delta)$ for all $a, b > \delta > 0$. The last equality in~\eqref{eq:A_0/Nx} we have from equation~\eqref{eq:A_0^w}.

Now by Corollary~\ref{max_min_nei}, we find $|N(x^\prime)| = 3(n-|c|+1)^2 - 13(n-|c|+1) + 14$ and $|N(\phi_c^x)| = 3|c|^2 -13|c| +14$. Hence, by~\eqref{eq:A_0^x_max}, we conclude
\begin{equation*}
\frac{|A_0^{x,c}|}{|N(x)|} = \frac{-8|c|^2 + 8|c|n + 6|c| -8n-2}{3n^2 -2|c|^2 +2|c|n-15n+16} \,. 
\end{equation*}
Thus, we finish the proof for $|c| \ge 3$.

For $|c|=2$ we can do the same techniques and computations we did for $|c| \ge 3$. By Lemma~\ref{lem:A_0}, when $|c| = 2$, we have $|A_0^{x,c}| = 8n-22$. Since $|N(\phi_c^x)| = 0$ for $|c| = 2$, we find
\begin{equation*}
\frac{|A_0^{x,c}|}{|N(x)|} = \frac{8n-22}{3n^2-11n+8} \,,    
\end{equation*}
which finishes the proof.
\end{proof}

For a tree $x \in S_0(c)$, we define the set $B_1^{x,c}:= N(x) \cap S_1(c)$, i.e., the of neighbours the tree $x$ that have the clade $c$.
Subsequently, we aim to determine the possible cardinalities of $B_1^{x,c}$, which will vary according to the tree's shape.

For these computations, preliminary definitions are required.
Given a tree $y$ in $S_0(c)$, we introduce $I$ as a set of labels for which there exists a subtree $\phi_I$ in $y$, comprising exclusively the labels from $I$.
Furthermore, the union $c \cup I$ forms a subtree. 
For a better understanding see an example in Figure~\ref{fig:intruso}. 
We can now state Lemma~\ref{lem:B_1^xc}, which determines $|B_1^{x,c}|$.

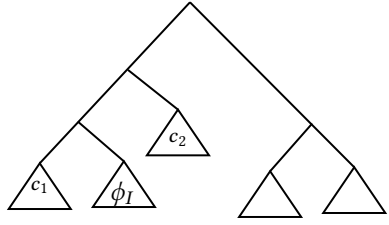
\begin{figure}[h]
    \centering
    \tikzset{every picture/.style={line width=0.75pt}} 

\begin{tikzpicture}[x=0.75pt,y=0.75pt,yscale=-1,xscale=1]

\draw    (264,47) -- (346,130) ;
\draw    (264,47) -- (189,128) ;
\draw    (233,81) -- (258,101) ;
\draw    (208,107) -- (231,127) ;
\draw   (189,128) -- (204.5,151) -- (173.5,151) -- cycle ;
\draw   (231,127) -- (246.5,150) -- (215.5,150) -- cycle ;
\draw   (258,101) -- (273.5,124) -- (242.5,124) -- cycle ;
\draw    (325,108) -- (304,132) ;
\draw   (304,132) -- (319.5,155) -- (288.5,155) -- cycle ;
\draw   (346,130) -- (361.5,153) -- (330.5,153) -- cycle ;

\draw (183,135.4) node [anchor=north west][inner sep=0.75pt]  [font=\scriptsize]  {$c_{1}$};
\draw (252,111.4) node [anchor=north west][inner sep=0.75pt]  [font=\scriptsize]  {$c_{2}$};
\draw (223,135.4) node [anchor=north west][inner sep=0.75pt]  [font=\footnotesize]  {$\phi _{I}$};

\end{tikzpicture}

    \caption{\textbf{Illustration of a clade-breaking subtree.}
    Suppose $c_1$ and $c_2$ are clades such that $c_1 \cup c_2 = c$.
    It is possible to see that $I \cup c$ generates a subtree.}
    \label{fig:intruso}
\end{figure}

\begin{lemma}[\textbf{The number of neighbours of a tree $x \in S_0(c)$ which have $c$}.]
\label{lem:B_1^xc}
Let $x \in S_0(c)$ and $3 \le |c| \le n-3$ then $|B_1^{x,c}| \in \{0, 2(|c| -1), 2(n - |I|)-3 \}$.
\end{lemma}

The proof of Lemma~\ref{lem:B_1^xc} will be postponed to Appendix~\ref{sec:app_proofs}.
Lemma~\ref{lem:B_1^xc} shows that with $3 \le |c| \le n-3$, the cardinality of $B_1^{x,c}$ can assume three different values depending on the distribution of the components of the clade $c$ within the tree $x \in S_0(c)$.
The next result calculates $|B_1^{x,c}|$ when $|c| = 2$.
\begin{corollary}\label{cor:B_1^x,2}
Let $x \in S_0(c)$ and $|c| = 2$. Then $|B_1^{x,c}|$ can be three different values $|B_1^{x,c}| \in \{1, 2, 2(n - |I|) -3 \}$.
\end{corollary}

The proof of Corollary~\ref{cor:B_1^x,2} follows similarly to in the proof of Lemma~\ref{lem:B_1^xc}.
However, in Corollary~\ref{cor:B_1^x,2} it is not possible for $|B_1^{x,c}|$ to be equal 0. If $|c| = 2$ and the tree does not have the clade $c$, it always will have at least 1 neighbour with the clade because it is always possible to regraft one leaf in the other in a tree.

Also, it is important to note that the value of $|B_1^{x,c}|$ can be 1.
This occurs in special cases when, for a clade $c = \{i, j\}$, regrafting $i$ onto the edge of $j$ or vice versa results in the same tree.
Figure~\ref{fig:B_1^x,c=1} illustrates this case.

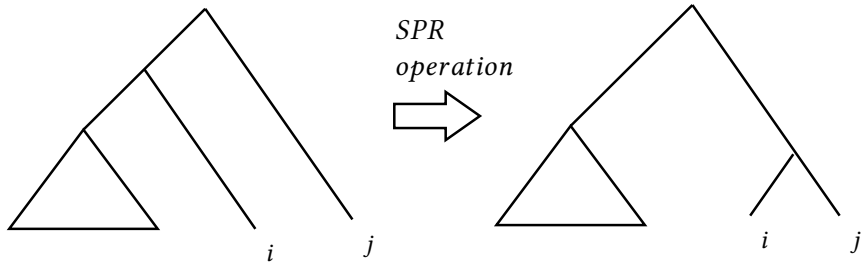
\begin{figure}[h]
    \centering

\tikzset{every picture/.style={line width=0.95pt}} 

\begin{tikzpicture}[x=0.95pt,y=0.95pt,yscale=-1.5,xscale=1.5]

\draw    (98.83,123.33) -- (66.83,155.33) ;
\draw    (98.83,123.33) -- (137.5,179.09) ;
\draw    (82.83,139.33) -- (112,181.59) ;
\draw   (66.83,155.33) -- (86.33,181.59) -- (47.33,181.59) -- cycle ;
\draw    (226.83,122.33) -- (194.83,154.33) ;
\draw    (226.83,122.33) -- (265.5,178.09) ;
\draw    (253.33,161.83) -- (242,178.09) ;
\draw   (194.83,154.33) -- (214.33,180.59) -- (175.33,180.59) -- cycle ;
\draw   (148.5,148.52) -- (162,148.52) -- (162,145.34) -- (171,151.71) -- (162,158.09) -- (162,154.9) -- (148.5,154.9) -- cycle ;

\draw (114,184.99) node [anchor=north west][inner sep=0.75pt]  [font=\footnotesize]  {$i$};
\draw (139.5,182.49) node [anchor=north west][inner sep=0.75pt]  [font=\footnotesize]  {$j$};
\draw (244,181.49) node [anchor=north west][inner sep=0.75pt]  [font=\footnotesize]  {$i$};
\draw (267.5,181.49) node [anchor=north west][inner sep=0.75pt]  [font=\footnotesize]  {$j$};
\draw (144.5,123.74) node [anchor=north west][inner sep=0.75pt]  [font=\small]  {$ \begin{array}{l}
SPR\\
operation
\end{array}$};

\end{tikzpicture}
    
    \caption{\textbf{SPR operation resulting on the same tree}.
    Let $c = \{i, j\}$ be the clade of interest, then the tree on the left side above clearly does not have $c$.
    The triangle in the tree above is any subtree, then if we regraft $i$ in the edge of $j$ or vice-versa it will result in the same tree.}
    \label{fig:B_1^x,c=1}
\end{figure}

Our next auxiliary result (Lemma~\ref{lem:max_B_N}) is quite similar to the one in Lemma~\ref{lem:max_A_N}, however, we are now interested in the case that $|c| = 2$, the tree is in $S_0(c)$ and the ratio will be between the number of neighbours with clade $c$ and the total number of neighbours, see Figure~\ref{fig:max_B_1^x_Nx} for an example of this maximum ratio.

\begin{figure}[h]
    \centering

     \tikzset{every picture/.style={line width=0.95pt}} 

\begin{tikzpicture}[x=0.95pt,y=0.95pt,yscale=-1.2,xscale=1.2]

\draw    (241,41) -- (297.5,111.25) ;
\draw    (241,41) -- (152.5,109.75) ;
\draw    (226,52.5) -- (271,110.25) ;
\draw    (180.75,89.88) -- (197.5,110.75) ;
\draw    (168.75,98.88) -- (180.5,111.25) ;
\draw  [dash pattern={on 0.84pt off 2.51pt}]  (214,117.76) -- (300.5,118.24) ;
\draw [shift={(302.5,118.25)}, rotate = 180.32] [color={rgb, 255:red, 0; green, 0; blue, 0 }  ][line width=0.75]    (10.93,-3.29) .. controls (6.95,-1.4) and (3.31,-0.3) .. (0,0) .. controls (3.31,0.3) and (6.95,1.4) .. (10.93,3.29)   ;
\draw [shift={(212,117.75)}, rotate = 0.32] [color={rgb, 255:red, 0; green, 0; blue, 0 }  ][line width=0.75]    (10.93,-3.29) .. controls (6.95,-1.4) and (3.31,-0.3) .. (0,0) .. controls (3.31,0.3) and (6.95,1.4) .. (10.93,3.29)   ;
\draw [line width=0.75]  [dash pattern={on 0.84pt off 2.51pt}]  (215,86.5) -- (205,91.75) ;
\draw    (216,61.5) -- (255.5,111.25) ;
\draw    (190.5,81.5) -- (214.5,111.25) ;

\draw (239,125.9) node [anchor=north west][inner sep=0.75pt]  [font=\footnotesize]  {$n-3\ \text{leaves}$};
\draw (147,111.4) node [anchor=north west][inner sep=0.75pt]  [font=\footnotesize]  {$c_{1}$};
\draw (178,112.9) node [anchor=north west][inner sep=0.75pt]  [font=\footnotesize]  {$I$};
\draw (195.5,111.9) node [anchor=north west][inner sep=0.75pt]  [font=\footnotesize]  {$c_{2}$};

\end{tikzpicture}
    \caption{\textbf{Illustration of a tree satisfying the conditions of Lemma~\ref{lem:max_B_N}.}
    The tree is clearly in $S_0(c)$, then by Lemma~\ref{lem:B_1^xc} we have $|B_1^{x,c}| = 2n-5$, since $|I| = 1$.}
    \label{fig:max_B_1^x_Nx}
\end{figure}
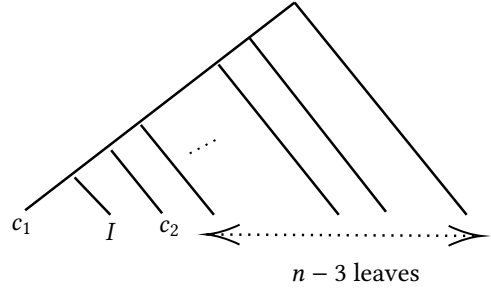

\begin{lemma}
\label{lem:max_B_N}
Let $x \in S_0(c)$ be a ladder tree such that has a subtree with leaves $\{c_1, c_2, I\}$ and $c :=\{c_1, c_2\}$. Then for all $z \in S_0(c)$ we have
\begin{equation*}
\begin{split}
\frac{|B_1^{x,c}|}{|N(x)|} & \ge \frac{|B_1^{z,c}|}{|N(z)|} \, \text{ and}
\\
\frac{|B_1^{x,c}|}{|N(x)|} & = \frac{2n -5}{3n^2 -13n +14} \,.
\end{split}
\end{equation*}
\end{lemma}
\begin{proof}
By Corollary~\ref{cor:B_1^x,2} we have $|B_1^{x,c}| = 2n-5$ and $|B_1^{x,c}| \ge |B_1^{z,c}|$, for every $z \in S_0 (c)$. 

By Corollary~\ref{max_min_nei}, we obtain that $|N(x)| = 3n^2 - 13n + 14$ and $|N(x)| \le |N(y)|$ for every $y \in \TT_n$. Ergo, we have
\begin{equation*}
    \frac{|B_1^{x,c}|}{|N(x)|} = \frac{2n-5}{3n^2 - 13n + 14} \ge \frac{|B_1^{z,c}|}{|N(z)|}\,,
\end{equation*}
for every $z \in S_0 (c)$.
\end{proof}

\section{Lumpability in tree space}
\label{sec:tree_lumpability}






Armed with the preparatory results in Section~\ref{sec:prelim2}, we are now finally ready to study lumpability of Markov processes in the space of rooted binary phylogenetic trees.

Let $\Bar{F} :=\{F_1, F_2, \dots, F_v\}$ be the tree shape partition of $\TT_n$. 
It is important to notice that for all $i \in \{1,2, \dots, v\}$ and $x, y \in F_i$, we have $|N(x)| = |N(y)|$ by Proposition 3.2 in~\cite{Song2003}.
For a given $x \in \TT_n$, we define $F_i^{x}$ as the set of neighbours of 
$x$ that have the tree shape $F_i \in \Bar{F}$.
Formally, $F_i^{x}:=\{ y \in T_n : d_{rSPR}(x, y) = 1 \text{ and } y \in F_i \} = N(x) \cap F_i$.

Suppose we have two trees, $x$ and $y$, with the same shape, that is, $x, y \in F_i$ for some $i \in \{1,2, \dots, v\}$.
Then, performing a SPR operation on $x$ can be mirrored identically on $y$. This is important because trees with the same shape will have the same number of trees with the same shape.
This implies that a subtree can be pruned and regrafted at the same location in both trees -- Figure~\ref{fig:shapeSPR} provides an example for $x, y \in \TT_5$. 
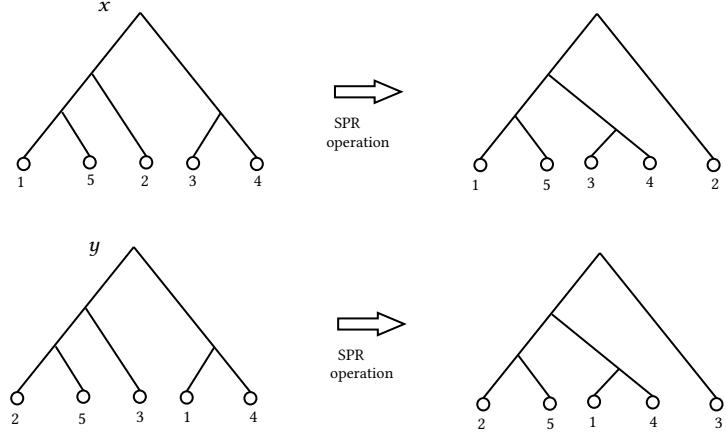
\begin{figure}[h]
    \centering           

\tikzset{every picture/.style={line width=0.75pt}} 

\begin{tikzpicture}[x=0.75pt,y=0.75pt,yscale=-0.8,xscale=1.2]

\draw    (121,19.5) -- (169.5,110.75) ;
\draw    (121,19.5) -- (72.5,110.75) ;
\draw    (88.5,82) -- (100,109.75) ;
\draw    (101,58) -- (123.5,109.75) ;
\draw    (154.5,82.5) -- (143,110.25) ;
\draw   (202,65.75) -- (218.5,65.75) -- (218.5,62.25) -- (229.5,69.25) -- (218.5,76.25) -- (218.5,72.75) -- (202,72.75) -- cycle ;
\draw   (70,114.63) .. controls (70,112.48) and (71.12,110.75) .. (72.5,110.75) .. controls (73.88,110.75) and (75,112.48) .. (75,114.63) .. controls (75,116.77) and (73.88,118.5) .. (72.5,118.5) .. controls (71.12,118.5) and (70,116.77) .. (70,114.63) -- cycle ;
\draw   (97.5,113.63) .. controls (97.5,111.48) and (98.62,109.75) .. (100,109.75) .. controls (101.38,109.75) and (102.5,111.48) .. (102.5,113.63) .. controls (102.5,115.77) and (101.38,117.5) .. (100,117.5) .. controls (98.62,117.5) and (97.5,115.77) .. (97.5,113.63) -- cycle ;
\draw   (121,113.63) .. controls (121,111.48) and (122.12,109.75) .. (123.5,109.75) .. controls (124.88,109.75) and (126,111.48) .. (126,113.63) .. controls (126,115.77) and (124.88,117.5) .. (123.5,117.5) .. controls (122.12,117.5) and (121,115.77) .. (121,113.63) -- cycle ;
\draw   (140.5,114.13) .. controls (140.5,111.98) and (141.62,110.25) .. (143,110.25) .. controls (144.38,110.25) and (145.5,111.98) .. (145.5,114.13) .. controls (145.5,116.27) and (144.38,118) .. (143,118) .. controls (141.62,118) and (140.5,116.27) .. (140.5,114.13) -- cycle ;
\draw   (167,114.63) .. controls (167,112.48) and (168.12,110.75) .. (169.5,110.75) .. controls (170.88,110.75) and (172,112.48) .. (172,114.63) .. controls (172,116.77) and (170.88,118.5) .. (169.5,118.5) .. controls (168.12,118.5) and (167,116.77) .. (167,114.63) -- cycle ;
\draw    (118.33,166.83) -- (166.83,258.08) ;
\draw    (118.33,166.83) -- (69.83,258.08) ;
\draw    (85.83,229.33) -- (97.33,257.08) ;
\draw    (98.33,205.33) -- (120.83,257.08) ;
\draw    (151.83,229.83) -- (140.33,257.58) ;
\draw   (67.33,261.96) .. controls (67.33,259.82) and (68.45,258.08) .. (69.83,258.08) .. controls (71.21,258.08) and (72.33,259.82) .. (72.33,261.96) .. controls (72.33,264.1) and (71.21,265.83) .. (69.83,265.83) .. controls (68.45,265.83) and (67.33,264.1) .. (67.33,261.96) -- cycle ;
\draw   (94.83,260.96) .. controls (94.83,258.82) and (95.95,257.08) .. (97.33,257.08) .. controls (98.71,257.08) and (99.83,258.82) .. (99.83,260.96) .. controls (99.83,263.1) and (98.71,264.83) .. (97.33,264.83) .. controls (95.95,264.83) and (94.83,263.1) .. (94.83,260.96) -- cycle ;
\draw   (118.33,260.96) .. controls (118.33,258.82) and (119.45,257.08) .. (120.83,257.08) .. controls (122.21,257.08) and (123.33,258.82) .. (123.33,260.96) .. controls (123.33,263.1) and (122.21,264.83) .. (120.83,264.83) .. controls (119.45,264.83) and (118.33,263.1) .. (118.33,260.96) -- cycle ;
\draw   (137.83,261.46) .. controls (137.83,259.32) and (138.95,257.58) .. (140.33,257.58) .. controls (141.71,257.58) and (142.83,259.32) .. (142.83,261.46) .. controls (142.83,263.6) and (141.71,265.33) .. (140.33,265.33) .. controls (138.95,265.33) and (137.83,263.6) .. (137.83,261.46) -- cycle ;
\draw   (164.33,261.96) .. controls (164.33,259.82) and (165.45,258.08) .. (166.83,258.08) .. controls (168.21,258.08) and (169.33,259.82) .. (169.33,261.96) .. controls (169.33,264.1) and (168.21,265.83) .. (166.83,265.83) .. controls (165.45,265.83) and (164.33,264.1) .. (164.33,261.96) -- cycle ;
\draw   (203.33,211.75) -- (219.83,211.75) -- (219.83,208.25) -- (230.83,215.25) -- (219.83,222.25) -- (219.83,218.75) -- (203.33,218.75) -- cycle ;
\draw    (311,20.17) -- (359.5,111.42) ;
\draw    (311,20.17) -- (262.5,111.42) ;
\draw    (276.9,84.27) -- (290.2,111) ;
\draw    (291,58.67) -- (333,110.2) ;
\draw    (318.9,93.17) -- (308.6,109.8) ;
\draw   (260,115.29) .. controls (260,113.15) and (261.12,111.42) .. (262.5,111.42) .. controls (263.88,111.42) and (265,113.15) .. (265,115.29) .. controls (265,117.43) and (263.88,119.17) .. (262.5,119.17) .. controls (261.12,119.17) and (260,117.43) .. (260,115.29) -- cycle ;
\draw   (287.7,114.88) .. controls (287.7,112.73) and (288.82,111) .. (290.2,111) .. controls (291.58,111) and (292.7,112.73) .. (292.7,114.88) .. controls (292.7,117.02) and (291.58,118.75) .. (290.2,118.75) .. controls (288.82,118.75) and (287.7,117.02) .. (287.7,114.88) -- cycle ;
\draw   (306.1,113.68) .. controls (306.1,111.53) and (307.22,109.8) .. (308.6,109.8) .. controls (309.98,109.8) and (311.1,111.53) .. (311.1,113.68) .. controls (311.1,115.82) and (309.98,117.55) .. (308.6,117.55) .. controls (307.22,117.55) and (306.1,115.82) .. (306.1,113.68) -- cycle ;
\draw   (330.5,114.08) .. controls (330.5,111.93) and (331.62,110.2) .. (333,110.2) .. controls (334.38,110.2) and (335.5,111.93) .. (335.5,114.08) .. controls (335.5,116.22) and (334.38,117.95) .. (333,117.95) .. controls (331.62,117.95) and (330.5,116.22) .. (330.5,114.08) -- cycle ;
\draw   (357,115.29) .. controls (357,113.15) and (358.12,111.42) .. (359.5,111.42) .. controls (360.88,111.42) and (362,113.15) .. (362,115.29) .. controls (362,117.43) and (360.88,119.17) .. (359.5,119.17) .. controls (358.12,119.17) and (357,117.43) .. (357,115.29) -- cycle ;
\draw    (312.2,170.6) -- (360.7,261.85) ;
\draw    (312.2,170.6) -- (263.7,261.85) ;
\draw    (278.1,234.7) -- (291.4,261.43) ;
\draw    (292.2,209.1) -- (334.2,260.63) ;
\draw    (320.1,243.6) -- (309.8,260.23) ;
\draw   (261.2,265.73) .. controls (261.2,263.58) and (262.32,261.85) .. (263.7,261.85) .. controls (265.08,261.85) and (266.2,263.58) .. (266.2,265.73) .. controls (266.2,267.87) and (265.08,269.6) .. (263.7,269.6) .. controls (262.32,269.6) and (261.2,267.87) .. (261.2,265.73) -- cycle ;
\draw   (288.9,265.31) .. controls (288.9,263.17) and (290.02,261.43) .. (291.4,261.43) .. controls (292.78,261.43) and (293.9,263.17) .. (293.9,265.31) .. controls (293.9,267.45) and (292.78,269.18) .. (291.4,269.18) .. controls (290.02,269.18) and (288.9,267.45) .. (288.9,265.31) -- cycle ;
\draw   (307.3,264.11) .. controls (307.3,261.97) and (308.42,260.23) .. (309.8,260.23) .. controls (311.18,260.23) and (312.3,261.97) .. (312.3,264.11) .. controls (312.3,266.25) and (311.18,267.98) .. (309.8,267.98) .. controls (308.42,267.98) and (307.3,266.25) .. (307.3,264.11) -- cycle ;
\draw   (331.7,264.51) .. controls (331.7,262.37) and (332.82,260.63) .. (334.2,260.63) .. controls (335.58,260.63) and (336.7,262.37) .. (336.7,264.51) .. controls (336.7,266.65) and (335.58,268.38) .. (334.2,268.38) .. controls (332.82,268.38) and (331.7,266.65) .. (331.7,264.51) -- cycle ;
\draw   (358.2,265.73) .. controls (358.2,263.58) and (359.32,261.85) .. (360.7,261.85) .. controls (362.08,261.85) and (363.2,263.58) .. (363.2,265.73) .. controls (363.2,267.87) and (362.08,269.6) .. (360.7,269.6) .. controls (359.32,269.6) and (358.2,267.87) .. (358.2,265.73) -- cycle ;

\draw (68.4,120.6) node [anchor=north west][inner sep=0.75pt]  [font=\tiny]  {$1$};
\draw (120,120.2) node [anchor=north west][inner sep=0.75pt]  [font=\tiny]  {$2$};
\draw (140,120.2) node [anchor=north west][inner sep=0.75pt]  [font=\tiny]  {$3$};
\draw (167,120.03) node [anchor=north west][inner sep=0.75pt]  [font=\tiny]  {$4$};
\draw (165.2,268.6) node [anchor=north west][inner sep=0.75pt]  [font=\tiny]  {$4$};
\draw (97.6,119.4) node [anchor=north west][inner sep=0.75pt]  [font=\tiny]  {$5$};
\draw (257.6,122.6) node [anchor=north west][inner sep=0.75pt]  [font=\tiny]  {$1$};
\draw (287.2,122.2) node [anchor=north west][inner sep=0.75pt]  [font=\tiny]  {$5$};
\draw (357.2,121.8) node [anchor=north west][inner sep=0.75pt]  [font=\tiny]  {$2$};
\draw (305.6,121.4) node [anchor=north west][inner sep=0.75pt]  [font=\tiny]  {$3$};
\draw (330,120.6) node [anchor=north west][inner sep=0.75pt]  [font=\tiny]  {$4$};
\draw (66,269.6) node [anchor=north west][inner sep=0.75pt]  [font=\tiny]  {$2$};
\draw (94,269.6) node [anchor=north west][inner sep=0.75pt]  [font=\tiny]  {$5$};
\draw (137.6,268.4) node [anchor=north west][inner sep=0.75pt]  [font=\tiny]  {$1$};
\draw (118,268.4) node [anchor=north west][inner sep=0.75pt]  [font=\tiny]  {$3$};
\draw (260,273) node [anchor=north west][inner sep=0.75pt]  [font=\tiny]  {$2$};
\draw (288.8,272.6) node [anchor=north west][inner sep=0.75pt]  [font=\tiny]  {$5$};
\draw (306.4,271.4) node [anchor=north west][inner sep=0.75pt]  [font=\tiny]  {$1$};
\draw (331.6,271.4) node [anchor=north west][inner sep=0.75pt]  [font=\tiny]  {$4$};
\draw (358.8,272.2) node [anchor=north west][inner sep=0.75pt]  [font=\tiny]  {$3$};
\draw (197.2,83.6) node [anchor=north west][inner sep=0.75pt]  [font=\tiny] [align=left] { \ \ SPR\\operation};
\draw (198.8,229.2) node [anchor=north west][inner sep=0.75pt]  [font=\tiny] [align=left] { \ \ SPR\\operation};
\draw (102.33,11.07) node [anchor=north west][inner sep=0.75pt]  [font=\scriptsize]  {$x$};
\draw (97.67,161.4) node [anchor=north west][inner sep=0.75pt]  [font=\scriptsize]  {$y$};

\end{tikzpicture}

\bigskip
    \caption{Same SPR operation on different trees, however with the same shape.}
    \label{fig:shapeSPR}
\end{figure}

The following result is crucial in showing that the Metropolis-Hastings random walk and lazy random walks are lumpable with respect to the partition $F$.

\begin{lemma}[\textbf{Identical shape tree neighbourhood}]\label{lem:F_j^xn=F_j^yn}
Let $x$ and $y$ be trees in $\TT_n$, with $n \ge 3$, such that they have the same shape, i.e. $x, y \in F_i$.
Then for all $j \in \{1, 2, \dots, v\}$ we have $|F_j^{x}| = |F_j^{y}|$.
\end{lemma}
\begin{proof}
Suppose we select a tree in $F_i$ and remove its labels.
We perform all possible SPR operations for this tree to obtain all possible shapes from the initial tree.
Therefore, these shapes represent all the tree forms that 
$x$ and $y$ can potentially reach.
Since all SPR operations performed on the tree without labels can also be applied to 
$x$ and $y$, we thereby obtain the desired result.   
\end{proof}

Now we are ready to prove that the Metropolis-Hastings random walk is lumpable with respect to the tree shape partition $\Bar{F}$.

\begin{theorem}[\textbf{Metropolis-Hastings is lumpable with respect to tree shapes}]
\label{teo:shape_lump_MH}
The SPR Metropolis-Hastings random walk is lumpable with respect to the tree shape partition $\bar{F} :=\{F_1, F_2, \dots, F_v\}$.
\end{theorem}
\begin{proof}
Fix $j \in \{1,2, \dots, v\}$ and $i \in \{1,2, \dots, v\}$.
For all $x, y \in F_i$, by  \eqref{eq:MH_SPR_RW} we have
\begin{equation}\label{eq:shape_lump_1}
\sum_{z \in F_j} p(x, z) - \sum_{z \in F_j} p(y, z) = \sum_{z \in F_j^{x}} \frac{1}{|N(x)|}\min\left\{1, \frac{|N(x)|}{|N(z)|} \right\} - \sum_{z \in F_j^{y}} \frac{1}{|N(y)|}\min\left\{1, \frac{|N(y)|}{|N(z)|} \right\}\,.
\end{equation}

Then we can have two possible situations that we will analyse separately.
First, suppose that $|N(x)| \ge |N(z')|$, for $z' \in F_j$. Since we have for all $k \in \{1,2,\dots, v\}$ and $\phi , \phi' \in F_k$ that $|N(\phi)| = |N(\phi')|$, we can continue the computation in~\eqref{eq:shape_lump_1} and by Lemma~\ref{lem:F_j^xn=F_j^yn} obtain the following
\begin{equation*}
\sum_{z \in F_j} p(x, z) - \sum_{z \in F_j} p(y, z) = \frac{|F_j^{x}| - |F_j^{y}|}{|N(x)|} = 0 \,.    
\end{equation*}

Hence we would have the Metropolis-Hastings random walk lumpable with respect to $\bar{F}$.

Now we suppose that $|N(x)| < |N(z')|$, for $z'\in F_j$.
We continue the computation in~\eqref{eq:shape_lump_1} and by Lemma~\ref{lem:F_j^xn=F_j^yn} obtain the following
\begin{equation*}
\sum_{z \in F_j} p(x, z) - \sum_{z \in F_j} p(y, z) =  \frac{|F_j^{x}| - |F_j^{y}|}{|N(z')|} = 0\,.    
\end{equation*}

Hence we have the desired result.

\end{proof}

By the exact computations we did in Theorem~\ref{teo:shape_lump_MH} we obtain the following result

\begin{theorem}[\textbf{Random walk is lumpable with respect to tree shapes}]
\label{teo:shape_lump_LRW}
The lazy random walk is lumpable with respect to the tree shape partition $\bar{F}$.
\end{theorem}

\begin{proof}
By Lemma~\ref{lem:F_j^xn=F_j^yn} and Eq.~\eqref{eq:simple_SPR_RW} we have the desired result. 
\end{proof}


\subsection{Recovering the stationary distribution from the lumpable process}\label{subsec:stat_proj}

As we have seen in Theorem~\ref{teo:shape_lump_MH} and~\ref{teo:shape_lump_LRW}, both the Metropolis-Hastings random walk and the lazy random walk are lumpable with respect to the tree shape partition $\bar{F}$. Hence, the projected chain forms a Markov Chain on $\bar{F}$. 

In the case of the Metropolis-Hastings random walk, the stationary distribution is uniform, given by $\pi_X(x) = 1/|\TT_n|$ for all $x \in \boldsymbol{T}_n$. Thus, by Proposition~\ref{prop: sta_lump}, stationary distribution of the projected chain $(Y_k)_{k \ge 0}$ is given by $\pi_Y (F_i) = |F_i|/|\TT_n|$ for all $i \in \{1, 2, \dots, v\}$.

Now in the lazy random walk case, its stationary distribution is given by $\pi_X (x) = |N(x)|/2|E_n|$ for all $x \in \TT_n$. Hence, by Proposition~\ref{prop: sta_lump}, $\pi_Y (F_i) = |F_i||N(x^i)|/2|E_n|$, for all $i \in \{1, 2, \dots, v \}$ and any $x^i \in F_i$.

The results above regarding the stationary distributions and lumpability are important because they allow us to recover the original stationary distribution from the empirical distribution of the projected chain in a smaller space as we will see in the next result. The empirical measure of a Markov chain $(Z_k)_{k \geq 0}$ is defined as 
\begin{equation}\label{eq:empirical}
    \hat{\mu}_{Z_k} (x) := \frac{1}{k} \sum_{i = 1}^k \mathds{1}_{\{Z_i = x\}} \quad \text{for all } x \in \mathcal{S} \,.
\end{equation}

\begin{proposition}\label{prop:sta-Lump_emp}
For every $x \in F_i$, we define $\hat{\eta}_k(x):= \hat{\mu}_{Y_k}(F_i) \times |F_i|^{-1}$ for all $k \in \mathbb{N}$. Then, 
\begin{equation*}
    || \hat{\eta}_{k} - \pi_X || \to 0\,,
\end{equation*}
as $k \to \infty$.
\end{proposition}
\begin{proof}
For both processes we have $\pi_X(x) = \pi_X(x')$ for any $x, x' \in F_i$ and we remind the reader that, from Proposition~\ref{prop: sta_lump}, $\mu_Y(F_i) = \sum_{x \in F_i} \pi_X(x)$.
Thus we have
\begin{equation}\label{eq:new_sec11}
\begin{split}
|| \hat{\eta}_{k} - \pi_X || & = \frac{1}{2} \sum_{x \in \TT_n} | \hat{\eta}_k (x) - \pi_X (x) | = \frac{1}{2} \sum_{i = 1}^v \sum_{x \in F_i} |\hat{\eta}_k(x) - \pi_X(x)|
\\
& = \frac{1}{2} \sum_{i = 1}^v | \hat{\mu}_{Y_k}(F_i) - \mu_Y(F_i) | = || \hat{\mu}_{Y_k} - \mu_Y ||\,.
\end{split}
\end{equation}
where we have used, in the third equality of~\eqref{eq:new_sec11}, the definition of $\hat{\eta}_k$ and the fact that $\pi_X$ is constant on $F_i$, for each $i$.

Since both processes are lumpable with respect to $\bar{F}$ by Theorems~\ref{teo:shape_lump_MH},~\ref{teo:shape_lump_LRW}, and Proposition~\ref{prop: sta_lump}, we have that $|| \hat{\mu}_{Y_k} - \mu_Y ||$ goes to zero as $k \to \infty$. This completes the proof. 
\end{proof}

\section{$\varepsilon$-Lumpability for tree-valued Markov chains}
\label{sec:eps_tree_lumpability}



In contrast with the results in Section~\ref{sec:tree_lumpability}, we now explore a situation where exact lumpability is not attainable.
In particular, we show that a tree-valued Markov chain need not necessarily induce a Markov chain on the space of clades, $\boldsymbol{C_n}$.
This section discusses the conditions for $\varepsilon$-lumpability to hold and bound the error.

Let us recall the definition of lumping error. Consider a Markov process on a state space $\mathcal{S}$ with a partition $\bar{S} = \{E_1, \ldots,  E_h\}$. Then for any $x,y \in E_i$ we define
\[  |R_{i, j} (x,y)| =  \Big| \sum_{z \in E_j} p(x, z) - p(y, z) \Big| \,, \]
When $|R_{i,j}(x, y)| \leq \varepsilon$ for every pair $x, y$ and every $i, j \in \{1,2, \dots, h\}$, we define  $\varepsilon$ as the lumping error for a Markov chain with respect to the partition $\bar{S}$ (where $\varepsilon \in (0,1)$). We are always interested in obtaining the smallest possible value of $\varepsilon$ that satisfies the condition.



\subsection{$\varepsilon$-Lumpability for the Metropolis-Hastings random walk}  

We start by investigating the lumpability of the Metropolis-Hastings random walk with respect to the clade partition.
With the results obtained in section~\ref{sec:prelim2}, we are now equipped with the necessary tools to analyse the lumping error associated with this process.
It will become evident from the forthcoming results that the lumping error tends to zero as the number of leaves increases, particularly for small clades.

First, we analyse the lumping error for the process as it transitions from the partition containing the clade, $S_1(c)$, to its complement, $S_0(c)$.
Essentially, this lumping error indicates that determining the probability of the process moving from $S_1(c)$ to $S_0(c)$ requires knowledge of the specific tree in which the process currently is, a requirement that is unnecessary for the lumped situation.
Theorem~\ref{teo:error_lump_mh} provides explicit bounds on the lumping error of the Metropolis-Hastings random walk with respect to clades.

\begin{theorem}[\textbf{Bounds for the lumping error in the Metropolis-Hastings random walk}]
\label{teo:error_lump_mh}
Consider the Metropolis-Hastings random walk $(X_k)_{k \ge 0}$ and the partition $\bar{S}:=\{S_0(c), S_1(c)\}$ of $\TT_n$.
Then the lumping error for $(X_k)_{k \ge 0}$ with respect to the partition $\bar{S}$ is evaluated, for $|c| = 2$,
\begin{equation*}
    \varepsilon = \begin{cases}
     \dfrac{2n-5}{3n^2-13n +14} -  \dfrac{5}{6(4(n-2)^2 - 2\sum_{j=1}^{n-2} \lfloor \log_2 (j+1) \rfloor)}, \quad \text{for } 4 \le n \le 8,
     \vspace{10pt}\\ 
     \dfrac{8n-22}{3n^2 -11n +8} - \dfrac{5(6n - 16)}{6(4(n-2)^2 - 2\sum_{j=1}^{n-2} \lfloor \log_2 (j+1) \rfloor)}, \quad \text{for } n \ge 9\,.
    \end{cases}
\end{equation*}
For $3 \le |c| \le \lfloor n^{1/2} \rfloor$ and $n \ge 9$,
\begin{equation*}
    \varepsilon = \frac{-8|c|^2 + 8|c|n + 6|c| - 8n -2}{3n^2 -2|c|^2 +2|c|n-15n+16} - \frac{5(-8|c|^2 + 8|c|n - 2(n-3)(|c|-1) + 6|c| -8n-2)}{6(4(n-2)^2 - 2\sum_{j=1}^{n-2} \lfloor \log_2 (j+1) \rfloor)}\,.
\end{equation*}
\end{theorem}

\begin{remark}\label{rem:Lump_error}
In Theorem~\ref{teo:error_lump_mh},  we compute upper bounds for $R_{1,0}(x, y)$ and $R_{0,1}(x, y)$ for all $x,y \in S_1(c)$ and $x,y \in S_0(c)$, respectively, from which we can directly derive upper bonds to $R_{1,1}(x, y)$ and $R_{0,0}(x, y)$. Indeed, note that for any $x, y \in S_1(c)$, we have
\begin{equation*}
\begin{split}
|R_{1,1}(x, y)| & = \Big| \sum_{z \in S_1(c)} p(x,z) - p(y,z) \Big|  = \Big| 1 - \sum_{z \in S_0(c)} p(x,z) - 1 + \sum_{z \in S_0(c)} p(y,z) \Big|
\\
& = \Big| \sum_{z \in S_0(c)} p(y, z) - p(x, z) \Big| = |R_{0,1}(x,y)|\,.
\end{split}    
\end{equation*}
\end{remark}

\begin{proof}[Porof of Theorem \ref{teo:error_lump_mh}]
We start the proof computing an upper bound to $R_{1,0}(x,y)$ for any $x,y \in S_1(c)$, which we will denote by $\delta_{1,0}$. It is worth to remember that for any $x \in S_1(c)$, we set $A_1^{x,c}:=N(x) \cap S_1(c)$ and $A_0^{x,c}:= N(x) \setminus A_1^{x,c}$.  
For all $x, y \in S_1 (c)$, by \eqref{eq:lumping_error} and \eqref{eq:MH_SPR_RW}, we find
\begin{equation}\label{eq:lumping_error_c_1}
\begin{split}
R_{1, 0}(x,y) & = \sum_{z \in A_0^{x,c}} p(x, z) - \sum_{z \in A_0^{y,c}} p(y, z)
\\
& = \sum_{z \in A_0^{x,c}} \frac{1}{|N(x)|}\min\left\{1, \frac{|N(x)|}{|N(z)|} \right\} - \sum_{z \in A_0^{y,c}} \frac{1}{|N(y)|}\min\left\{1, \frac{|N(y)|}{|N(z)|} \right\}
\\
& \le \frac{|A_0^{x,c}|}{|N(x)|} - \frac{5 |A_0^{y,c}|}{6|N(y)|}\,.
\end{split}
\end{equation}
We obtain the last inequality in~\eqref{eq:lumping_error_c_1} by Lemma 5.5 in \cite{Whidden2017}, which tightens the bound on the ratio of neighbourhood sizes. 

Now we will analyse separately the case $|c| = 2$ and $3 \le |c| \le \lfloor n^{1/2} \rfloor$. We begin with the case $3 \le |c| \le \lfloor n^{1/2} \rfloor$. By Lemma~\ref{lem:max_A_N} and~\eqref{eq:lumping_error_c_1} we obtain the following
\begin{equation}\label{eq:R_c(x,y)1}
R_{1, 0}(x,y) \le \frac{-8|c|^2 + 8|c|n + 6|c| -8n-2}{3n^2 -2|c|^2 +2|c|n-15n+16} - \frac{5|A_0^{y,c}|}{6|N(y)|}\,. 
\end{equation}

By Lemma~\ref{lem:A_0} and the fact for all $y \in \TT_n$ we have $\gamma_y(w) \leq n-3$, we obtain
\begin{equation}\label{eq:Ny-Ny'-Nphi}
\begin{split}
|A_0^{y,c}|& =-8|c|^2 + 8|c|n - 2\gamma_{y}(w_c)(|c|-1) + 6|c| -8n-2 
\\
& \leq 8|c|^2 + 8|c|n - 2(n-3)(|c|-1) + 6|c| -8n-2\,. 
\end{split}
\end{equation}

Thus by Corollary~\ref{max_min_nei}, \eqref{eq:lumping_error_c_1} \eqref{eq:R_c(x,y)1} and~\eqref{eq:Ny-Ny'-Nphi} we obtain
\begin{equation}\label{eq:R_c(x,y)_upper}
\begin{split}
&R_{1, 0}(x,y) \le 
\\
& \frac{-8|c|^2 + 8|c|n + 6|c| -8n-2}{3n^2 -2|c|^2 +2|c|n-15n+16} - \frac{5(-8|c|^2 + 8|c|n - 2(n-3)(|c|-1) + 6|c| -8n-2)}{6(4(n-2)^2 - 2\sum_{j=1}^{n-2} \lfloor \log_2 (j+1) \rfloor)} \,.
\end{split}    
\end{equation}



For the case that $|c| = 2$, we obtain by Lemma~\ref{lem:max_A_N} and~\eqref{eq:lumping_error_c_1} the following
\begin{equation*}\label{eq:R_c(x,y)2}
R_{1, 0}(x,y) \le \frac{8n-22}{3n^2 -11n +8} - \frac{5|A_0^{y,c}|}{6|N(y)|}\,. 
\end{equation*}

By Lemma~\ref{lem:A_0} and the fact for all $y \in \TT_n$ we have $\gamma_y(w) \leq n-3$, we obtain
\begin{equation*}\label{eq:Ny-Ny'-Nphi_2}
\begin{split}
|A_0^{y,c}|& =8n-22 - 2\gamma_y(w_c)
\\
& \geq 8n-22 - 2(n-3) = 6n - 16\,. 
\end{split}
\end{equation*}

Then by the same techniques we did for the case that $|c| \ge 3$ we obtain 
\begin{equation*}\label{eq: R_2(x,y)_MH}
R_{1, 0}(x,y) \le \frac{8n-22}{3n^2 -11n +8} - \frac{5(6n - 16)}{6(4(n-2)^2 - 2\sum_{j=1}^{n-2} \lfloor \log_2 (j+1) \rfloor)}\,. 
\end{equation*}

We will now compute an upper bound for $R_{0,1}(x,y)$ for all $x, y \in S_0 (c)$, which we will denote by $\delta_{0,1}$. We start when $3 \le |c| \le \lfloor n^{1/2}\rfloor$. We recall that for $x \in S_0(c)$ we set $B_1^{x,c}:= N(x) \cap S_1(c)$ and $B_0^{x,c}:=N(x) \setminus B_1^{x,c}$. For all $x, y \in S_0 (c)$, by \eqref{eq:lumping_error} and \eqref{eq:MH_SPR_RW} we have that
\begin{equation}\label{eq:lumping_error_S0_1}
\begin{split}
R_{0, 1}(x,y) & = \sum_{z \in B_1^{x,c}} p(x, z) - \sum_{z \in B_1^{y,c}} p(y, z)
\\
& \leq \sum_{z \in B_1^{x,c}} \frac{1}{|N(x)|}\min\left\{1, \frac{|N(x)|}{|N(z)|} \right\}  \le \frac{|B_1^{x,c}|}{3n^2-13n +14}
\\
& \leq \frac{2n}{3n^2-13n +14}  \,.
\end{split}
\end{equation}
In the first inequality in~\eqref{eq:lumping_error_S0_1} we use Lemma~\ref{lem:B_1^xc}, since there exists $y \in S_0(c)$, which $|B_1^{y,c}|=0$. In the second inequality in~\eqref{eq:lumping_error_S0_1} we applied Corollary~\ref{max_min_nei} and Lemma~\ref{lem:B_1^xc}, since for all $x \in S_0(c)$ we have $|B_1^{x,c}| \leq 2n$.

Now for the case that $|c| = 2$, we obtain by Lemma~\ref{lem:max_B_N} and~\eqref{eq:lumping_error_c_1} the following
\begin{equation*}
R_{0,1}(x,y) \le \frac{2n-5}{3n^2-13n +14} - \frac{5}{6|N(y)|}\,.
\end{equation*}

Hence by Corollary~\ref{max_min_nei}, for all $x, y \in S_0(c)$ we obtain
\begin{equation*}
R_{0, 1}(x,y) \le \frac{2n-5}{3n^2-13n +14} - \frac{5}{6(4(n-2)^2 - 2\sum_{j=1}^{n-2} \lfloor \log_2 (j+1) \rfloor)}\,.    
\end{equation*}

Now we are prepared to determine the lumping error, specifically identifying the larger of the two. Initially, we analyse the case when $|c| = 2$. We deduce that  $\delta_{1,0} - \delta_{0,1}$ is a polynomial that asymptotically approaches 0 as $n$ increases to infinity. For all $n \ge 4$, the analysis yields the following results:
\begin{equation}\label{eq:R1-R0}
\begin{split}
\delta_{1,0} - \delta_{0,1} & \ge \frac{8n -22}{3n^2 - 11n + 8} - \frac{2n -2}{3n^2 - 13n +14} - \frac{5(6n - 15)}{6(3n^2 -11n + 8)}
\\
& \ge \frac{18n -57}{6(3n^2 - 11n + 8)} - \frac{2n - 2}{3n^2 - 13n + 14} = f(n) \,.
\end{split}
\end{equation}
In the first inequality in~\eqref{eq:R1-R0}, we used the fact that $6(4(n-2)^2 - 2\sum_{j=1}^{n-2} \lfloor \log_2 (j+1) \rfloor) \ge 3n^2 -11n +8$, by Corollary~\ref{max_min_nei}.  

From~\eqref{eq:R1-R0}, $f(n)$ has one real root for $n \ge 4$, which is approximately $8.8716$ and goes to zero with $n \to \infty$, then we can conclude that for all $n \ge 9$, we have $\delta_{1,0} > \delta_{0,1}$. 

For $3 \le |c| \le \lfloor n^{1/2} \rfloor$ we find 
\begin{equation}\label{eq:gn}
\begin{split}
& \delta_{1,0} - \delta_{0,1} 
\\
& > \frac{h_{n,|c|}}{3n^2 -2|c|^2 +2|c|n-15n+16}- \frac{5g_{n,|c|}}{6(3n^2 -2|c|^2 +2|c|n-15n+16)} - \frac{2n}{3n^2-13n +14} 
\\
& = \frac{-8c^2 + 18cn - 24c -18n + 28}{6(3n^2 -2|c|^2 +2|c|n-15n+16)} - \frac{2n}{3n^2-13n +14}
\\
& > \frac{36n -116}{6(3n^2 - 9n -2)} - \frac{2n}{3n^2-13n +14} = g(n) \,.
\end{split}
\end{equation}
In~\eqref{eq:gn}, the first inequality we obtain by Corollary~\ref{max_min_nei} and the last by Lemma~\ref{lem:aux1} (Appendix~\ref{sec:app_proofs}).  

From~\eqref{eq:gn}, $g(n) > 0$ for $n \ge 9$, since it has no real root for $n \ge 9$ and goes to zero with $n \to \infty$, then we can conclude that for all $n \ge 9$, we have $\delta_{1,0} > \delta_{0,1}$.

  
\end{proof}

\begin{remark}
We choose $|c| \le \lfloor n^{1/2} \rfloor$ to facilitate practical applications.
For example, practitioners usually examine processes in clade space to determine whether they can halt the original process by assessing the average standard deviation of clade frequencies; see~\cite{Lakner2008}.
Since tracking large clades with (usually) low probability does not aid the goal of reducing dimension, it is not worthwhile to consider large clades.   
\end{remark}

\subsection{$\varepsilon$-Lumpability for the lazy random walk}  

We will now compute the lumping error for a lazy random walk with transition probabilities as described in~\eqref{eq:simple_SPR_RW}.
The invariant distribution for this process is not uniform, which makes it interesting to study the behaviour of processes more similar to those which are employed in practice, where the target distribution is rarely uniform.
The techniques employed for this computation closely mirror those used for the Metropolis-Hastings random walk.

\begin{theorem}[\textbf{Bounds for the lumping error in the lazy random walk}]
\label{teo:error_lump_rw}
Consider the lazy random walk $(X_k)_{k \ge 0}$ and the partition $\bar{S}:=\{S_0(c), S_1(c)\}$ for $\TT_n$. Then the lumping error for any for $(X_k)_{k \ge 0}$ with respect to the partition $\bar{S}$ is evaluated, for $|c| = 2$,
\begin{equation*}
    \varepsilon = \begin{cases}
    (1-\rho) \Big( \dfrac{2n-5}{3n^2-13n +14} -  \dfrac{5}{6(4(n-2)^2 - 2\sum_{j=1}^{n-2} \lfloor \log_2 (j+1) \rfloor)} \Big), \quad \text{for } 4 \le n \le 8,
     \vspace{10pt}\\ 
     (1-\rho) \Big(\dfrac{8n-22}{3n^2 -11n +8} - \dfrac{5(6n - 16)}{6(4(n-2)^2 - 2\sum_{j=1}^{n-2} \lfloor \log_2 (j+1) \rfloor)} \Big), \quad \text{for } n \ge 9\,.
    \end{cases}
\end{equation*}
For $3 \le |c| \le \lfloor n^{1/2} \rfloor$ and $n \ge 9$,
\begin{equation*}
    \frac{\varepsilon}{1-\rho} = \frac{-8|c|^2 + 8|c|n + 6|c| - 8n -2}{3n^2 -2|c|^2 +2|c|n-15n+16} - \frac{5(-8|c|^2 + 8|c|n - 2(n-3)(|c|-1) + 6|c| -8n-2)}{6(4(n-2)^2 - 2\sum_{j=1}^{n-2} \lfloor \log_2 (j+1) \rfloor)}\,.
\end{equation*}
\end{theorem}

\begin{proof}
The proof follows closely the same techniques of the Theorem~\ref{teo:error_lump_mh}.

\end{proof}

\subsection{Improving estimation of clade probabilities from a $\varepsilon$-lumpable process}
\label{sec:aux_pro}

In this session we present one of the main results of our paper: how to employ the $\varepsilon$-lumpability with respect to a much smaller space and obtain a good approximation for the stationary of the original process in our setting.

It is important to note that, by Theorems~\ref{teo:error_lump_mh} and~\ref{teo:error_lump_rw}, the Metropolis-Hastings and lazy random walks with respect to the partition $\bar{S}:=\{S_0(c), S_1(c)\}$ are $\varepsilon$-lumpable and the lumping error goes to zero as the number of leaves $n$ grows to infinity, particularly for a small clade, i.e., $|c| \le n^{1-\delta}$ for some $\delta > 0$.
This suggests that, for sufficiently large $n$, the projected process will be close to a Markov process, but unfortunately not Markovian.

We are interested in estimating  $\pi_X(S_1(c)) = \sum_{x \in S_1(c)} \pi_X(x)$, and to avoid confusion on the notation, we define $\eta_X := (\pi_X(S_0(c)), \pi_X(S_1(c)))$.
We will explain how to obtain a reliable estimation of this probability $\eta_X$ using an auxiliary process that resembles the projected process on $\bar{S}$.

This  auxiliary process, denoted by $(\tilde{Y}_k)_{k \geq 0}$, and taking values in the same state space $\bar{S}$, will be a Markov chain, both aperiodic and irreducible, with its probability transition matrix $\tilde{P}$ specified in what follows. The goal of defining this auxiliary process is to create a Markov chain ``close'' to the non-Markovian projected process $(Y_k)_{k \geq 0}$ so its empirical measure (see \eqref{eq:empirical}) is an approximation to $\eta_X$.
Remember that $\hat{\mu}_{\tilde{Y}_k}$ is the empirical measure of $(\tilde{Y}_k)_{k \geq 0}$. Hence, we have
\begin{equation}\label{eq:2exp}
 || \hat{\mu}_{\Tilde{Y}_k} - \eta_X || \le || \hat{\mu}_{\Tilde{Y}_k} - \hat{\mu}_{Y_k} || + || \hat{\mu}_{Y_k}  - \eta_X || \,.   
\end{equation}
In~\eqref{eq:2exp}, the second portion of the sum goes to zero with $k \to \infty$, i.e. the empirical measure of $(Y_k)_{k \ge 0}$ converges to $\eta_X$, since it is smaller or equal to $|| \hat{\mu}_{X_k} - \pi_X||$, where $\hat{\mu}_{X_k}$ is the empirical measure of the original process. 
Thus we will only analyse the first term.
The main idea is to determine the transition probability matrix $\Tilde{P}$ so that we minimise our error in estimating the desired probability in~\eqref{eq:2exp}.
Then, we compute and bound this term as follows:
\begin{equation}\label{eq:2exp_2}
\begin{split}
|| \hat{\mu}_{\Tilde{Y}_k} - \hat{\mu}_{Y_k} || & = \frac{1}{2} \sum_{j = 0}^1 | \hat{\mu}_{\Tilde{Y}_k}(S_j(c)) - \hat{\mu}_{Y_k} (S_j(c)) | =  \frac{1}{2} \sum_{j = 0}^1 \Big| \hat{\mu}_{\Tilde{Y}_k}(S_j(c)) - \sum_{x \in \boldsymbol{T}_n} \hat{\mu}_{X_{k-1}}(x) p(x, S_j(c)) \Big| 
\\
& = \frac{1}{2} \sum_{j = 0}^1 \Big| \sum_{x \in \boldsymbol{T}_n} \hat{\mu}_{X_{k-1}}(x) (\hat{\mu}_{\Tilde{Y}_k}(S_j(c)) - p(x, S_j(c))) \Big| 
\\
& = \frac{1}{2} \sum_{j = 0}^1 \Big| \sum_{x \in \boldsymbol{T}_n} \hat{\mu}_{X_{k-1}}(x) \Big( \sum_{h = 0}^1 \hat{\mu}_{\Tilde{Y}_{k-1}}(S_h(c)) (\Tilde{p}(S_h(c), S_j(c)) -p(x, S_j(c))) \Big) \Big|
\\
& \le \frac{1}{2} \sum_{j = 0}^1 \sum_{i = 0}^1 \sum_{x \in S_i(c)} \sum_{h = 0}^1 \hat{\mu}_{X_{k-1}}(x) \hat{\mu}_{\Tilde{Y}_{k-1}}(S_h(c)) \big| \Tilde{p}(S_h(c), S_j(c)) -p(x, S_j(c)) \big|\,.
\end{split}    
\end{equation}

Let us denote $\Tilde{p}(S_0(c), S_1(c)) = p$ and $\Tilde{p}(S_1(c), S_1(c)) = q$.
The idea is to choose $p$ and $q$ to minimise $|| \Tilde{\mu}_{\Tilde{Y}_k} - \eta_X ||$. However, since this would not be tractable, we choose to minimise the upper bound in \eqref{eq:2exp_2}. Hence, we consider the following minimisation problems:
\begin{equation*}
\begin{split}
& \argmin_{p \in [0,1]} \sum_{x \in \boldsymbol{T}_n} \hat{\mu}_{X_{k-1}}(x)  (|1 - p - p(x, S_0(c))| + |p - p(x, S_1(c))|)  \quad \text{and} 
\\  
& \argmin_{q \in [0,1]} \sum_{x \in \boldsymbol{T}_n} \hat{\mu}_{X_{k-1}}(x) (|1 - q - p(x, S_0(c))| + |q - p(x, S_1(c)|)) \,. 
\end{split}   
\end{equation*}
Since $p(x, S_0(c)) = 1 - p(x, S_1(c))$ we have 
\begin{equation*}
\mathbb{E}[|1 - p - p(X_{k-1}, S_0(c))|] + \mathbb{E}[|p - p(X_{k-1}, S_1(c))|] = 2\mathbb{E}[|p - p(X_{k-1}, S_1(c))|]\,.   
\end{equation*}

We define $M_{ji}$ and $m_{ji}$, for $i, j \in \{0,1\}$, such that $\pr(X_{k} \in S_i(c) | X_{k-1} \in S_j(c)) \in [m_{ji}, M_{ji}]$ for all $k \ge 1$. Hence we can use Theorems~\ref{teo:error_lump_mh} and~\ref{teo:error_lump_rw} to choose the values of $M_{ij}$ and $m_{ij}$, for any $i,j \in \{0,1\}$. To minimise $\mathbb{E}[|p - p(X_{k-1}, S_1)|]$ with respect to $p$, we apply the law of total expectation to obtain
\begin{equation}\label{eq:Exp_min}
\begin{split}
& \mathbb{E}[|p - p(X_{k-1}, S_1(c))|]  = 
\mathbb{E}[|p - p(X_{k-1}, S_1(c))| \mid X_{k-1} \in S_0(c)]\pr[X_{k-1} \in S_0(c)] +
\\
& \mathbb{E}[|p - p(X_{k-1}, S_1(c))| \mid X_{k-1} \in S_1(c)] \pr[X_{k-1} \in S_1(c)]\,. 
\end{split}
\end{equation}

We will focus only on the first portion of the sum in the equality of~\eqref{eq:Exp_min}. This approach is justified by the fact that we do not have the values of the probabilities; however, for any
$c$, it is known that $\pr(X_{k-1} \in S_0(c)) \gg \pr(X_{k-1} \in S_1(c))$, since $S_0(c)$ is much larger than $S_1(c)$, see Theorem 3, part $i)$ in~\cite{Zhu2015}. 

Thus we only have to minimise $\mathbb{E}[|p - p(X_{k-1}, S_1(c))| \mid X_{k-1} \in S_0(c)]$ and that is our next result. To minimise $ \mathbb{E}[|p - p(X_{k-1}, S_1(c))| \mid X_{k-1} \in S_0(c)]$ with respect with $p \in [0,1]$, since we lack lumpability, we treat $\pr( X_k \in S_1(c) \mid X_{k-1} \in S_0(c))$ for all $k \in \mathbb{N}$ as a random variable $Z$, whose support is contained in $[m_{01}, M_{01}]$, as proven in Theorem~\ref{teo:error_lump_mh} and~\ref{teo:error_lump_rw}. We will consider minimising $p$ under the worst-case distribution of $Z$, as stated in the next result:
\begin{proposition}\label{prop:exp_res}
Let $RV[m_{01}, M_{01}]$ be a set of random variables with support contained in $[m_{01}, M_{01}]$. If
\begin{equation}\label{eq:min_mea}
  p^* = \argmin_{p \in [0,1]} \max_{ Z \in RV[m_{01}, M_{01}]} \mathbb{E}[|p - Z|]\,, 
\end{equation}
then $p^*$ is the median of the uniform distribution in $[m_{01}, M_{01}]$:
\begin{equation*}
p^* = \frac{m_{01} + M_{01}}{2}.
\end{equation*}
\end{proposition}

\begin{proof}
We note that to maximise in~\eqref{eq:min_mea} for any $p$, it is sufficient to choose $Z$ with the distribution as the Dirac measure $\delta_{z}$ at the value $z \in [m_{01}, M_{01}]$ that achieves $\max_{z \in [m_{01}, M_{01}]} |p-z|$. That is, we choose the value $z \in [m_{01}, M_{01}]$ most distant from $p$ and set $Z$ to be distributed as $\delta_{z}$. Hence, the maximisation problem becomes a deterministic problem. Therefore, it suffices to analyse the three cases where $p \in [0, m_{01}]$, $p \in [m_{01}, M_{01}]$ and $p \in [M_{01}, 1]$. After considering these cases, we find that number equal to $ 2^{-1}(M_{01} + m_{01})$, which is the median of a uniform distribution on $[m_{01}, M_{01}]$.
Ergo we finish our proof.
\end{proof}
Very similar computations can be done for $q$, yielding $q^* = 2^{-1}(M_{01} + m_{01})$.

\section{Empirical investigation}\label{sec: Emp_Inv}

Having studied both exact and approximate lumpability of tree-valued Markov processes, we return to some of the empirical issues raised in the Introduction and conduct experiments to showcase how the results obtained for lumpability and $\varepsilon$-lumpability can illuminate practical issues related to Markov chain Monte Carlo in treespace.
We present illustrations using exact lumpability on the space of shapes (Section~\ref{sec:shape_experiment}), and $\varepsilon$-lumpability on the space of clades (Section~\ref{sec:eps_clade_lump_experiment}), investigating small ($ 4 \leq n \leq 7$) and moderate ($n=50$) numbers of taxa (leaves).
For $\varepsilon$-lumpability in particular, we show how to use the bounds derived in Section~\ref{sec:eps_tree_lumpability} to create an auxiliary chain that leads to better estimates of clade probabilities as detailed in Section~\ref{sec:aux_pro}.

\paragraph{General setup:} In what follows -- unless otherwise stated -- we ran $J = 500$ independent replicates of each chain (original and  projected) for $M = 10000$ iterations each.
We write $(X_k^j)_{k \ge 0}$ for the original chain and $(Y_k^j)_{k \ge 0}$ for the projected chain, with $j \in \{1,2, \dots, J\}$.
We also define $\hat{\mu}_k^j$ and $\hat{\nu}_k^j$ as the empirical measures of the processes $(X_k^j)_{k \ge 0}$ and $(Y_k^j)_{k \ge 0}$, respectively.

We compare $\hat{\mu}_k^j$ to  $\pi_X$ in terms of total variation distance at each iteration.
For this comparison, we compute the following metrics
\begin{equation*}
\begin{split}
& m_k:= \min_{j \in \{1,2,\dots, J\}} ||\hat{\mu}_k^j - \pi_X ||
\\
& M_k:= \max_{j \in \{1,2,\dots, J\}} ||\hat{\mu}_k^j - \pi_X ||
\\
& E_k:= \frac{1}{J} \sum_{j=1}^{J} ||\hat{\mu}_k^j - \pi_X || \,.
\end{split}
\end{equation*}


\subsection{Shapes: moderate dimension reduction and exact lumpability}
\label{sec:shape_experiment}

In this section we will investigate the results in Section~\ref{sec:tree_lumpability} empirically in order to quantify the differences in speed of convergence between the original and lumped processes when exact lumpability is attainable.
As shown in Theorems~\ref{teo:shape_lump_MH} and~\ref{teo:shape_lump_LRW}, both the Metropolis-Hastings and the (lazy) random walk on the SPR graph are lumpable with respect to phylogenetic shapes.
Furthermore, as shown in Table~\ref{tab:}, the space of phylogenetic trees with $n$ leaves ($\TT_n$) is much larger than the space of unlabelled rooted binary trees, denoted by $\boldsymbol{S}_n$.
Combining this information with Lemma 4 from~\cite{Simper2022}, we argue that the lumped chain should mix faster than the original one.

These observations raise the question of whether it is feasible to sample from the lumped chain and obtain an empirical measure for the original space that accurately approximates the original chain's stationary distribution.
Here we demonstrate that such sampling is possible in certain situations, particularly in our case where the process occurs on the SPR graph and the lumping partition is based on the shapes of the trees.  

We conducted experiments for both the lazy version of the Metropolis-Hastings random walk and the lazy random walk on the SPR graph.
For the lazy version of the Metropolis-Hastings random walk on the SPR graph, we consider the following transition probability matrix
\begin{equation}\label{eq:MH_tpm}   
P_{MH}^{\rho} = \rho I + (1-\rho)P_{MH} \,,
\end{equation}
where $P_{MH}$ is transition probability matrix generated by~\eqref{eq:MH_SPR_RW}.
All experiments were conducted with $\rho \in \{0, 0.1, 0.5, 0.9\}$. 
The initial step of our experiment involves constructing the transition probability matrix for the projected chain.
Here we describe this construction for the lazy Metropolis-Hastings random walk on the SPR graph, but the steps are equally applicable to any lazy random walk on the SPR graph.

\paragraph{Generating the lumped transition probability matrix:} Consider $(X_k)_{k \ge 0}$ as a Markov Chain on the SPR graph, denoted by $G(\TT_n, E_n)$, with its transition probability matrix represented by $P_{MH}^{\rho}$ and a stationary distribution $\pi_X$, which is uniform on $\TT_n$.
We set $(Y_k)_{k \ge 0}$ as the projected chain on the tree shape partition of $\TT_n$, $\Bar{F} = \{F_1, F_2, \dots, F_v\}$.
Thus, the process $(Y_k)_{k \ge 0}$ possesses a transition probability matrix $Q$, derived from $P_{MH}^{\rho}$ as follows
\[
 q(F_i, F_j) = p(x, F_j) \,,
\]
by Theorem~\ref{teo:shape_lump_MH} and the lumpability property, where $x \in F_i$ and $i,j \in \{1,2, \dots, v\}$.
Additionally, the stationary distribution $\pi_Y$ of  $(Y_k)_{k \ge 0}$ is defined by $\pi_Y(F_i) = |F_i|/|\TT_n|$, according to Proposition~\ref{prop: sta_lump}.
We generate the new empirical measure $\hat{\eta}_k^j$ as described in Section~\ref{subsec:stat_proj} -- which converges in total variation to $\pi_X$ as shown in Proposition~\ref{prop:sta-Lump_emp} -- and compute the same metrics as above for $\hat{\eta}_k^j$ at each iteration. 



\paragraph{Results:} The experiments herein focus on the SPR graph for trees with $n = 6$.
According to Table~\ref{tab:}, we find that the total number of trees, $\TT_6$, is 945, and the total number of unique shapes is  $|\boldsymbol{S}_6| = 6$.
We executed all the previously detailed steps for values of $\rho$ within the set $\{0, 0.1, 0.5, 0.9\}$, applying both the lazy Metropolis-Hastings random walk and the lazy random walk, which resulted in the generation of the plots presented below.

We begin by presenting the results for the scenario where the process $(X_k)_{k \ge 0}$ follows the Metropolis-Hastings random walk on the SPR graph, denoted as $G_6 = G(\TT_6, E_6)$, for each $\rho \in \{0, 0.1, 0.5, 0.9\}$ and for the projected process $(Y_k)_{k \ge 0}$ on partition $\bar{F}$. 
\begin{figure}[h]
    \centering
    \includegraphics[scale=0.5]{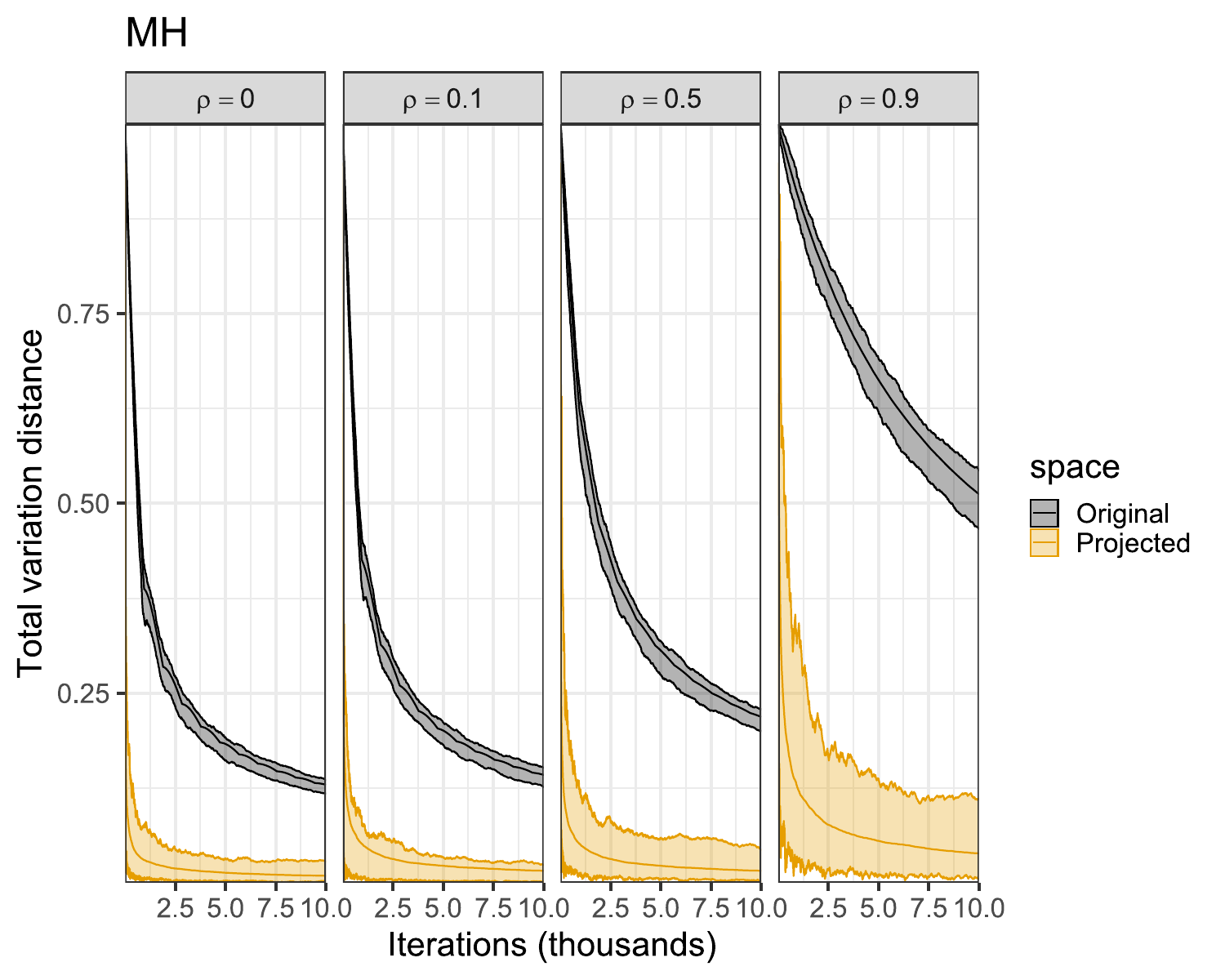}

    \caption{\textbf{Total variation distance per iteration for the shape-projected Metropolis-Hastings (MH)}.
    For each iteration, ranging from 1 to 10,000, we observe the values of $m_k$, $M_k$, and $E_k$ for $\rho \in \{0, 0.1, 0.5, 0.9\}$. }
    \label{fig:MH_tv_err}
\end{figure}
As observed in Figure~\ref{fig:MH_tv_err}, the values of $m_k$, $M_k$, and $E_k$ decrease to zero more rapidly when working with the projected chain for all values of $\rho$. This suggests that utilising the projected chain to sample the stationary distribution results in an approximation closer to the true distribution. As previously mentioned, this outcome is feasible in our context because each point in the partition shares an identical probability mass.

Now, we present the results of our experiments for the scenario in which the process $(X_k)_{k \ge 0}$ is a lazy random walk for each $\rho \in \{0, 0.1, 0.5, 0.9\}$.
In Figure~\ref{fig:RW_tv_err}, it is possible to observe the same phenomenon  depicted in Figure~\ref{fig:MH_tv_err}.
As previously noted, we can draw similar conclusions as those in the context of the Metropolis-Hastings random walk, for the underlying reason remains the same: in the lazy random walk on $G_6$, each point within any given partition has the same probability mass. 

\begin{figure}[h]
    \centering
    \includegraphics[scale=0.5]{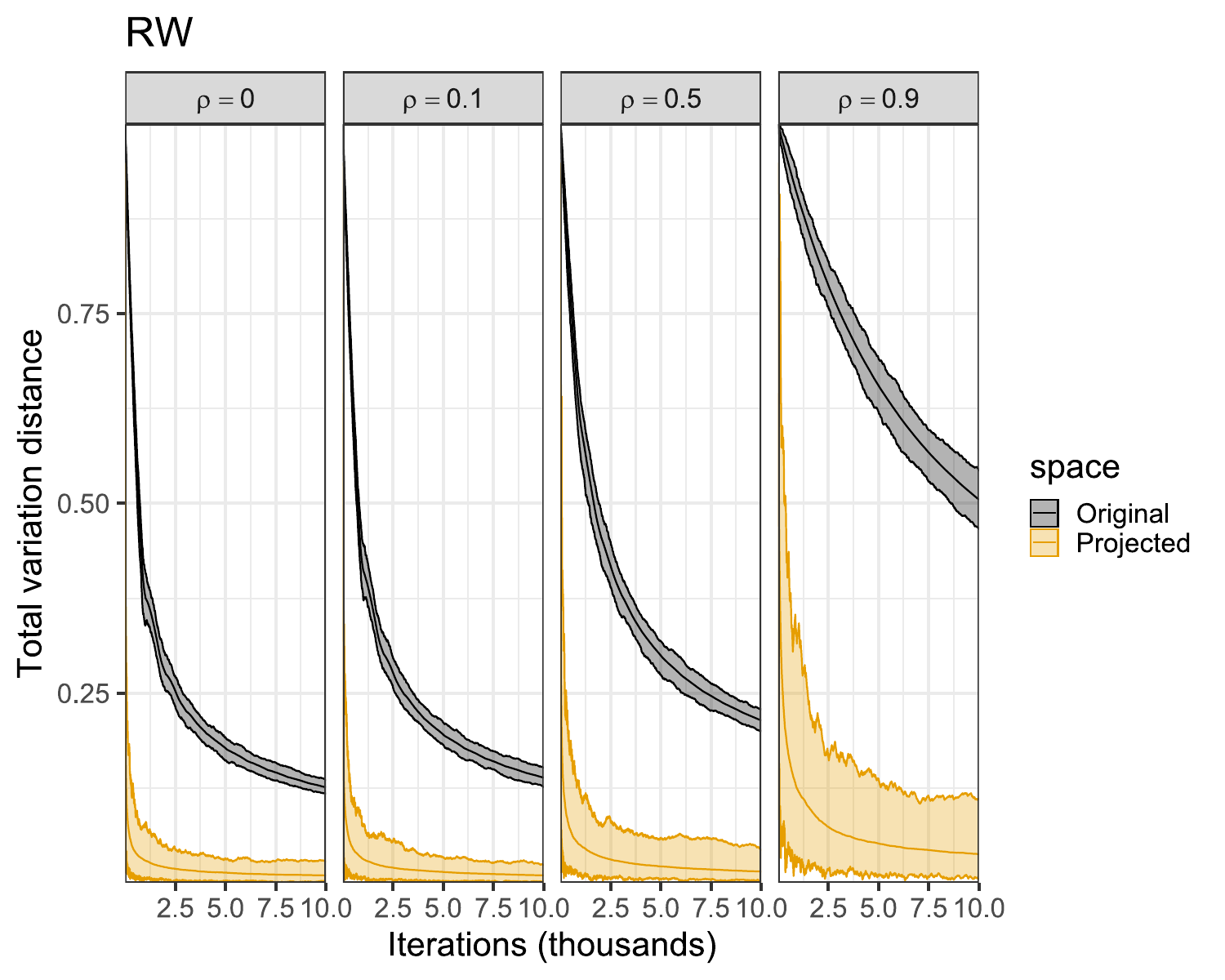}

    \caption{\textbf{Total variation distance per iteration for the shape-projected random walk (RW)}.
    For each iteration, ranging from 1 to 10,000, we observe the values of $m_k$, $M_k$, and $E_k$ for $\rho \in \{0, 0.1, 0.5, 0.9\}$. }
    \label{fig:RW_tv_err}
\end{figure}

\subsection{Clades: aggressive dimension reduction and $\varepsilon$-lumpability}
\label{sec:eps_clade_lump_experiment}

We note that the dimensionality reduction provided by projecting onto shapes, while significant, can still be prohibitive for large $n$.
One might thus wish to project onto an even smaller space.
As shown in Section~\ref{sec:eps_tree_lumpability}, projecting onto clades maintains the interpretability of the process but breaks exact lumpability.
In this experiment, our objective is to estimate the probability of a clade with $|c| = 2$ for cases involving $n = 50$ and $n=100$ leaves, by simulating a smaller Markov Chain with only two states.
The construction of these auxiliary chains draws upon the lumping error calculated in Theorems~\ref{teo:error_lump_mh} and~\ref{teo:error_lump_rw}.
As before, the projected chain we construct is both aperiodic and irreducible, thereby possessing a unique stationary distribution.

\paragraph{The auxiliary Markov chain:} As explained at  Section~\ref{sec:aux_pro} our aim is to minimise $|| \hat{\mu}_{\Tilde{Y}_k} - \eta_X ||$ -- see~\eqref{eq:2exp} -- using a two-state auxiliary chain.
Hence, this auxiliary chain $(\tilde{Y}_k)_{k \ge 0}$ can be constructed using Theorems~\ref{teo:error_lump_mh} and~\ref{teo:error_lump_rw} to set the values of its probability transition matrix $\Tilde{P}$.
We denote $\hat{\mu}_{\Tilde{Y}_k}$ as the empirical measure of the process $(\Tilde{Y}_k)_{k \ge 0}$.

Since in our experiments here we have $|c| = 2$, by Proposition~\ref{prop:exp_res} we obtain
\begin{equation*}
\Tilde{p}(S_0(c), S_1(c)) =   \Tilde{p}(S_1(c), S_1(c)) =  \frac{1}{2} \Big(\frac{2n-5}{3n^2-13n +14} + \frac{5}{6(4(n-2)^2 - 2\sum_{j=1}^{n-2} \lfloor \log_2 (j+1) \rfloor)}  \Big) \,. 
\end{equation*}

\paragraph{Results:} We now show how to use the proposed setup to obtain an estimate of the probability of a tree possessing a clade for cases where $n = 50$ and $n=100$, for a clade $c$ with $|c| = 2$.
It is crucial to note that the specific choice of clade does not affect our results; rather, only the size of the clade is relevant in these experiments.
We first present the results for the process $(\Tilde{Y}_k)_{k \ge 0}$, characterised by the transition probability matrix $\tilde{P}$, which was generated as described previously for $n = 50$, $n = 100$ and $|c| = 2$, to evaluate $\eta_X(S_1(c))$ for a Metropolis-Hastings random walk.  

\begin{figure}[h]
    \centering
    \includegraphics[scale=0.5]{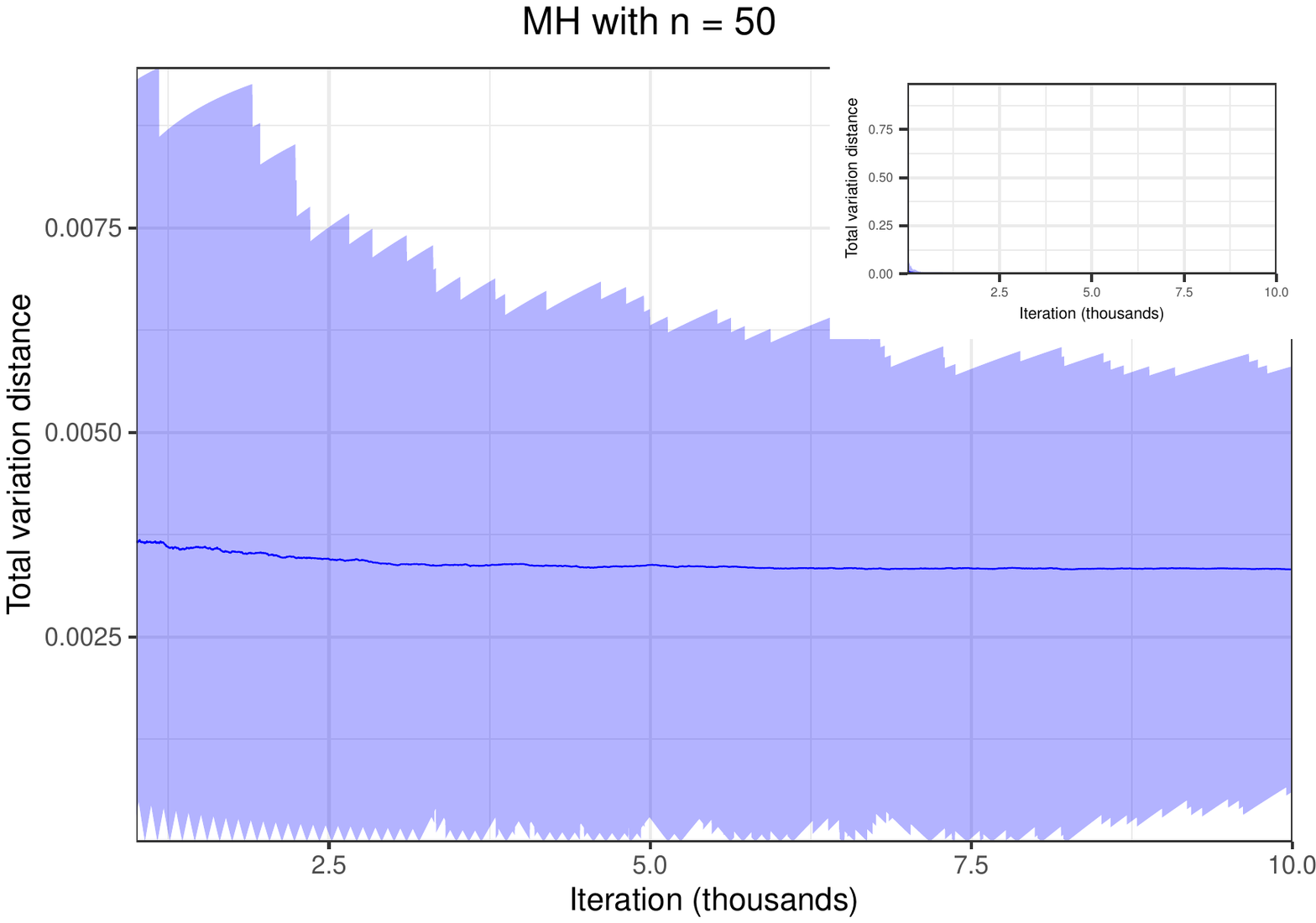}

    \caption{\textbf{Total variation distance per iteration for the clade-projected Metropolis-Hastings(MH) with $n=50$ leaves}.
    The plot on the top left side shows the values of $m_k$, $M_k$ and $E_k$ for each iteration, ranging from 1 to 10.000.
    The main plot shows the same values for iterations ranging from 1.000 to 10.000.}
    \label{fig:tv_err_exp_almost_2}
\end{figure}

Figure~\ref{fig:tv_err_exp_almost_2} shows that $m_t$, $M_t$, and $E_t$ decrease to zero at a fast rate.
We thus argue that the auxiliary chain provides reliable estimation for $\eta_X(S_1(c))$—the probability of the clade $c$.
A significant benefit of utilising the auxiliary chain is its reduced state space size.
Specifically, for the auxiliary chain, the size is now 2, whereas for the original chain, the dimensions is bigger than $10^{80}$.
Consequently, running the auxiliary chain is more feasible than operating the original one. 

To illustrate the accuracy of our estimation regarding the probability of the clade, we present the approach described at the end of Section~\ref{sec:eps_tree_lumpability}.
We have $\eta_X = (|S_0(c)|/|\TT_{50}|, |S_1(c)|/|\TT_{50}|) = (0.9897, 0.0103)$.
For this specific case, we have generated the following transition probability matrix in the manner previously described for the auxiliary process
\begin{equation}\label{eq:P_mat_exp_2}
 \Tilde{P} = \begin{bmatrix}
    1 - \Tilde{p}(S_0(c), S_1(c)) & \Tilde{p}(S_0(c), S_1(c)) \\
    1 - \Tilde{p}(S_1(c), S_1(c)) & \Tilde{p}(S_1(c), S_1(c))
 \end{bmatrix}  = \begin{bmatrix}
     0.9930 & 0.0070 \\
     0.9930 & 0.0070 
 \end{bmatrix}\,.  
\end{equation}

Thus, from~\eqref{eq:P_mat_exp_2}, we derive the following stationary distribution for the auxiliary process, $\mu_{\Tilde{Y}} = (0.9930, 0.0070)$. Given that $||\mu_{\Tilde{Y}} -\eta_X|| = 0.0033$ 
, as illustrated in  Figure~\ref{fig:tv_err_exp_almost_2} we can see that $E_k$ approaches $|| \mu_{\Tilde{Y}} -\eta_X ||$, suggesting we achieve a reliable estimation of the desired probability.

\begin{figure}[h]
    \centering
    \includegraphics[scale=0.5]{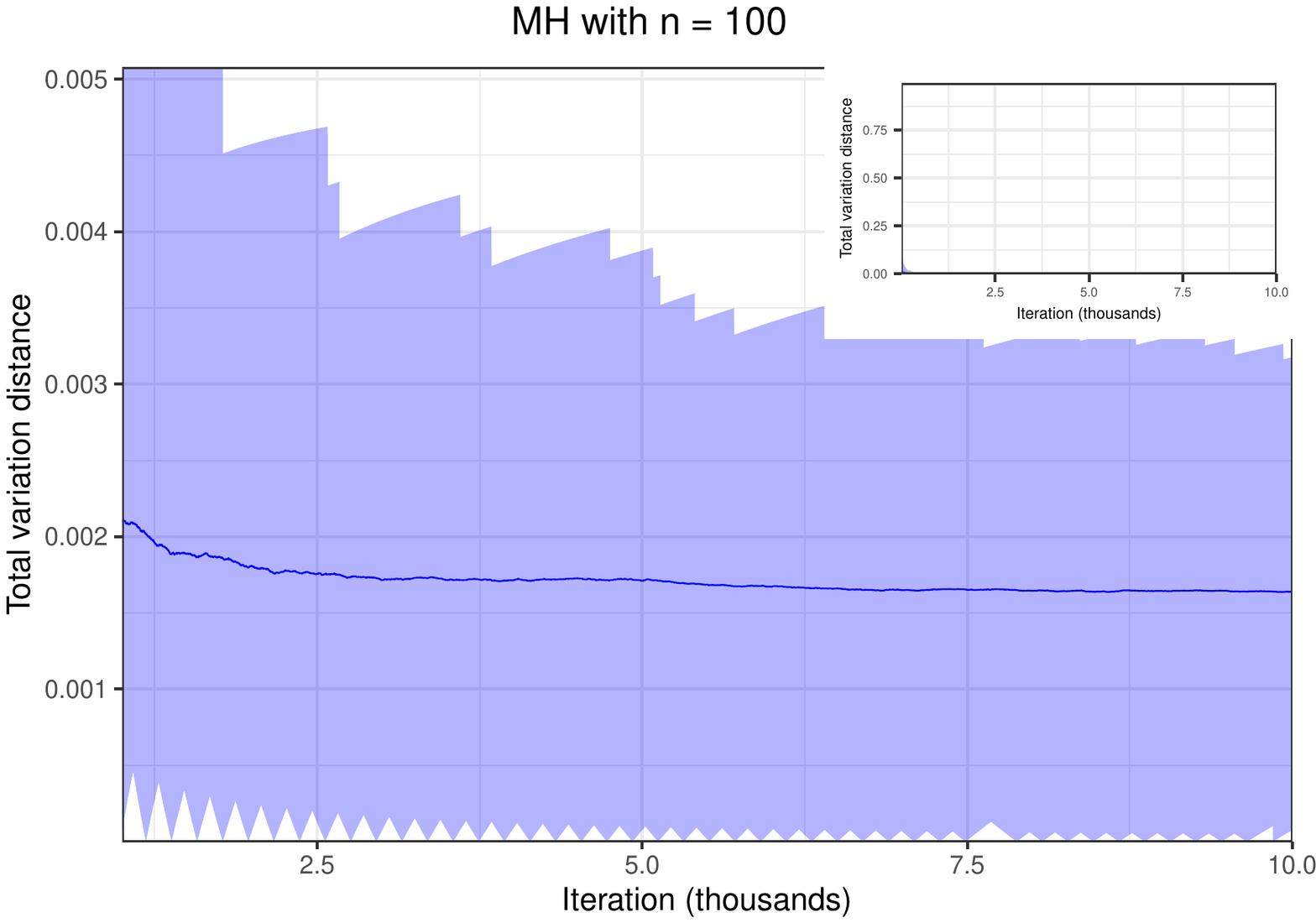}

    \caption{\textbf{Total variation distance per iteration for the clade-projected Metropolis-Hastings(MH) with $n=100$ leaves}.
      The plot on the top left side shows the values of $m_t$, $M_t$ and $E_t$ for each iteration, ranging from 1 to 10.000.
    The main plot shows the same values for iterations ranging from 1.000 to 10.000.}
    \label{fig:tv_err_exp_almost_22}
\end{figure}

Now for Figure~\ref{fig:tv_err_exp_almost_22}  in the same manner as we did for Figure~\ref{fig:tv_err_exp_almost_2}, we have
\begin{equation*}
    \eta_X = (|S_0(c)|/|T_{100}|, |S_1(c)|/|T_{100}|) = (0.9949, 0.0051),
\end{equation*}
and the generated transition probability matrix for the auxiliary process
\begin{equation}\label{eq:P_mat_exp_22}
 \Tilde{P} = \begin{bmatrix}
    1 - \Tilde{p}(S_0(c), S_1(c)) & \Tilde{p}(S_0(c), S_1(c)) \\
    1 - \Tilde{p}(S_1(c), S_1(c)) & \Tilde{p}(S_1(c), S_1(c))
 \end{bmatrix}  = \begin{bmatrix}
     0.9966 & 0.0034 \\
     0.9966 & 0.0034 
 \end{bmatrix}\,.  
\end{equation}

Thus, from~\eqref{eq:P_mat_exp_22}, we obtain $\mu_{\Tilde{Y}} = (0.9966, 0.0034)$.
Given that $|| \mu_{\Tilde{Y}} - \eta_X || = 0.0017$, 
as illustrated in  Figure~\ref{fig:tv_err_exp_almost_2}, $E_k$ approaches $|| \mu_{\Tilde{Y}} -\eta_X ||$.
Hence, we achieve a reliable estimation of the desired probability. 
Our estimation of the probability of the clade improves as we increase the number of leaves from 50 to 100.
This estimation tends to approach the true value as $n$ increases.
We believe this trend can be explained by Theorem~\ref{teo:error_lump_mh}, which shows that the lumping error decreases asymptotically to zero as $n$ increases. 

We also evaluate $\eta_X (S_1(c))$ for a lazy random walk with $n = 50$ where $|c| = 2$.
As our methods for estimating the lumping error in the Metropolis-Hastings random walk are similar, we employ the same auxiliary process $(\Tilde{Y}_k)_{k \ge 0}$ with the identical transition probability matrix $\Tilde{P}$. 
\begin{figure}[h]
    \centering
    \includegraphics[scale=0.5]{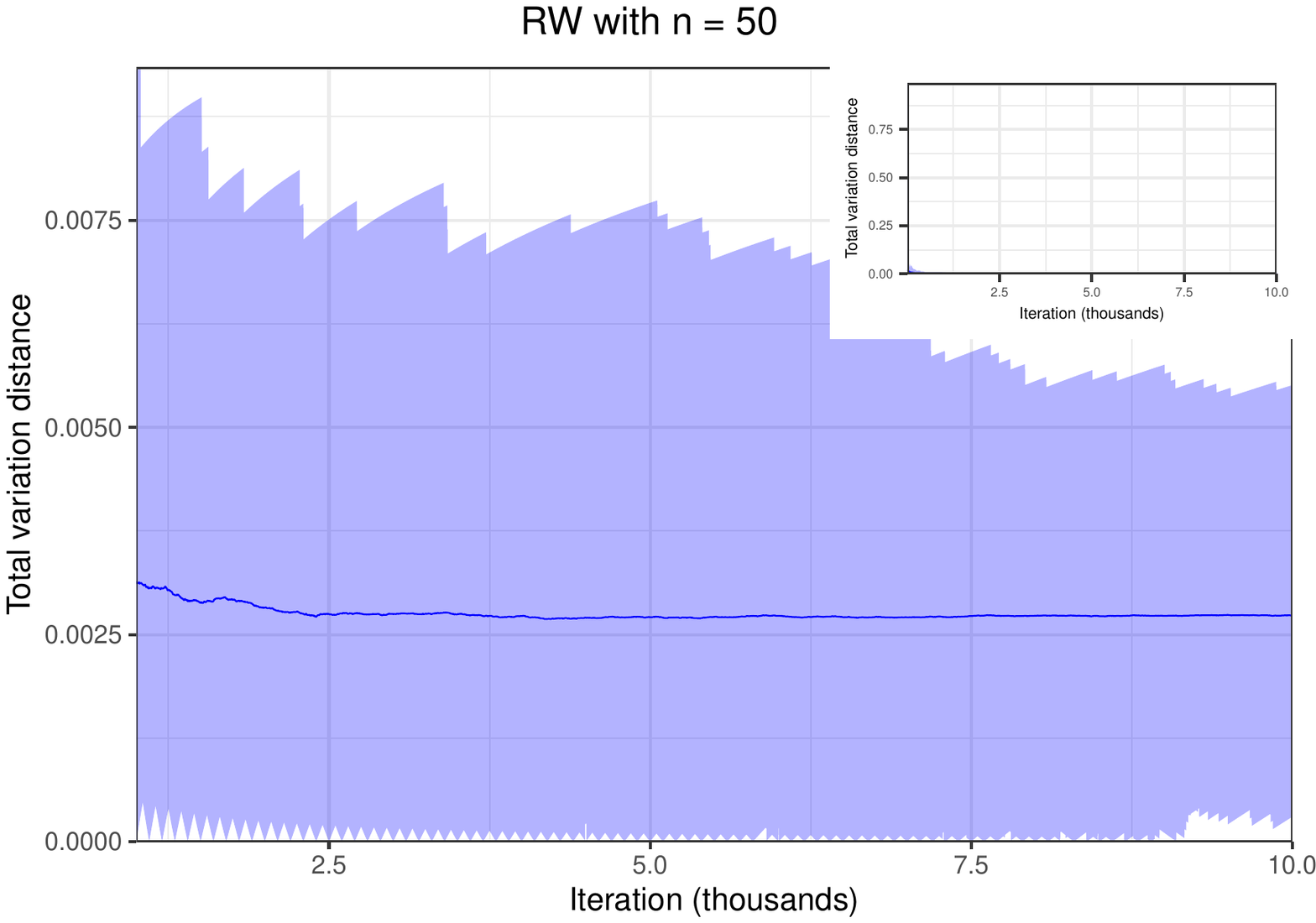}

    \caption{\textbf{Total variation distance per iteration for the clade-projected random walk (RW) with $n=50$ leaves}
    The plot on the top left side shows the values of $m_k$, $M_k$ and $E_k$ for each iteration, ranging from 1 to 10.000. The main one shows the same values, but now, iterations range from 1.000 to 10.000.}
    \label{fig:tv_err_exp_almost_21_RW}
\end{figure}
Unlike the Metropolis-Hastings random walk, where it is possible to analytically compute the probability of having a clade, this computation is not available for the lazy random walk.
To evaluate $\eta_X(S_1(c))$ for the lazy random walk on the SPR graph with $n = 50$, we simulate four independent chains, each with $10^6$ iterations. Since $\eta_X(S_1(c))$ is the same for every clade $c$ of the same cardinality, we observe every clade with $|c| = 2$ across these four chains and calculate the mean frequency to obtain an approximate value for both $\eta_X(S_1(c))$ and $\eta_X(S_0(c))$.
We denote this approximation as $\hat{\eta}_X = (\hat{\eta}_X(S_0(c)), \hat{\eta}_X(S_1(c))$, which will be used in our experiments here. For $n = 50$ we obtained $ \hat{\eta}_X = (0.99, 0.01)$. 
Due to the computational demands of evaluating $\eta_X$, we could not perform the calculation for $n = 100$.

As Figure~\ref{fig:tv_err_exp_almost_21_RW} shows, $m_k$, $M_k$, and $E_k$ decrease to zero at a fast rate.
From this observation, we can conclude that the auxiliary chain provides reliable estimation for $\hat{\eta}_X(S_1(c))$ -- the probability of clade $c$.
To illustrate the accuracy of our estimation regarding the probability of the clade, we refer to the approach described in the experiments with the Metropolis-Hastings random walk.
Since we have $\mu_{\Tilde{Y}} = (0.9930, 0.0070)$ 
as illustrated in  Figure~\ref{fig:tv_err_exp_almost_21_RW} we can see that $E_k$ approaches $|| \mu_{\Tilde{Y}} -\eta_X ||$, thus leading to accurate estimation.

\section{Discussion and Conclusion}
\label{sec:conclusion}

In this paper, we give what is, to the best of our knowledge, the first thorough account of lumpability for tree-valued Markov processes on the SPR graph.
We show it is possible to have exact lumpability on the space of phylogenetic shapes and bound the lumping error when projecting onto clades (subtrees).

Our results help illuminate the observation that clade-projected processes often appear Markov, as shown in Figure~\ref{fig:cladeautocorr} in the Introduction.
As the number of leaves $n$ increases and one restricts attention to relatively small clades, the lumping error tends to zero.
This has implications for applications, for instance when considering MCMC diagnostics.
One could treat the indicators of small clades as approximately Markov and thus estimate clade probabilities and mixing times using well-known results for a two-state discrete-time Markov chain, much in the vein of~\cite{Raftery1992} who approach quantile estimation by assuming lumpability of a continuous Markov process onto $\{0,1\}$.

Our findings also open up the possibility of constructing auxiliary processes on smaller spaces in order to improve Monte Carlo estimation of relevant phylogenetic quantities.
However, Finding a partition that minimises lumping error while retaining some interpretability is an open problem.
Another important open question is how to link lumping error and speed of convergence: if we can bound the lumping error of a given process with a known convergence rate, what can one quantitatively say about the speed of convergence of the projected process?

Finally, the invariant distributions of the processes studied here may not be representative of real-world posterior distributions, so some work should be devoted to understanding how to transport some of the combinatorial results obtained here to more general processes and target distributions.
Overall, we believe the research presented here is a first step towards a more complete understanding of dimension-reduction in tree-valued Markov processes which we hope will be useful in practical applications.




\section*{Acknowledgements}
We would like to thank Mike Steel, Giulio Iacobelli and Rapha\"el Tinarrage for insightful discussions.
The authors acknowledge financial aid from the School of Applied Mathematics,  Getulio Vargas Foundation.

\bibliography{tree_lumpability}

\begin{thebibliography}{}

\bibitem[Aldous, 1996]{Aldous1996}
Aldous, D. (1996).
\newblock Probability distributions on cladograms.
\newblock In {\em Random discrete structures}, pages 1--18. Springer.

\bibitem[Alfaro and Holder, 2006]{Alfaro2006}
Alfaro, M.~E. and Holder, M.~T. (2006).
\newblock The posterior and the prior in {B}ayesian phylogenetics.
\newblock {\em Annu. Rev. Ecol. Evol. Syst.}, 37(1):19--42.

\bibitem[Billera et~al., 2001]{Billera2001}
Billera, L.~J., Holmes, S.~P., and Vogtmann, K. (2001).
\newblock Geometry of the space of phylogenetic trees.
\newblock {\em Advances in Applied Mathematics}, 27(4):733--767.

\bibitem[Bittracher and Sch{\"u}tte, 2021]{Bittracher2021}
Bittracher, A. and Sch{\"u}tte, C. (2021).
\newblock A probabilistic algorithm for aggregating vastly undersampled large {M}arkov chains.
\newblock {\em Physica D: Nonlinear Phenomena}, 416:132799.

\bibitem[Blum et~al., 2006]{Blum2006}
Blum, M., Francois, O., and Janson, S. (2006).
\newblock The mean, variance and limiting distribution of two statistics sensitive to phylogenetic tree balance.
\newblock {\em The Annals of Applied Probability}, 16(4):2195--2214.

\bibitem[Bordewich and Semple, 2005]{bordewich2005}
Bordewich, M. and Semple, C. (2005).
\newblock On the computational complexity of the rooted subtree prune and regraft distance.
\newblock {\em Annals of combinatorics}, 8:409--423.

\bibitem[Brusselmans et~al., 2024]{Brusselmans2024}
Brusselmans, M., Carvalho, L.~M., Hong, S.~L., Gao, J., Matsen~IV, F.~A., Rambaut, A., Lemey, P., Suchard, M.~A., Dudas, G., and Baele, G. (2024).
\newblock On the importance of assessing topological convergence in bayesian phylogenetic inference.
\newblock {\em arXiv preprint arXiv:2402.11657}.

\bibitem[Buchholz, 1994]{buchholz1994}
Buchholz, P. (1994).
\newblock Exact and ordinary lumpability in finite {M}arkov chains.
\newblock {\em Journal of applied probability}, 31(1):59--75.

\bibitem[Candido et~al., 2020]{Candido2020}
Candido, D.~S., Claro, I.~M., De~Jesus, J.~G., Souza, W.~M., Moreira, F.~R., Dellicour, S., Mellan, T.~A., Du~Plessis, L., Pereira, R.~H., Sales, F.~C., et~al. (2020).
\newblock Evolution and epidemic spread of sars-cov-2 in brazil.
\newblock {\em Science}, 369(6508):1255--1260.

\bibitem[Cowles and Carlin, 1996]{Cowles1996}
Cowles, M.~K. and Carlin, B.~P. (1996).
\newblock {M}arkov chain {M}onte {C}arlo convergence diagnostics: a comparative review.
\newblock {\em Journal of the American statistical Association}, pages 883--904.

\bibitem[Dayar and Stewart, 1997]{dayar1997}
Dayar, T. and Stewart, W.~J. (1997).
\newblock Quasi lumpability, lower-bounding coupling matrices, and nearly completely decomposable {M}arkov chains.
\newblock {\em SIAM Journal on Matrix Analysis and Applications}, 18(2):482--498.

\bibitem[Dinh and IV, 2017]{Dinh2017}
Dinh, V. and IV, F. A.~M. (2017).
\newblock {The shape of the one-dimensional phylogenetic likelihood function}.
\newblock {\em The Annals of Applied Probability}, 27(3):1646 -- 1677.

\bibitem[Disanto et~al., 2022]{diSanto2022}
Disanto, F., Fuchs, M., Paningbatan, A.~R., and Rosenberg, N.~A. (2022).
\newblock The distributions under two species-tree models of the number of root ancestral configurations for matching gene trees and species trees.
\newblock {\em The Annals of Applied Probability}, 32(6):4426--4458.

\bibitem[Drummond et~al., 2012]{Drummond2012}
Drummond, A.~J., Suchard, M.~A., Xie, D., and Rambaut, A. (2012).
\newblock {B}ayesian phylogenetics with {BEAU}ti and the {BEAST} 1.7.
\newblock {\em Molecular Biology and Evolution}, 29(8):1969--1973.

\bibitem[Dudas et~al., 2017]{Dudas2017}
Dudas, G., Carvalho, L.~M., Bedford, T., Tatem, A.~J., Baele, G., Faria, N.~R., Park, D.~J., Ladner, J.~T., Arias, A., Asogun, D., et~al. (2017).
\newblock Virus genomes reveal factors that spread and sustained the {E}bola epidemic.
\newblock {\em Nature}, 544(7650):309--315.

\bibitem[Franceschinis and Muntz, 1994]{franceschinis1994}
Franceschinis, G. and Muntz, R.~R. (1994).
\newblock Bounds for quasi-lumpable {M}arkov chains.
\newblock {\em Performance Evaluation}, 20(1-3):223--243.

\bibitem[Gavryushkin and Drummond, 2016]{Gavryushkin2016}
Gavryushkin, A. and Drummond, A.~J. (2016).
\newblock The space of ultrametric phylogenetic trees.
\newblock {\em Journal of Theoretical Biology}, 403:197--208.

\bibitem[Kelly et~al., 2022]{Kelly2022}
Kelly, L.~J., Ryder, R.~J., and Clart{\'e}, G. (2022).
\newblock Lagged couplings diagnose {M}arkov chain {M}onte {C}arlo phylogenetic inference.
\newblock {\em arXiv preprint arXiv:2108.13328}.

\bibitem[Lakner et~al., 2008]{Lakner2008}
Lakner, C., Van Der~Mark, P., Huelsenbeck, J.~P., Larget, B., and Ronquist, F. (2008).
\newblock Efficiency of markov chain monte carlo tree proposals in bayesian phylogenetics.
\newblock {\em Systematic {B}iology}, 57(1):86--103.

\bibitem[Lee et~al., 2003]{lee2003}
Lee, C. P.-C., Golub, G.~H., and Zenios, S.~A. (2003).
\newblock A fast two-stage algorithm for computing pagerank and its extensions.
\newblock Technical report, Citeseer.

\bibitem[Levin and Peres, 2017]{Levin2017}
Levin, D.~A. and Peres, Y. (2017).
\newblock {\em {M}arkov chains and mixing times}, volume 107.
\newblock American Mathematical Soc.

\bibitem[Lueg et~al., 2021]{Lueg2021}
Lueg, J., Garba, M.~K., Nye, T.~M., and Huckemann, S.~F. (2021).
\newblock Wald space for phylogenetic trees.
\newblock In {\em Geometric Science of Information: 5th International Conference, GSI 2021, Paris, France, July 21--23, 2021, Proceedings 5}, pages 710--717. Springer.

\bibitem[Magee et~al., 2023]{Magee2023}
Magee, A., Karcher, M., Matsen~IV, F.~A., and Minin, V.~M. (2023).
\newblock How trustworthy is your tree? {B}ayesian phylogenetic effective sample size through the lens of {M}onte {C}arlo error.
\newblock {\em {B}ayesian Analysis}, 1(1):1--29.

\bibitem[Mossel and Peres, 2003]{Mossel2003}
Mossel, E. and Peres, Y. (2003).
\newblock Information flow on trees.
\newblock {\em The Annals of Applied Probability}, 13(3):817--844.

\bibitem[Raftery et~al., 1992]{Raftery1992}
Raftery, A.~E., Lewis, S., et~al. (1992).
\newblock How many iterations in the gibbs sampler.
\newblock {\em Bayesian statistics}, 4(2):763--773.

\bibitem[Ronquist et~al., 2012]{Ronquist2012}
Ronquist, F., Teslenko, M., Van Der~Mark, P., Ayres, D.~L., Darling, A., H{\"o}hna, S., Larget, B., Liu, L., Suchard, M.~A., and Huelsenbeck, J.~P. (2012).
\newblock Mrbayes 3.2: efficient {B}ayesian phylogenetic inference and model choice across a large model space.
\newblock {\em Systematic biology}, 61(3):539--542.

\bibitem[Schr{\"o}der, 1870]{Schroder1870}
Schr{\"o}der, E. (1870).
\newblock Vier combinatorische probleme.
\newblock {\em Z. Math. Phys}, 15:361--376.

\bibitem[Semple and Steel, 2003]{Semple2003}
Semple, C. and Steel, M.~A. (2003).
\newblock {\em Phylogenetics}, volume~24.
\newblock Oxford University Press on Demand.

\bibitem[Simper and Palacios, 2022]{Simper2022}
Simper, M. and Palacios, J.~A. (2022).
\newblock An adjacent-swap {M}arkov chain on coalescent trees.
\newblock {\em Journal of Applied Probability}, 59(4):1243--1260.

\bibitem[Song, 2003]{Song2003}
Song, Y.~S. (2003).
\newblock On the combinatorics of rooted binary phylogenetic trees.
\newblock {\em Annals of Combinatorics}, 7(3):365--379.

\bibitem[Steel, 2014]{Steel2014}
Steel, M. (2014).
\newblock Tracing evolutionary links between species.
\newblock {\em The American Mathematical Monthly}, 121(9):771--792.

\bibitem[Sumita and Rieders, 1989]{Sumita1989}
Sumita, U. and Rieders, M. (1989).
\newblock Lumpability and time reversibility in the aggregation-disaggregation method for large {M}arkov chains.
\newblock {\em Stochastic Models}, 5(1):63--81.

\bibitem[Vehtari et~al., 2021]{Vehtari2021}
Vehtari, A., Gelman, A., Simpson, D., Carpenter, B., and B{\"u}rkner, P.-C. (2021).
\newblock Rank-normalization, folding, and localization: An improved $\hat{R}$ for assessing convergence of mcmc (with discussion).
\newblock {\em {B}ayesian analysis}, 16(2):667--718.

\bibitem[Warren et~al., 2017]{Warren2017}
Warren, D.~L., Geneva, A.~J., and Lanfear, R. (2017).
\newblock Rwty (r we there yet): an r package for examining convergence of {B}ayesian phylogenetic analyses.
\newblock {\em Molecular Biology and Evolution}, 34(4):1016--1020.

\bibitem[Whidden and Matsen~IV, 2017]{Whidden2017}
Whidden, C. and Matsen~IV, F.~A. (2017).
\newblock Ricci--{O}llivier curvature of the rooted phylogenetic subtree--prune--regraft graph.
\newblock {\em Theoretical Computer Science}, 699:1--20.

\bibitem[Zhu et~al., 2015]{Zhu2015}
Zhu, S., Than, C., and Wu, T. (2015).
\newblock Clades and clans: a comparison study of two evolutionary models.
\newblock {\em Journal of Mathematical Biology}, 71(1):99--124.

\end{thebibliography}

\appendix
\section{Proofs and technical results}
\label{sec:app_proofs}

\begin{lemma}\label{lem:preserve_d}
Let $x$ be a tree in   $\boldsymbol{T}_n$,  
then for any $y^\prime \in N(f_c(x))$, we have that $d_{rSPR}(g_{\phi_c^x}(y^\prime), x) =1$.
\end{lemma}

\begin{proof}
Let us denote $x^\prime = f_c(x)$. If $y' \in N(x')$, it follows that there exists a subtree within $x'$ that, when regrafted, results in $y'$. This subtree may or may not include the leaf $l$.

Let $g_{\phi_c^x}(y') = y$. The structure of the tree $y$ mirrors that of $y'$, with the significant difference being the replacement of the leaf $l$ by the subtree $\phi_c^x$. By definition, the subtree $\phi_c^x$ is inherently present in $x$.

Initially, we examine the scenario where the subtree regrafted from $x'$ to produce $y'$ does not contain the leaf $l$. Consequently, this identical subtree is present in $x$. Applying the same SPR operation to $x$ yields $y$. Therefore, we deduce that $d_{rSPR}(y, x) = 1$.

In the second case, we denote the subtree regrafted from $x'$ to yield $y'$ as $\phi_l$.

The procedure is as follows: we start with $x'$ and regraft $\phi_l$ such that we obtain $y'$. Consequently, applying $g_{\phi_c^x}$ to $y'$ results in $y$, where in the subtree $\phi_l$, the leaf $l$ is substituted by $\phi_c^x$. We then regraft this newly formed subtree into the same position as $\phi_l$, thereby reconstructing $x$. Thus, we conclude that $d_{rSPR}(y, x) = 1$, completing the proof.




\end{proof}

\begin{lemma}\label{lem:preserve_d2}
For all $\phi_1, \phi_2 \in \TT_{|c|}$, with $d_{r-SPR}(\phi_1, \phi_2) = k$ we obtain that $d_{rSPR}(h_{x^\prime}(\phi_1), h_{x^\prime}(\phi_2)) = k$, where $k \in \mathbb{Z}_{\ge 0}$. 
\end{lemma}

\begin{proof}
If $d_{rSPR}(\phi_1, \phi_2) = k$ then we have $d_{rSPR}(h_{x^\prime}(\phi_1), h_{x^\prime}(\phi_2)) \ge k$. 

Suppose that
$d_{rSPR}(h_{x^\prime}(\phi_1), h_{x^\prime}(\phi_2)) > k$. 

However, $d_{r-SPR}(h_{x^\prime}(\phi_1), h_{x^\prime}(\phi_2)) > k$ is a contradiction because the new nodes and edges added in $\phi_1$ and $\phi_2$ have the same configuration. Thus the only difference would be in the subtrees $\phi_1$ and $\phi_2$. 

Hence $d_{rSPR}(h_{x^\prime}(\phi_1), h_{x^\prime}(\phi_2)) = k$. 

\end{proof}

Next, we present the proof of Lemma~\ref{lem:card_A_1^xc}.

\begin{proof}[Proof of Lemma~\ref{lem:card_A_1^xc}]
Let us denote $x^\prime = f_c(x)$. First, we will prove that for all $y', z' \in N(x')$ such that $d_{rSPR}(y', z') > 0$, we have $d_{rSPR}(g_{\phi_c^x}(y'), g_{\phi_c^x}(z')) > 0$.

Suppose $y', z' \in N(x')$ such that $d_{rSPR}(y', z') > 0$ and $d_{rSPR}(y, z) = 0$, where $y = g_{\phi_c^x}(y')$ and $z = g_{\phi_c^x}(z')$. Then, we have $d_{rSPR}(f_c(y), f_c(z)) = 0$, and by the definition of $f_c$, we obtain $f_c(y) = y'$ and $f_c(z) = z'$, a contradiction because it would imply $d_{rSPR}(y', z') = 0$.

Clearly, for all $y^\prime, z^\prime \in N(x^\prime)$ such that $d_{r-SPR}(y^\prime, z^\prime) = 0$ then $d_{r-SPR}(g_{\phi_c^x}(y^\prime), g_{\phi_c^x}(z^\prime)) = 0$. 

Hence, we have $|g_{\phi_c^x}(N(x'))| = |N(x')|$, and by Lemma~\ref{lem:preserve_d}, $g_{\phi_c^x}(N(x')) \subset A_1^{x,c}$.


By Lemma~\ref{lem:preserve_d2}, for all $\phi_1, \phi_2 \in N(\phi_c^x)$ such that $d_{r-SPR}(\phi_1, \phi_2) = k$, it follows that $d_{r-SPR}(h_{x'}(\phi_1), h_{x'}(\phi_2)) = k$, where $k \in \mathbb{Z}_{\ge 0}$. Additionally, for every $\phi \in N(\phi_c^x)$, we observe that $d_{rSPR}(h_{x'}(\phi), h_{x'}(\phi_c^x)) = 1$, with $h_{x'}(\phi_c^x) = x$. Consequently, we deduce that $|h_{x'}(N(\phi_c^x))| = |N(\phi_c^x)|$ and $h_{x'}(N(\phi_c^x)) \subseteq A_1^{x,c}$.

Since $g_{\phi_c^x}(N(x')) \cap h_{x'}(N(\phi_c^x)) = \emptyset$, we obtain $|A_1^{x,c}| \geq |N(x')| + |N(\phi_c^x)|$.

Now, we show that $A_1^{x,c} \subset g_{\phi_c^x}(N(x') \cup h_{x'}(N(\phi_c^x)))$. Suppose there exists a $y \in A_1^{x,c}$ that is not included in $g_{\phi_c^x}(N(x')) \cup h_{x'}(N(\phi_c^x))$.

If $y \notin g_{\phi_c^x}(N(x'))$, then $y$ does not result from regrafting a node in edges outside of $\phi_c$. Additionally, if $y \notin h_{x'}(N(\phi_c^x))$, the regraft cannot involve an internal node of $\phi_c^x$ on an edge of $\phi_c^x$.

Since $y \in A_1^{x,c}$, there are two possible regrafting operations in $x$ to become $y$: either regrafting a subtree on edges inside of $\phi_c^x$ or a subtree of $\phi_c^x$ on edges outside of $\phi_c^x$. In both cases, $c$ would not be included in $y$, leading to a contradiction. Therefore, we conclude that $A_1^{x,c} \subset g_{\phi_c^x} (N(x^\prime)) \cup h_{x^\prime}(N(\phi_c^x))$ and we finish the proof. 
\end{proof}

Now we prove Lemma~\ref{lem:B_1^xc}.

\begin{proof}[Proof of Lemma~\ref{lem:B_1^xc}]

First, we will prove the case that $|B_1^{x,c}| = 0$. Suppose that we have the clades $c_1, c_2,\dots, c_k$, for a $k \in \{3, 4, \dots, n-3\}$ such that $c_i \cap c_j = \emptyset$, for all $i,j \in \{1,2,\dots, k\}$ and $\cup_{j=1}^k c_j = c$. 
 
Given that $|c| \leq n-3$, for every possible cardinality of $c$, we can construct a tree in which each element $c_j$, for $j \in {1,2,\dots,k}$, possesses a sibling that may either be a leaf (occurring when $|c| = n-k$) or a subtree. Consider such a tree $x$, which, by construction, belongs to $S_0(c)$. For $x$, achieving a tree in $S_1(c)$ via a single SPR operation is not feasible. For instance, regrafting $c_1$ onto the parent edge of $c_2$ results in a scenario where the least common ancestor of $c_1, c_2, \dots, c_k$ remains unchanged and is the same as that of the siblings of $c_3, c_4, \dots, c_k$. Consequently, the resulting tree from this SPR operation lacks the clade $c$.


From the proof presented above, there exists a unique manner in which a tree $x \in S_0(c)$ can be connected to any tree $y \in S_1(c)$, specifically when $d_{rSPR}(x, y) = 1$. This situation occurs if, and only if, in $x$, there are two distinct clades $c_1$ and $c_2$, satisfying $c_1 \cap c_2 = \emptyset$ and $c_1 \cup c_2 = c$.

If $c_1$ and $c_2$ have other subtrees that share the same least common ancestor, then to form the clade $c$, we must regraft $c_1$ either onto an edge of $c_2$ or onto the edge connecting $c_2$ with its parent. This operation can be performed reciprocally. Consequently, there are $2|c_2| - 1$ possible edges for regrafting $c_1$, yielding $2|c_1| - 1$ distinct neighbors. By applying the same reasoning with $c_1$ and $c_2$ interchanged, we ascertain that $|B_1^{x,c}| = 2(|c_1| + |c_2|) - 2 = 2(|c| - 1)$.

Now, consider a tree $x \in S_0(c)$, where between $c_1$ and $c_2$, there exists a subtree $\phi_I$. We can perform the same regrafting operation described above or regraft the subtree $\phi_I$ onto edges that are not beneath the least common ancestor of $c_1$, $c_2$, and $\phi_I$. It is important to note that regrafting $\phi_I$ onto the edge above the parent of the clade ${c_1 \cup c_2 \cup \phi_I}$, or regrafting $c_1$ above the parent of $c_2$ (or vice versa), results in the same tree. This observation is crucial to avoid double-counting the same tree as a neighbour. Considering that any tree in $\TT_n$ possesses $2n-1$ edges and the clade ${c_1 \cup c_2 \cup \phi_I}$ has $2(|c_1| + |c_2| + |I| - 1)$ edges, we deduce that $|B_1^{x,c}| = 2n - 1 - 2(|c| + |I| - 1) + 2(|c| - 1) - 2 = 2(n - |I|) - 3$. Thus, we conclude our proof.
\end{proof}

Next, we obtain an interesting combinatorial result, which shows some regularity in the neighbourhoods in the SPR graph. 

\begin{lemma}
Consider $(X_k)_{k \ge 0}$ a Metropolis-Hastings random walk and $(Y_k)_{k \ge 0}$ the projected chain on $\bar{F}:=\{F_1, F_2, \dots, F_v\}$ the tree shape partition with $\Tilde{P}$ its induced transition probability matrix.Then for any $F_i, F_j \in F$, $t^i \in F_i$ and $t^j \in F_j$, such that $N(t^i) \cap F_j \neq \emptyset$ we have
\begin{equation*}
    \frac{|F_i|}{|F_j|} = \frac{|N(t^j) \cap F_i|}{|N(t^i) \cap F_j|} \,.
\end{equation*}
\end{lemma}

\begin{proof}
For all $t, \hat{t} \in \TT_n$ by the detailed balanced equation we have
\begin{equation}\label{eq: bal_eq_1}
\pi(t) p(t, \hat{t}) = \frac{1}{|\TT_n|} p(t, \hat{t}) = \frac{1}{|\TT_n|} p(\hat{t}, t) = \pi(\hat{t}) p(\hat{t}, t) \,.
\end{equation}

Now for any $F_i , F_j \in F$, $t^i \in F_i$ and $t^j \in F_j$, we have the following
\begin{equation}\label{eq:prob_lump}
\Tilde{p}(F_i, F_j) = \sum_{z \in N(t^i) \cap F_j} p(t^i, z) = |N(t^i) \cap F_j|p(t^i, t^j) \,.
\end{equation}
In the first equality in~\eqref{eq:prob_lump} we used Theorem~\ref{teo:shape_lump_MH} and the second equality by Lemma~\ref{lem:F_j^xn=F_j^yn}. 

Thus 
\begin{equation}\label{eq: pi_y}
\begin{split}
& \pi_Y (F_i) \Tilde{p}(F_i, F_j) = \frac{|F_i|}{|\TT_n|} \times |N(t^i) \cap F_j|p(t^i, t^j) \quad \text{and}
\\
& \pi_Y(F_k) \Tilde{p}(F_j, F_i) = \frac{|F_j|}{|\TT_n|} \times |N(t^j) \cap F_i|p(t^j, t^i) \,.
\end{split}    
\end{equation}

Hence, we have the desired result by the detailed balanced equation,~\eqref{eq: bal_eq_1} and~\eqref{eq: pi_y}. 


\end{proof}

\begin{lemma}\label{lem:aux1}
Let $n$ and $|c|$ be positive integers such that $n \ge 9$ and $3 \le |c| \le \lfloor n^{1/2} \rfloor$. Then we have
\begin{equation*}
   \frac{-8|c|^2 + 18|c|n - 24|c| -18n + 28}{6(3n^2 -2|c|^2 +2|c|n-15n+16)} \ge \frac{36n -116}{6(3n^2 - 9n -2)} \,. 
\end{equation*}
\end{lemma}
\begin{proof}
Let us first denote 
\[
f(|c|,n) := \frac{-8|c|^2 + 18|c|n - 24|c| -18n + 28}{6(3n^2 -2|c|^2 +2|c|n-15n+16)}\,.
\]
Thus, we aim to prove that $f(|c|, n) \ge f(3,n)$ for every $n \ge 9$ and $3 \le |c| \le \lfloor n^{1/2} \rfloor$, where $n, |c| \in \mathbb{N}$. To complete our proof, it suffices to show that $f(|c|,n)$ is non-decreasing with respect to $|c|$ and for every $n \ge 9$. 

Note that we can rewrite $f(|c|, n)$ in the following way
\begin{equation*}
  f(|c|, n) = \frac{2|c|(9n - 12) -8|c|^2 + 28}{6(2|c|n - 2|c|^2 + (3n^2 - 15n +16))} \,.   
\end{equation*}

Then we have
\begin{equation}\label{eq:ap_lem1}
\begin{split}
& f(|c|+1, n) = \frac{2(|c|+1) - 8(|c|+1)^2 + 28}{6(2(|c|+1)n - 2(|c|+1)^2 + (3n^2 - 15n +16))}
\\
& = \frac{2|c|(9n - 12) -8|c|^2 + 28 + 18n - 16|c| - 16}{6(2|c|n - 2|c|^2 + (3n^2 - 15n +16) + 2n - 4|c| -2)} \ge f(|c|, n)  \,.
\end{split}
\end{equation}
It is important to note that $18n - 16|c| - 16 \ge 2n - 4|c| - 2 > 0$ for every $n \ge 9$ and $3 \le |c| \le \lfloor n^{1/2} \rfloor$. Given that $a/b \le (a+c)/(b+d)$ when $b > a > 0$ and $c > d > 0$, we obtain the final inequality in~\eqref{eq:ap_lem1}, thereby completing the proof.   

\end{proof}


\end{document}